\documentclass[reqno,a4paper,11pt]{amsart}

\usepackage[english]{babel}

\usepackage{amsfonts}
\usepackage{amsmath}
\usepackage{amsthm}
\usepackage{amssymb}
\usepackage{indentfirst}
\usepackage{xcolor}
\usepackage{tabularx}
\usepackage{graphicx}
\usepackage{esint}
\usepackage{dsfont}

\allowdisplaybreaks

\usepackage{soul}

\usepackage[a4paper,top=3cm,bottom=3cm,left=2.5cm,right=2.5cm]{geometry}

\usepackage{subcaption}

\usepackage{mathrsfs}
\usepackage{mathtools}
\usepackage{enumitem}
\usepackage[normalem]{ulem}

\usepackage[utf8]{inputenc}
\usepackage[T1]{fontenc}
\usepackage[english]{babel}

\numberwithin{equation}{section}

\theoremstyle{plain}
\begingroup
\theoremstyle{plain}
\newtheorem{theorem}{Theorem}[section]
\newtheorem{corollary}[theorem]{Corollary}
\newtheorem{proposition}[theorem]{Proposition}
\newtheorem{lemma}[theorem]{Lemma}
\theoremstyle{definition}
\newtheorem{definition}[theorem]{Definition}

\theoremstyle{remark}
\newtheorem{remark}[theorem]{Remark}
\endgroup

\theoremstyle{definition}
\theoremstyle{remark}

\mathchardef\emptyset="001F

\newcommand{\R}{\mathbb{R}}
\newcommand{\E}{\mathcal{E}}

\newcommand{\B}{\mathrm{B}}
\newcommand{\dive}{\mathrm{div}}

\newcommand{\spt}{\mathrm{spt}}

\newcommand{\T}{\mathcal{T}}
\newcommand{\di}{\mathrm{d}}

\newcommand{\F}{\mathcal{F}}

\newcommand{\sP}{\mathcal{P}}
\newcommand{\Lip}{\mathrm{Lip}}

\newcommand{\spann}{\mathrm{span}}

\newcommand{\y}{\boldsymbol{y}}
\newcommand{\nnu}{\boldsymbol{\nu}}
\newcommand{\mmu}{\boldsymbol{\mu}}
\newcommand{\de}{\mathrm{d}}
\newcommand{\z}{\boldsymbol{z}}

\newcommand{\bx}{\boldsymbol{x}}
\newcommand{\blambda}{\boldsymbol{\lambda}}

\newcommand{\sfd}{\mathsf{d}}
\newcommand{\cF}{\mathcal{F}}
\newcommand{\cL}{\mathcal{L}}
\newcommand{\cP}{\mathcal{P}}

\newcommand{\cT}{\mathcal{T}}

\newcommand{\BL}{\mathrm{BL}}
\newcommand{\uu}{\boldsymbol{u}}
\newcommand{\ev}{\mathrm{ev}}

\definecolor{dred}{rgb}{.8,0,0}
\definecolor{ddmagenta}{rgb}{0.7,0,0.9}
\definecolor{ddcyan}{rgb}{0,0.2,1.0}
\definecolor{Orchid}{rgb}{0.7,0.4,0}
\definecolor{blue_links}{RGB}{13,0,180} 
\definecolor{lightblue}{RGB}{0.9,0.9,1}

\newcommand{\eps}{\varepsilon}

\newcommand{\res}{\mathop{\hbox{\vrule height 7pt width .5pt depth
0pt\vrule height .5pt width 6pt depth0pt}}\nolimits}

\usepackage{hyperref}
\hypersetup{
    colorlinks=true, 
    linktoc=all,     
    linkcolor=blue_links,  
    citecolor=blue_links,
    urlcolor=blue_links,
}

\newcommand{\N}{\mathbb{N}}

\newcommand{\coloneq }{\hspace{1pt}\raisebox{0.74pt}{\scalebox{0.8}{:}}\hspace{-2.2pt}=}

\newcommand{\eeta}{\boldsymbol{\eta}}
\newcommand{\Pp}{\mathcal{P}}

\makeatletter
\@namedef{subjclassname@2020}{%
  \textup{2020} Mathematics Subject Classification}
\makeatother

\title[Mean-field selective optimal control via transient leadership]{Mean-field selective optimal control\\ via transient leadership}

\author[G. Albi]{Giacomo Albi}
\address[Giacomo Albi]{Dipartimento di Informatica, 
Universit\`a di Verona, Strada Le Grazie 14, Ca Vignal 2, 37134 Verona, Italy.} 
\email{giacomo.albi@univr.it}

\author[S. Almi]{Stefano Almi}
\address[Stefano Almi]{Faculty of Mathematics, University of Vienna, 
Oskar-Morgenstern-Platz 1, 1090 Wien, Austria.}
\email{stefano.almi@univie.ac.at}

\author[M. Morandotti]{Marco Morandotti}
\address[Marco Morandotti]{Dipartimento di Scienze Matematiche ``G.~L.~Lagrange'',
Politecnico di Torino, Corso Duca degli Abruzzi 24,
10129 Torino, Italy.}
\email{marco.morandotti@polito.it}

\author[F. Solombrino]{Francesco Solombrino}
\address[Francesco Solombrino]{Dipartimento di Matematica e Applicazioni ``R.~Caccioppoli'',
Universit\`a di Napoli Federico II, via Cintia, 80126 Napoli, Italy.}
\email{francesco.solombrino@unina.it}

\date{\today}


\begin{document}

\begin{abstract}
A mean-field selective optimal control problem of multipopulation dynamics via transient leadership is considered. 
The agents in the system are described by their spatial position and their probability of belonging to a certain population. 
The dynamics in the control problem is characterized by the presence of an activation function which tunes the control on each agent according to the membership to a population, which, in turn, evolves according to a Markov-type jump process.
This way, a hypothetical policy maker can select a restricted pool of agents to act upon based, for instance, on their time-dependent influence on the rest of the population. 
A finite-particle control problem is studied and its mean-field limit is identified via $\Gamma$-convergence, ensuring convergence of optimal controls. The dynamics of the mean-field optimal control is governed by a continuity-type equation without diffusion.
Specific applications in the context of opinion dynamics are discussed with some numerical experiments.

\end{abstract}

\maketitle

\tableofcontents


\section{Introduction}
Multi-population agent systems have drawn much attention in the last decades as a tool to describe the evolution of groups of individuals with some features that can change with time. These models find their application in contexts as varied as evolutionary population dynamics~\cite{AmbForMorSav18,MR1635735,TJ1978}, economics \cite{W1995}, chemical reaction networks~\cite{LLN2019,CRN,O1989}, and kinetic models of opinion formation \cite{BMPW2009,Toscani2006}.
In these models, each agent carries a label that may describe, for instance, membership to a population (\emph{e.g.}, leaders or followers), or the strategy used in a game. While this label space is often discrete, for many applications (and also as a necessary condition for the existence   of Nash equilibria~\cite{Nash1951}) it is useful to attach to each agent located at a point~$x\in\R^d$ a continuous variable which describes their \emph{mixed strategies} or, referring back to the context of leaders and followers, their degree of influence. If $U$ denotes the space of labels, this may be encoded by a probability measure $\lambda\in\sP(U)$.
It is natural to postulate that $\lambda$ can vary with time according to a spatially inhomogeneous Markov-type jump process with a transition rate~$\T(x,\lambda,(\bx,\blambda))$ that may depend on the position $x$ of the agent and on the global state of the system $(\bx,\blambda)$, containing the positions and the labels of all the agents. Leadership may indeed be temporary and affected, for instance, by circumstances, need, location, and mutual distance among the agents.

Mean-field descriptions of such systems allow one for an efficient treatment by replacing the many-agent paradigm with a kinetic one \cite{CCR2011,CCH2014}, consisting of a limit PDE whose unknown is the distribution of agents with their label, as those obtained in \cite{ABRS,AlmMorSol21,MorSol19,thai} (see also \cite{LT2020} for a related Boltzmann-type approach).

A further step which we devise in this paper is the extension of the mean-field point of view to the problem of controlling such systems, possibly in a selective way.
The underlying idea is the presence of a policy maker whose control action, at any instant of time, concentrates on a subset of the population chosen according to the level of influence of the agents. 

More precisely, in a population of $N$ agents, the time-dependent state of the $i$-th agent is given by $t\mapsto y_i(t)=(x_i(t),\lambda_i(t))$, where $x_i\in \R^d$ and $\lambda_i\in\sP(U)$, for every $i=1,\ldots,N$, and evolves according to the controlled ODE system
\begin{equation}\label{1.1}
\begin{cases}
\dot{x}_{i} = v (x_i,\lambda_i,(\bx, \blambda)) + h(x_i,\lambda_i,(\bx,\blambda))u_i\\
\dot{\lambda}_i = \cT(x_i,\lambda_i,(\bx,\blambda))
\end{cases}
\end{equation}
where $v$ is a velocity field, $u_i$ is the control on the $i$-th agent, and $h\geq0$ is a non-negative activation function selecting the set of agents targeted by the decision of the policy maker, depending on their state and, possibly, on the global state of the system.
The values $u_i$ are determined by minimization of the cost functional 
\begin{equation}
\label{e:cost-N-particles-intro}
\E_{N} ( \y, \uu)\coloneqq
\displaystyle  \frac{1}{N} \sum_{i=1}^{N} \fint_{0}^{T} \cL_{N}( y_{i}(t) , \Psi^{N}_{t} ) \, \di t + \frac{1}{N}  \sum_{i=1}^{N} \fint_{0}^{T} \phi(u_{i}(t)) \, \di t 
\end{equation}
where $\Psi^N_t$ is the empirical measure defined as $\Psi^N_t \coloneqq \frac{1}{N}\sum_{i=1}^{N} \delta_{y_i(t)}$ and $\phi$ is a positive convex cost function, superlinear at infinity, and such that $\phi(0)=0$; finally, the Lagrangian $\cL_N(\cdot)$ is continuous and symmetric (see Definition~\ref{d:sym} and Remark~\ref{r:sym} below).
 
In this paper we show that the variational limit, in the sense of $\Gamma$-convergence \cite{Braides2002,DM1993} in a suitable topology, of the functional introduced in \eqref{e:cost-N-particles-intro} is given by 
\begin{equation}\label{e:limit-cost-intro}
\E (\Psi, w) \coloneq \fint_{0}^{T} \int_{\R^d\times\sP(U)} \cL(y , \Psi_{t}) \, \di \Psi_{t}(y) \, \di t + \fint_{0}^{T} \int_{{\R^d\times\sP(U)}} \phi(w(t, y) ) \, \di \Psi_{t}(y)\, \di t\,,
\end{equation}
where $\cL$ is a certain limit Lagrangian cost and where $\Psi_t\in\sP(\R^d\times\sP(U))$ and $w$ are coupled by the mean-field continuity equation
\begin{equation}\label{e:conteq-2-intro}
\partial_{t} \Psi_{t} + \dive ( b (t, \cdot) \Psi_{t}) = 0\quad\text{for}\quad b (t, y) \coloneq \left( \begin{array}{cc}
v(y,\Psi_t)+ h(y,\Psi_{t}) w (t, y) \\
\T(y,\Psi_t)
\end{array} \right),
\end{equation}
with the request that $w$ be integrable with respect to the measure $h\Psi_t\otimes\di t$. From the point of view of the applications, we remark that our main result Theorem~\ref{t:Gamma-convergence} implies that a minimizing pair $(\Psi,w)$ for the optimal control problem \eqref{e:limit-cost-intro} can be obtained as the limit of minimizers~$(\y^N,\uu^N)$ of the finite-particle optimal control problem \eqref{e:cost-N-particles-intro} (a precise statement is given in Corollary~\ref{c:minima}).

In this sense, our result extends to the multi-population setting the results of \cite{Flos,FS2014} with the relevant feature that the activation function $h$ allows the policy maker to tune the control action on a subset of the entire population which is \emph{not prescribed} a priori, but rather depends on the evolution of the system.
At fixed time $t>0$, it can target its intervention on the most influential elements of the population according to a threshold encoded by $h$. This is similar, in spirit, to a principle of sparse control, as considered, \emph{e.g.}, in \cite{albi2018selective,caponigro2017mean,fornasier2014mean}. 
Again, our model includes additional features; in particular, a control action only through leaders is already present in \cite{fornasier2014mean}, where however the leaders' population is fixed a priori and discrete.
A localized control action on a small time-varying subset~$\omega(t)$ of the state space of the system is presented in \cite{caponigro2017mean} as an infinite-dimensional generalization of \cite{JQ1978}; there, no optimal control is considered and the evolution of~$\omega(t)$ is algorithmically constructed to reach a desired target, instead of being determined by the evolution itself.
The numerical approach of \cite{albi2018selective} makes use of a selective state-dependent control specifically designed for the Cucker--Smale model. For other recent examples of localized/sparse intervention in mean-field control systems, we refer the reader to \cite{albi2016invisible,albi2,burger2020instantaneous,CLOS2020,kalise2020sparse,PTZ2020,TZ2019}.

\smallskip

The role of the variable $\lambda$ deserves some attention. It can be generally intended as a measure of the influence of an agent, accounting for a number of different interpretations according to the context. Similar background variables have been used in recent literature to describe wealth distribution \cite{MR4147945,MR2551376,pareschi2014wealth}, degree of knowledge \cite{MR3544858,MR3579566}, degree of connectivity of an agent in a network \cite{albi2017opinion,burger2021network}, and also applications to opinion formation \cite{MR3420842}, just to name a few.
Comparing to these other approaches, our mean-field approximation \eqref{e:limit-cost-intro}, \eqref{e:conteq-2-intro} features a more profound interplay between the variable $\lambda$ and the spatial distribution $x$ of the agents, resulting in a higher flexibility of the model: not only is $\lambda$ changing in time, but its variation is driven by an optimality principle steered by the controls.

We present some applications in Section~\ref{s:appl} in the context of opinion dynamics, where $\lambda$ represents the transient degree of leadership of the agents. 
Specifically, in the former example we highlight the emergence of leaders and how this can be exploited by a policy maker; in the latter, two competing populations of leaders with different targets and campaigning styles are considered, and the effect of the control action in favoring one of them is analyzed.

\smallskip

The plan of the paper is the following: in Section~\ref{s:math} we introduce the functional setting of the problem and we list the standing assumptions on the velocity field $v$, on the transition operator~$\T$, and on the cost functions $\cL_N$, $\cL$, and $\phi$; in Section~\ref{s:finite-particle-problem} we present and discuss the existence of solutions to the finite-particle control problem; in Section~\ref{s:Gamma-limit} we introduce the mean-field control problem and prove the main theorem on the $\Gamma$-convergence to the continuous problem. In Section~\ref{s:appl} we discuss the applications mentioned above.

\subsection{Technical aspects}
We highlight the main technical aspects of the proof of Theorem~\ref{t:Gamma-convergence}. The $\Gamma$-liminf inequality builds upon a compactness property of sequences of empirical measures~$\Psi^{N}_{t}$ with uniformly bounded cost~$\E_{N}$. The hypotheses on the velocity field~$v$ and on the transition operator~$\T$ in~\eqref{1.1} (see Section~\ref{s:math}) imply, by a Gr\"onwall-type argument, a uniform-in-time estimate of the support of~$\Psi^{N}_{t}$, ensuring the convergence to a limit~$\Psi_{t}$. The lower bound and the identification of the control field $w$ are consequences of the convergence of~$\mathcal{L}_{N}$ to~$\mathcal{L}$ and of the convexity and superlinear growth of the cost function~$\phi$. 
As for the $\Gamma$-limsup inequality, we remark that the sole integrability of~$w$ (contrary to the situation considered in \cite{FS2014}) does not guarantee the existence of a flow map for the associated Cauchy problem
\begin{equation}\label{e:some-intro}
\begin{cases}
\dot{x}_{i} = v(y_{i}, \y) + h(y_{i}, \y) w(t, y_{i}) \,,\\
\dot{\lambda}_{i} = \T(y_{i}, \y) \,,\\
y_{i}(0) = y_{0,i}\,,
\end{cases}
\end{equation}
and therefore does not allow for a direct construction of a recovery sequence based on the analysis of~\eqref{e:some-intro} due to the lack of continuity with respect to the data. Following the main ideas of~\cite{Flos}, we base our approximation strategy on the superposition principle~\cite[Theorem~5.2]{AmbForMorSav18}, \cite[Theorem~3.11]{MorSol19} (see also \cite{AmbGigSav08,AmbTre14,CCDMP2021,Smirnov93}), which indeed selects a sequence of trajectories $\z^{N}$ such that the corresponding empirical measures~$\Lambda^{N}_{t} \coloneqq \frac{1}{N} \sum_{i=1}^{N} \delta_{z^{N}_{i}(t)}$ converge to~$\Psi_{t}$ and the cost $\E_N(\z^N,\uu^N)$ converges to $\E(\Psi,w)$, where we have set~$u_{i}^{N}(t) \coloneqq w(t, z^{N}_{i}(t))$. The explicit dependence of~\eqref{e:some-intro} on the global state of the system calls for a further modification of the trajectories~$\z^{N}$. Here, the fact that~$h$ may take the value~$0$ introduces an additional technical difficulty as we cannot exploit the linear dependence on the controls in~\eqref{1.1}.
To overcome this problem, we resort once again to the local Lipschitz continuity of~$v$ and of~$\T$, and construct the trajectories $\y^{N}$ by solving the Cauchy problem
\begin{equation}\label{e:some-intro-2}
\begin{cases}
\dot{x}_{i} = v(y_{i}, \y) + h(y_{i}, \y) u_{i} \,,\\
\dot{\lambda}_{i} = \T(y_{i}, \y) \,,\\
y_{i}(0) = y_{0,i} \,. 
\end{cases}
\end{equation}
By Gr\"onwall estimates, we can conclude that the distance between $\Lambda^{N}_{t}$ and the empirical measure~$\Psi^{N}_{t}$ generated by~$\y^{N}$ is infinitesimal, so that we obtain the desired convergences of $\Psi^{N}_{t}$ to~$\Psi_{t}$ and of $\E_{N}(\y^{N}, \uu^{N})$ to $\E(\Psi, w)$. Let us also mention that the symmetry of the cost is used in a crucial way to deal with the initial conditions in~\eqref{e:some-intro-2}.


\section{Mathematical setting}\label{s:math}

In this section we introduce the mathematical framework and notation to study our system.
\smallskip

\paragraph{\textbf{Basic notation.}} 
Given a metric space~$(X,\sfd_X)$, we denote by~$\mathcal{M}(X)$ the space of signed Borel measures~$\mu$ in~$X$ with finite total variation~$\| \mu \|_{\mathrm{TV}}$,
by~$\mathcal{M}_+(X)$ and~$\cP(X)$ the convex subsets of nonnegative measures and probability measures, respectively. 
We say that~$\mu \in \cP_c(X)$ if $\mu \in \Pp(X)$ and the support~$\spt \, \mu$ is a compact subset of~$X$. 
For any~$K \subseteq X$, the symbol~$\Pp(K)$ denotes the set of measures~$\mu \in \Pp(X)$ such that~$\spt\, \mu \subseteq K$.
Moreover, $\mathcal{M}(X;\R^d)$ denoted the space of $\R^d$-valued Borel measures with finite total variation.

As usual, if~$(Z, \sfd_{Z})$ is another metric space, for every~$\mu\in\mathcal{M}_+(X)$ and every~$\mu$-measurable function~$f\colon X\to Z$, we define the push-forward measure~$f_\#\mu\in\mathcal{M}_+(Z)$ by $(f_\#\mu)(B)\coloneqq \mu(f^{-1}(B))$, for any Borel set~$B\subseteq Z$.



For a Lipschitz function~$f\colon X\to\mathbb{R}$ we define its Lipschitz constant by
\begin{equation*}
\mathrm{Lip}(f)\coloneqq\sup_{x,y\in X \atop x\neq y}\frac{|f(x)-f(y)|}{\sfd_X(x,y)}
\end{equation*}
and we denote by~$\mathrm{Lip}(X)$ and~$\mathrm{Lip}_b(X)$  the spaces of Lipschitz and bounded Lipschitz functions on~$X$, respectively. Both are normed spaces with the norm~$\lVert f\rVert_{\mathrm{Lip}} \coloneqq \lVert f\rVert_\infty+ \mathrm{Lip}(f)$, where~$\lVert\cdot\rVert_\infty$ is the supremum norm. Furthermore, we use the notation~$\mathrm{Lip}_{1}(X)$ for the set of functions $f \in \mathrm{Lip}_{b} (X)$ such that~$\mathrm{Lip}(f) \leq 1$.

In a complete and separable metric space~$(X,\sfd_X)$, we shall use the Wasserstein distance~$W_1$ in the set~$\cP(X)$, defined as
\begin{equation*}
W_1(\mu,\nu)\coloneqq\sup\bigg\{\int_X\varphi\,\mathrm{d}\mu-\int_X\varphi\,\mathrm{d}\nu: \varphi\in\mathrm{Lip}_1(X) \bigg\} \,.
\end{equation*}
Notice that~$W_1(\mu,\nu)$ is finite if~$\mu$ and~$\nu$ belong to the space
\begin{equation*}
\cP_1(X)\coloneqq \bigg\{\mu\in\cP(X):  \text{$\int_X \sfd_X(x,\bar x)\,\mathrm{d}\mu(x)<+\infty$ for some $\bar x\in X$}\bigg\}
\end{equation*}
and that~$(\cP_1(X),W_1)$ is a complete metric space if~$(X,\sfd_X)$ is complete.

If~$(E, \| \cdot\|_{E})$ is a Banach space and~$\mu \in\mathcal{M}_+(E)$, we define the first moment~$m_1(\mu)$ as
\begin{equation*}
m_1(\mu)\coloneqq\int_{E} \lVert x \rVert_E\,\mathrm{d}\mu\,.
\end{equation*}
Notice that for a probability measure~$\mu$ finiteness of the integral above is equivalent to $\mu \in \cP_1(E)$, whenever~$E$ is endowed with the distance induced by the norm~$\lVert\cdot  \rVert_{E}$.
Furthermore, the notation~$C^1_b(E)$ will be used to denote the subspace of~$C_b(E)$ of functions having bounded continuous Fr\'echet differential at each point. The symbol~$\nabla$ will be used to denote the Fr\'echet differential. In the case of a function~$\phi \colon [0,T]\times E \to \mathbb{R}$, the symbol~$\partial_t$ will be used to denote partial differentiation with respect to~$t$. 
\smallskip

\paragraph{\textbf{Functional setting.}} We consider a set of pure strategies~$U$, where~$U$ is a compact metric space, and we denote by~$Y\coloneqq \R^{d} \times \Pp(U)$ the state-space of the system. 

According to the functional setting considered in~\cite{AmbForMorSav18,MorSol19}, we consider the space $\overline{Y}\coloneqq \R^{d} \times \F(U)$, where we have set (see, e.g.,~\cite{AP,MR81458} and \cite[Chapter~3]{MR3792558})
\begin{equation}\label{P1}
\F(U) \coloneqq \overline{ \spann ( \Pp (U) ) }^{\|\cdot\|_{\mathrm{BL}}} \subseteq (\mathrm{Lip}(U))'.
\end{equation}
The closure in~\eqref{P1} is taken with respect to the \emph{bounded Lipschitz} norm $\lVert\cdot\rVert_{\BL}$, defined as
\begin{equation*}
\lVert \mu \rVert_{\BL}\coloneqq\sup \big\{\langle \mu,\varphi\rangle:  \varphi\in \Lip(U),  \|\varphi\|_{\Lip}\leq 1\big\} \qquad \text{for every $\mu\in(\Lip(U))'$}\,.
\end{equation*}
We notice that, by definition of~$\| \cdot\|_{\mathrm{BL}}$, we always have
\begin{displaymath}
\| \mu \|_{\mathrm{BL}} \leq \| \mu \|_{\mathrm{TV}} \qquad \text{for every $\mu \in \mathcal{M}(U)$}\,,
\end{displaymath}
in particular, $\| \lambda \|_{\mathrm{BL}} \leq 1$ for every $\lambda \in \Pp(U)$.

Finally, we endow $\overline{Y}$ with the norm $\lVert y\rVert_{\overline Y}=\lVert(x,\lambda)\rVert_{\overline Y}\coloneqq \lvert x \rvert+\lVert \lambda \rVert_{\BL}$ .

For every $R>0$, we denote 
by~$\B_R^Y$ the closed ball of radius~$R$ in~$Y$, namely $\B_R^Y=\{y\in Y:\lVert y\rVert_{\overline Y}\leq R\}$ and notice that, in our setting,~$\B^{Y}_{R}$ is a compact set. 

As in~\cite{MorSol19}, we consider, for every~$\Psi \in \Pp_{1} (Y)$, a velocity field $v_{\Psi} \colon Y \to \R^{d}$ such that
\begin{itemize}
\item[($v_1$)] for every $R>0$, $v_{\Psi} \in \mathrm{Lip} (\B^{Y}_{R}; \R^{d})$ uniformly with respect to~$\Psi \in \Pp( \B^{Y}_{R})$, i.e., there exists~$L_{v, R}>0$ such that
\begin{displaymath}
| v_{\Psi} (y_{1}) - v_{\Psi}(y_{2}) | \leq L_{v, R} \| y_{1} - y_{2} \|_{\overline{Y}} \qquad \text{for every $y_{1}, y_{2} \in Y$}\,;
\end{displaymath}

\item[($v_2$)] for every $R>0$ there exists $L_{v, R}>0$ such that for every $\Psi_{1}, \Psi_{2} \in \Pp(\B^{Y}_{R})$ and every $y \in \B^{Y}_{R}$
\begin{displaymath}
| v_{\Psi_{1}} (y) - v_{\Psi_{2}} (y) | \leq L_{v, R} W_{1} (\Psi_{1}, \Psi_{2})\,;
\end{displaymath}

\item[($v_3$)] there exists $M_{v}>0$ such that for every $y \in Y$ and every $\Psi \in \Pp_{1}(Y)$
\begin{displaymath}
| v_{\Psi} (y) | \leq M_{v} \big( 1 + \| y \|_{\overline{Y}} + m_{1} ( \Psi) \big) \,.
\end{displaymath}
\end{itemize}

As for~$\T$, for every $\Psi \in  \Pp_{1} (Y)$ we assume that the operator $\T_{\Psi} \colon Y \to \F(U)$ is such that
\begin{itemize}

\item[($\T_0$)] for every $(y , \Psi) \in Y \times \Pp_{1}( Y )$, the constants belong to the kernel of~$\T_{\Psi}(y)$, i.e., 
\begin{displaymath}
\left\langle \T_{\Psi}(y), 1 \right\rangle_{\F(U), \mathrm{Lip}(U)} = 0 \,,
\end{displaymath}
where $\langle\cdot,\cdot\rangle$ denoted the duality product;

\item[($\T_1$)] there exists $M_{\T}>0$ such that for every $y \in Y$ and every $\Psi \in \Pp_{1}(Y)$
\begin{displaymath}
\| \T_{\Psi}(y) \|_{\mathrm{BL}} \leq M_{\T} \big( 1 + \| y \|_{\overline{Y}} + m_{1} ( \Psi) \big)\,;
\end{displaymath}

\item[($\T_2$)] for every $R>0$, there exists $L_{\T, R}>0$ such that for every $(y_{1}, \Psi_{1}), (y_{2}, \Psi_{2}) \in \B^{Y}_{R} \times \Pp (\B^{Y}_{R})$
\begin{displaymath}
\| \T_{\Psi_{1}} ( y_{1} ) - \T_{\Psi_{2}} (y_{2} ) \|_{{\rm BL}} \leq L_{\T, R} \big( \| y_{1} - y_{2} \|_{\overline{Y}} + W_{1} (\Psi_{1}, \Psi_{2}) \big)\,;
\end{displaymath}

\item[($\T_3$)] for every $R>0$ there exists $\delta_{R}>0$ such that for every $(y, \Psi) \in \B^{Y}_{R} \times \Pp_{1}(Y)$ we have
\begin{displaymath}
\T_{\Psi} (y) + \delta_{R} \lambda \geq 0\,.
\end{displaymath}
\end{itemize}

For every $y  \in Y$ and every $\Psi \in \Pp_{1}(Y)$ we set $
b_{\Psi}(y)  \coloneqq 
\left(\begin{array}{cc}
\displaystyle v_{\Psi}(y) \\ [1mm]
\displaystyle \T_{\Psi} (y)
\end{array}\right)
$,
which is the velocity field driving the evolution; we also consider an
activation function $h_{\Psi} \colon \overline{Y} \to [0,+\infty)$ satisfying:
\begin{itemize}
\item[($h_{1}$)] $h_{\Psi}$ is bounded in~$\overline{Y}$ uniformly with respect to~$\Psi \in \sP_{1}(Y)$;

\item [($h_{2}$)] for every $R>0$ there exists $L_{h, R}>0$ such that for every~$\Psi_{1}, \Psi_{2} \in \sP_{1}(\B^{Y}_{R})$ and every $y_{1}, y_{2} \in \B^{Y}_{R}$
\begin{displaymath}
| h_{\Psi_{1}} (y_{1}) - h_{\Psi_{2}} (y_{2}) | \leq L_{h, R} \big( \| y_{1} - y_{2} \|_{\overline{Y}} + W_{1}( \Psi_{1}, \Psi_{2}) \big)\,.
\end{displaymath}
\end{itemize}

In order to define the optimal control problems of Sections~\ref{s:finite-particle-problem} and~\ref{s:Gamma-limit}, we have to introduce some further notation. For every $N \in \mathbb{N}$, we define
\begin{displaymath}
\sP^{N}(Y) \coloneq  \bigg\{ \Psi \in \sP(Y):\ \text{there exist $y_{1}, \ldots, y_{N} \in Y$ such that $\Psi = \frac{1}{N} \sum_{i=1}^{N}\delta_{y_{i}}$} \bigg\}\,.
\end{displaymath}
In particular, we notice that, up to a permutation, every $N$-tuple $\y^{N}\coloneq (y_{1}, \ldots, y_{N}) \in Y^{N}$ can be identified with an element $\Psi \in \sP^{N}(Y)$. We now give the following two definitions (see also~\cite[Definition~2.1]{Flos}.

\begin{definition}\label{d:sym}
For every $N \in \mathbb{N}$, we say that a map $F_{N} \colon Y \times Y^{N} \to [0,+\infty)$ is symmetric if $F_{N}(y , \y) = F_{N}( y, \sigma(\y))$ for every $y \in Y$, every $\y \in Y^N$, and every permutation $\sigma \colon Y^{N} \to Y^{N}$.
\end{definition}

\begin{remark}\label{r:sym}
Notice that by symmetry and by the identifying $\y^{N}\coloneq (y_{1}, \ldots, y_{N}) \in Y^N$ with $\Psi^{N} = \frac{1}{N} \sum_{i=1}^{N} \delta_{y_{i}}$ we may write $F_{N}(y, \Psi^{N})$ for $F_{N}(y, \y^{N})$.
\end{remark}

\begin{definition}
Let $F_{N} \colon Y \times Y^{N} \to [0,+\infty)$ be symmetric. We say that $F_{N}$ $\sP_{1}$-converges to~$F \colon Y \times \sP_1(Y) \to [0,+\infty)$ uniformly on compact sets as $N\to \infty$ if for every subsequence~$N_{k}$ and every sequence $\Psi_{k} \in \sP^{N_{k}} (Y)$ narrowly converging to~$\Psi$ in $\sP_1(Y)$ we have
\begin{displaymath}
\lim_{k \to \infty} \, \sup_{y \in K}\, | F_{N_{k}}(y, \Psi_{k}) - F(y, \Psi) | = 0 \qquad \text{for every compact subset $K$ of~$Y$}.
\end{displaymath}
\end{definition}

For the cost functionals for the finite particle control problem and for their mean-field limit we consider the functions $\phi \colon \R^{d} \to [0, + \infty)$, $\cL_{N} \colon Y \times Y^{N} \to [0,+\infty)$, and $\cL\colon Y \times \sP_{1}(Y) \to [0,+\infty)$ such that
\begin{itemize}
\item[($\phi_{1}$)]\label{phi_{1}} $\phi$ is convex and superlinear with $\phi(0) = 0$;

\item[($\cL_{1}$)]\label{L_{1}} $\cL_{N}$ is continuous and symmetric;

\item[($\cL_{2}$)] $\cL_{N}$ $\sP_{1}$-converges to~$\cL$ uniformly on compact sets;

\item[($\cL_{3}$)] for every $R>0$, $\mathcal{L}$ is continuous on $\B^{Y}_{R} \times \sP(\B^{Y}_{R})$.
\end{itemize}


\section{The finite particle control problem}
\label{s:finite-particle-problem}

We now introduce the finite particle control problem. We fix a compact and convex subset~$K$ of~$\R^{d}$ of admissible controls with $0 \in K$. For every $N \in \mathbb{N}$ and every control function $u_{i} \in L^{1}([0,T]; K)$, $i=1, \ldots, N$, the dynamics of the $N$-particles system is driven by the Cauchy problem
\begin{equation}
\label{e:Cauchy}
\left\{ \begin{array}{ll}
\dot{y}_{i} (t) = b_{\Psi^{N}_{t}} (y_{i}(t)) + \left( \begin{array}{cc}
h_{\Psi^{N}_{t}} (y_{i}(t)) u_{i}(t) \\
0
\end{array}\right) & \text{for $i=1, \ldots, N$}\,,\\[2mm]
y_{i}(0) = y_{0,i} \in Y\,,
\end{array} \right.
\end{equation}
where we have set $\Psi^{N}_{t} \coloneq \frac{1}{N} \sum_{i=1}^{N} \delta_{y_{i}(t)} \in \sP^N (Y)$. For simplicity of notation, we set $\uu^{N}(t) \coloneq ( u_{1}(t), \ldots, u_{N}(t)) \in K^{N}$ for every $t \in [0,T]$. In view of~\cite[Section~I.3, Theorem~1.4, Corollary~1.1]{MR0348562} (see also~\cite[Theorem~B.1]{AmbForMorSav18} and \cite[Corollary~2.3]{MorSol19}), the Cauchy problem~\eqref{e:Cauchy} admits a unique solution $\y^{N} \coloneq ( y_{1}, \ldots, y_{N}) \in AC([0,T] ;Y^{N})$, which is also identified with the empirical measure~$\Psi^{N}_{t}$, up to a permutation. To ease the notation in our analysis, we give the following definition.
\begin{definition}
\label{d:generate}
We say that $(\y^{N}, \uu^{N}) \in AC([0,T]; Y^{N}) \times L^{1}([0,T]; K^{N})$ generates the pairs $(\Psi^{N}, \nnu^{N}) \in AC([0,T]; ( \sP^{N}(Y) ; W_{1})) \times \mathcal{M}([0,T]\times \overline{Y}; \R^{d})$ if $\Psi^{N} = \Psi^{N}_{t} \otimes \mathcal{L}^{1}\res[0,T]$ with $\Psi^{N}_{t} = \frac{1}{N} \sum_{i=1}^{N} \delta_{y_{i}(t)}$ and $\nnu^{N} = \nnu^{N}_{t} \otimes \mathcal{L}^{1} \res[0,T]$ with
\begin{displaymath}
\nnu^{N}_{t} \coloneq  \frac{1}{N} \sum_{i=1}^{N} h_{\Psi^{N}_{t}} (\cdot) u_{i}(t) \delta_{y_{i} (t) }( \cdot )\,,
\end{displaymath} 
where~$\mathcal{L}^{1}\res[0,T]$ denotes the Lebesgue measure on~$\R$ restricted to the interval~$[0,T]$.

In a similar way, if~$\y^{N}_{0}= (y_{0,1}, \ldots, y_{0,N}) \in Y^{N}$, we say that $\y_{0}^{N}$ generates~$\Psi^{N}_{0} \in \sP^{N}(Y)$ if $\Psi^{N}_{0} = \frac{1}{N} \sum_{i=1}^{N} \delta_{y_{0,i}}$.
\end{definition}

Given $\y^{N}_{0} = (y_{0,1}, \ldots, y_{0,N}) \in Y^{N}$, we define the set of couples trajectory-control solving the Cauchy problem~\eqref{e:Cauchy} as 
\begin{equation}
\label{e:S1}
\mathcal{S}(\y^{N}_{0}) \coloneq \big\{ (\y, \uu) \in AC([0,T]; Y^{N}) \times L^{1}([0,T]; K^{N}): \, (\y, \uu) \text{ solves~\eqref{e:Cauchy}}\big\}\,.
\end{equation}

Given functions $\phi$, $\cL_{N}$, and~$\cL$ satisfying conditions~$(\phi_{1})$,~$(\cL_{1})$, and~$(\cL_{2})$, for every initial condition~$\y^{N}_{0} \in Y^{N}$ and every $(\y, \uu) \in AC([0,T]; Y^{N}) \times L^{1} ( [ 0,T] ; K^{N} ) $,  we define the cost functional 
\begin{equation}
\label{e:cost-N-particles}
\!\!\!\! \E_{N}^{\y^{N}_{0}} ( \y, \uu) \coloneq \left\{ \begin{array}{ll}
\displaystyle  \frac{1}{N} \sum_{i=1}^{N} \fint_{0}^{T} \cL_{N}( y_{i}(t) , \Psi^{N}_{t} ) \, \di t + \frac{1}{N}  \sum_{i=1}^{N} \fint_{0}^{T} \phi(u_{i}(t)) \, \di t & \text{if $(\y, \uu) \in \mathcal{S}(\y^{N}_{0})$} \\[2mm]
\displaystyle \vphantom{\int}  +\infty & \text{otherwise},
\end{array}\right.
\end{equation}
where $(\Psi^{N}, \nnu^{N})$ is the pair generated by~$(\y, \uu)$. Therefore, the optimal control problem for the $N$-particle system reads as follows:
\begin{equation}
\label{e:min-N-particles}
\min \Big \{ \E_{N}^{\y^{N}_{0}}  ( \y, \uu) : \ \text{$(\y, \uu) \in AC([0,T]; Y^{N}) \times L^{1}([0,T]; K^{N})$} \Big\}\,.
\end{equation}

We now prove the existence of solutions of the minimum problem~\eqref{e:min-N-particles}. First, we state the boundedness of the trajectories~$\y$ for given control and initial datum, which will also be useful in the $\Gamma$-convergence analysis of Section~\ref{s:Gamma-limit}.

\begin{proposition}
\label{p:bound-N-particles}
For every $N \in \mathbb{N}$, every initial datum $\y^{N}_{0}= (y_{0,1}, \ldots , y_{0,N}) \in Y^{N}$, and every $(\y^{N}, \uu^{N}) \in \mathcal{S}(\y^{N}_{0})$ we have
\begin{equation}
\label{e:bound-N-particles}
\sup_{i=1, \ldots, N}\, \| y_{i}  \|_{L^{\infty}([0,T]; \overline Y)} \leq C\, \sup_{i=1, \ldots, N}\, \| y_{0,i}\|_{\overline Y}
\end{equation}
for a positive constant~$C$ independent of~$N$.
\end{proposition}

\begin{proof}
Let $(\Psi^{N}, \nnu^{N})$ be the pair generated by~$(\y^{N},  \uu^{N})$. Since the control~$\uu^{N}$ takes values in~$K^{N}$ with~$K$ compact in~$\R^{d}$ and in view of the assumptions~$(v_{1})$, $(\T_{1})$, and~$(h_{1})$, for every $t \in [0, T]$ we estimate
\begin{align*}
\| y_{i} (t) \|_{\overline Y} & \leq \| y_{0,i} \|_{\overline Y} + \int_{0}^{t} \| b_{\Psi^{N}_{\tau}} (y_{i} (\tau)) \|_{\overline Y} \, \di t + \int_{0}^{t} | h_{\Psi^{N}_{\tau}}(y_{i} (\tau)) u_{i} (\tau) | \, \di \tau
\\
&
\leq \| y_{0,i} \|_{\overline Y} + \int_{0}^{t} (M_{v} + M_{\T}) ( 1 + \| y_{i}(\tau)\|_{\overline Y} + m_{1}(\Psi^{N}_{\tau}) ) \, \di \tau +  \overline{C}t
\\
&
\leq \| y_{0,i} \|_{\overline Y} + \int_{0}^{t} (M_{v} + M_{\T}) \Big( 1 + \| y_{i} (\tau)\|_{\overline Y} + \sup_{j=1, \ldots, N}\, \| y_{j} (\tau) \|_{\overline Y} \Big) \, \di \tau + \overline{C}t\,,
\end{align*}
for some positive constant~$\overline{C}$ depending only on~$h$ and~$K$. Taking the supremum over $i\in \{1, \ldots, N\}$ in the previous inequality and applying Gr\"onwall inequality we deduce~\eqref{e:bound-N-particles}. 
\end{proof}

\begin{proposition}
\label{p:N-particles}
For every $N \in \mathbb{N}$ and every initial datum $\y^{N}_{0} \in Y^{N}$, the minimum problem~\eqref{e:min-N-particles} admits a solution  $(\y^{N}, \uu^{N})$. If $(\Psi^{N}, \nnu^{N})$ is the pair generated by~$(\y^{N}, \uu^{N})$, then also the pair $(\y^{N}, \tilde{\uu}^{N})$ where
\begin{equation}\label{e:redef}
\tilde{u}_{i}(t) = \left\{ \begin{array}{ll}
u_{i}(t) & \text{if $h_{\Psi^{N}_{t}} (y_{i}(t)) \neq 0$}\,,\\
0 & otherwise\,,
\end{array}\right.\qquad \text{$i=1, \ldots, N$}
\end{equation}
is a solution of~\eqref{e:min-N-particles}. 
If the cost function~$\phi$ satisfies $\{ \phi=0\} = \{0\}$, then every solution~$(\y^{N}, \uu^{N})$ of~\eqref{e:min-N-particles} satisfies $u_{i}(t)=0$ a.e.~on~$\{t\in [0,T]: \, h_{\Psi^{N}_{t}} (y_{i}(t)) = 0\}$ for $i= 1, \ldots, N$.
\end{proposition}

\begin{proof}
Let us fix~$N \in \mathbb{N}$ and let~$\uu_{k}^{N}= (u_{k,1}, \ldots, u_{k,N}) \in L^{1}([0,T]; K^{N})$ and $\y_{k}^{N}= (y_{k,1}, \ldots, y_{k,N}) \in AC([0,T]; Y^{N})$ be a minimizing sequence for the cost functional~$\E_{N}^{\y^{N}_{0}}$. In particular, we may assume $(\y_{k}^{N}, \uu_{k}^{N}) \in \mathcal{S}(\y^{N}_{0})$ for every $k$. Let us further denote~$(\Psi^{N}_{k},\nnu^{N}_{k}) \in AC([0,T]; (\sP^{N}(Y); W_{1})) \times \mathcal{M}([0,T]\times  \overline{Y}; \R^{d})$ the pair generated by~$(\y_{k}^{N}, \uu_{k}^{N})$

Since $\uu_{k}^{N}$ takes values in~$K^{N}$ and~$K$ is compact in~$\R^{d}$, up to a subsequence we have that $\uu_{k}^{N} \rightharpoonup \uu^{N}$ weakly$^{*}$ in~$L^{\infty}([0,T]; K^{N})$. By Proposition~\ref{p:bound-N-particles}, $\y_{k}^{N}$ is bounded in~$C([0,T]; Y^{N})$. Let us fix $R>0$ such that $\| \y_{k}^{N}(t) \|_{Y^{N}} \leq R$ for $t \in [0,T]$ and $k \in \mathbb{N}$. Then, by~$(v_{1})$,~$(\T_{1})$, and~$(h_{1})$, for every $s<t \in [0,T]$, every $i=1, \ldots, N$, and every~$k$ we have that
\begin{align*}
\| y_{k,i} (t) - y_{k,i}(s) \|_{\overline Y} & \leq \int_{s}^{t} \| b_{\Psi^{N}_{k,\tau}} (y_{k,i}( \tau) ) \|_{\overline Y} \, \di \tau + \int_{s}^{t} | h_{\Psi^{N}_{k,\tau}} (y_{k,i}(\tau) ) u_{k,i}(\tau) | \, \di \tau 
\\
&
\leq \int_{s}^{t} (M_{v} + M_{\T}) ( 1 + \| y_{k,i}(\tau) \|_{\overline Y} + m_{1}(\Psi^{N}_{k,\tau}) ) \, \di \tau + C | t-s| 
\\
&
\leq 2( M_{v} + M_{\T}) (1 + R) | t - s| + C| t-s|\,. 
\end{align*}
Thus, $\y_{k}^{N}$ is bounded and equi-Lipschitz continuous in~$[0,T]$. By Ascoli-Arzel\`a Theorem, $\y_{k}^{N}$ converges uniformly to some~$\y^{N} \in C([0,T]; Y^{N})$ along a suitable subsequence, and $\y^{N}(0) = \y_{0}^{N}$. Furthermore, if $(\Psi^{N}, \nnu^{N})$ is the pair generated by~$(\y^{N}, \uu^{N})$, we also deduce that $\Psi^{N}_{k} \to \Psi^{N}$ in $C([0,T]; (\sP^{N}(Y); W_{1}))$. In view of~$(v_{2})$,~$(\T_{2})$, and~$(h_{2})$, it is easy to see that $(\y^{N}, \uu^{N}) \in \mathcal{S}(\y^{N}_{0})$.

Finally, the continuity of~$\cL_{N}$ and the convexity of~$\phi$ yield the lower semicontinuity of the cost functional~$\E_{N}^{\y_{0}^{N}}$, so that
\begin{displaymath}
\E_{N}^{\y^{N}_{0}} (\y^{N}, \uu^{N}) \leq \liminf_{k \to \infty} \, \E_{N}^{\y^{N}_{0}} ( \y_{k}^{N}, \uu_{k}^{N})
\end{displaymath} 
and $(\y^{N}, \uu^{N})$ is a solution of~\eqref{e:min-N-particles}.

The second part of the statement follows from the structure of system~\eqref{e:Cauchy}. Indeed, if we define~$\tilde{\uu}^{N}$ as in~\eqref{e:redef}, the trajectory~$\y^{N}$ solution of~\eqref{e:Cauchy} does not change and $(\y^{N},\tilde{\uu}^{N})\in\mathcal{S}(\y_{0}^{N})$. Since the cost function~$\phi$ is non-negative with $\phi(0) = 0$, it is easy to see that $\E_{N}^{\y^{N}_{0}} (\y^{N}, \tilde{\uu}^{N}) \leq \E_{N}^{\y^{N}_{0}} (\y^{N}, \uu^{N})$. Finally, if~$\{\phi=0\} = \{0\}$, the previous inequality and the minimality of~$(\y^{N}, \uu^{N})$ imply that $u_{i}(t) = \tilde{u}_{i}(t)$ for $t \in [0,T]$ and $i=1, \ldots, N$, and the proof is concluded.
\end{proof}


\section{Mean-field control problem}
\label{s:Gamma-limit}

Before introducing the mean-field optimal control problem and stating the main $\Gamma$-convergence result, we discuss the compactness of sequences of pairs trajectory-control $(\y^{N}, \uu^{N})$ with bounded energy~$\E^{\y^{N}_{0}}_{N}$. To ease the notation, given a curve $\Psi \in C([0,T]; (\sP_{1}(Y); W_{1}))$ we denote by $h_{\Psi} \Psi$ the curve $t \mapsto h_{\Psi_{t}} \Psi_{t}$. Similarly to~\eqref{e:S1} we define, for every $\widehat{\Psi}_{0} \in \sP_{c}(Y)$, the set
\begin{align}
\label{e:S2}
\mathcal{S}( \widehat{\Psi}_{0}) \coloneq \bigg\{ ( \Psi, \nnu) \in & \  AC([0,T]; (\sP_{1}(Y); W_{1})) \times   \mathcal{M}([0,T]\times  \overline{Y}; \R^{d}):
\\
&
\ \nnu \ll h_{\Psi} \Psi,\, \frac{\di \nnu}{\di (h_{\Psi}\Psi) } \in K \text{ $h_{\Psi}\Psi$- a.e., $(\Psi, \nnu)$ solves} \nonumber
\\
&
 \partial_{t} \Psi_{t} + \dive ( b_{\Psi_{t}} \Psi_{t} + \overline\nnu_{t}) = 0  \text{ with $\Psi_{0} = \widehat{\Psi}_{0}$ and $\overline \nnu_{t} = (\nnu_{t}, 0)$}\bigg\}\,. \nonumber 
\end{align}

\begin{proposition}
\label{p:mean-field-dynamics}
For $N \in \mathbb{N}$, let $\y^{N}_{0} = (y_{0,1}, \ldots, y_{0,N}) \in Y^{N}$ and $(\y^{N}, \uu^{N}) \in \mathcal{S}(\y^{N}_{0})$ with corresponding generated measures ~$\Psi^{N}_{0} \in \sP^{N}(Y)$ and $(\Psi^{N}, \nnu^{N} ) \in AC([0,T]; (\sP^{N}(Y); W_{1}))\times \mathcal{M}([0,T]\times  \overline{Y} ; \R^{d})$. Assume that $\Psi^{N}_{0} \to \widehat{\Psi}_{0} \in \sP_{c}(Y)$ in the 1-Wasserstein distance and that
\begin{equation}
\label{e:energy-bound}
\sup_{N \in \mathbb{N}} \, \E_{N}^{\y^{N}_{0}} ( \y^{N}, \uu^{N}) < +\infty\,.
\end{equation} 
Then, up to a subsequence, the curve~$\Psi^{N}$ converges uniformly in~$C([0,T]; (\sP_{1}(Y); W_{1}))$ to~$\Psi \in AC([0,T]; (\sP_{1}(Y); W_{1}))$,~$\nnu^{N}$ converges weakly$^*$ to $\nnu \in \mathcal{M}([0,T]\times  \overline{Y} ; \R^{d})$, and  $(\Psi, \nnu) \in \mathcal{S}(\widehat{\Psi}_{0})$.
\end{proposition}

\begin{remark}
Since~$\nnu \ll h_{\Psi}\Psi$ for $(\Psi, \nnu) \in \mathcal{S}(\widehat{\Psi}_{0})$, there exists a function~$v \in L^{1}_{h_{\Psi}\Psi}([0,T] \times Y;\R^d)$ such that $\nnu =  v h_{\Psi}\Psi$. Furthermore, if we consider $\overline{v}(t, y) \coloneq v(t, y) \mathbf{1}_{\{ h_{\Psi} \neq0\}} (t, y)$, we still have $\nnu =  \overline{v} h_{\Psi}\Psi$. 
\end{remark}

In view of the compactness result in Proposition~\ref{p:mean-field-dynamics}, for $\Psi \in C([0,T]; (\sP_{1}(Y); W_{1}))$ and $\nnu \in \mathcal{M}([0,T]\times  \overline{Y} ; \R^{d})$ we define the cost functional for the mean-field control problem as
\begin{equation}
\label{e:limit-cost}
\E^{\widehat{\Psi}_{0}} (\Psi, \nnu) \coloneq \left\{ \begin{array}{ll}
\displaystyle \fint_{0}^{T} \int_{Y} \cL(y , \Psi_{t}) \, \di \Psi_{t}(y) \, \di t + \Phi_{\rm min}( \Psi, \nnu) & \text{if $(\Psi, \nnu) \in \mathcal{S}(\widehat{\Psi}_{0})$}\,,\\[2mm]
\displaystyle \vphantom{\int} +\infty & \text{otherwise}\,,
\end{array}\right.
\end{equation}
where we have set for $(\Lambda, \mmu) \in \mathcal{S}(\widehat{\Psi}_{0})$
\begin{align}
\label{e:limit-phi1}
&\Phi_{\rm min} (\Lambda, \mmu) \coloneq \min \{ \Phi(w, \Lambda): \, w \in L^{1}_{h_{\Lambda} \Lambda} ([0,T]\times Y; K), \, \mmu = w h_{\Lambda} \Lambda\} \,,\\
& \Phi(w,\Lambda) \coloneq \fint_{0}^{T} \int_{Y} \phi(w(t, y) ) \, \di \Lambda_{t}(y) \, \di t \qquad \text{for $w \in L^{1}_{h_{\Lambda} \Lambda}([0,T]\times Y; K)$}\,. \label{e:limit-phi2}
\end{align}

With the above notation at hand, the mean-field optimal control reads as
\begin{align}
\label{e:min-mean-field}
\min\,\Big\{ \E^{\widehat{\Psi}_0}(\Psi, \nnu) :  (\Psi, \nnu) \in C([0,T];(\sP_1(Y);W_1)) \times \mathcal{M}([0,T]\times  \overline{Y};\R^d) \Big\}\,.
\end{align}
In order to discuss the existence of solutions to~\eqref{e:min-mean-field}, we introduce the auxiliary functionals
\begin{align}\label{e:phi-bar}
& \overline{\Phi}(\Lambda, \mmu ) \coloneq \left\{ \begin{array}{ll}
\displaystyle \fint_{0}^{T} \overline{\phi} (  \Lambda_{t}, \mmu_{t} ) \, \di t & \text{if $\mmu \ll \Lambda$}\,,\\[2mm]
\displaystyle \vphantom{\int} + \infty & \text{otherwise}\,,
\end{array}\right. \qquad  \overline{\phi} ( \Lambda_{t},  \mmu_{t}) \coloneq \int_{\overline{Y}} \phi \bigg( \frac{\di \mmu_{t}}{\di \Lambda_{t}} (y) \bigg) \, \di \Lambda_{t}(y) 
\end{align}
for every $(  \Lambda, \mmu) \in AC([0,T]; (\sP_{1}(Y); W_{1})) \times  \mathcal{M}([0,T]\times  \overline{Y} ; \R^{d}) $. In the next two propositions we show the existence of solutions to~\eqref{e:min-mean-field}. We start by proving that for each~$(\Psi, \nnu) \in \mathcal{S}(\widehat{\Psi}_{0})$, the support of~$\Psi_{t}$ is bounded in~$Y$ uniformly.

\begin{proposition}
\label{p:bound-support}
Let $\widehat{\Psi}_{0} \in \sP_{c}(Y)$. Then, there exist $R>0$ and $L>0$ such that for every $(\Psi, \nnu) \in \mathcal{S}(\widehat{\Psi}_{0})$ the curve $t \mapsto \Psi_{t}$ is $L$-Lipschitz continuous and satisfies $\spt(\Psi_{t}) \subseteq \B^{Y}_{R}$.
\end{proposition}
\begin{proof}
Let $(\Psi, \nnu)$ be as in the statement of the proposition. In particular, we may write $\nnu = w h_{\Psi} \Psi$ for $w \in L^{1}_{ h_{\Psi} \Psi} ([0, T]; K)$ such that
\begin{displaymath}
\Phi_{\rm min} (\Psi, \nnu) = \fint_{0}^{T} \int_{Y} \phi(w (t, y) ) \, \di \Psi_{t}(y) \, \di t\,.
\end{displaymath} 
Since $\phi(0) = 0$ and $\phi \geq0$, without loss of generality we may suppose $w (t, y) = 0$ in $\{ (t, y) \in [0, T] \times Y: \,  h_{\Psi} (t, y) = 0\}$.

Let us first give a bound on the first moment~$m_{1}(\Psi_{t})$. To do this, we fix a function $\zeta \in C_{c}(\mathcal{F}(U))$ such that $0 \leq \zeta\leq 1$ and $\zeta(\lambda) = 1 $ for $\lambda \in \sP(U)$, which is possible since $\sP(U)$ is a compact subset of~$\mathcal{F}(U)$. For every $n \in \mathbb{N}$ and every $\varepsilon>0$, let us fix $g_{\varepsilon}(x) \coloneqq \sqrt{|x|^{2} + \varepsilon^{2}}$ and $\theta_{n}(x) \coloneqq \theta(\frac{x}{n})$, where $\theta \in C_{c} (\R^{d})$ is such that $0 \leq \theta \leq 1$, $|\nabla_{x} \theta| \leq 1$ in~$\R^{d}$, $\theta (x) = 1 $ for $|x| \leq 1$, and $\theta(x) = 0$ for $|x| \geq2$. Then, the function $\zeta g_{\varepsilon} \theta_{n} \in C_{c}(\overline{Y})$ and
\begin{equation}\label{e:first-moment}
\begin{split}
\int_{\overline{Y}} & \ g_{\varepsilon}(x) \theta_{n}(x) \, \di \Psi_{t} (y)   - \int_{\overline{Y}} g_{\varepsilon}(x) \theta_{n}(x) \, \di \widehat{\Psi}_{0} (y) 
\\
&
= \ \int_{0}^{t} \int_{\overline{Y}} \nabla_{x} (g_{\varepsilon}(x) \theta_{n}(x)) \cdot b_{\Psi_{\tau}} (y) \, \di \Psi_{\tau}(y) \, \di \tau 
\\
&
\quad + \int_{0}^{t} \int_{\overline{Y}} h_{\Psi_{\tau}}(y) \nabla_{x} (g_{\varepsilon}(x) \theta_{n}(x)) \cdot w (\tau, y) \, \di \Psi_{\tau}(y) \, \di \tau\,.
\end{split}
\end{equation}
Since $|\nabla_{x} \theta_{n}| \leq \frac{1}{n}$,  $g_{\varepsilon}(x) \leq |x| + \varepsilon$, and $(h_{1})$--$(h_{2})$ hold, we continue in~\eqref{e:first-moment} with
\begin{equation}\label{e:first-moment-2}
\begin{split}
\int_{\overline{Y}}& \ g_{\varepsilon}(x) \theta_{n}(x) \, \di \Psi_{t} (y)  - \int_{\overline{Y}} g_{\varepsilon}(x) \theta_{n}(x) \, \di \widehat{\Psi}_{0} (y) 
\\
&
\leq \Big(  2 + \frac{\varepsilon}{n} \Big) \int_{0}^{t} \int_{\overline{Y}} \| b_{\Psi_{\tau}} (y) \|_{\overline{Y}} \, \di \Psi_{\tau}(y) \, \di t + C T \Big(  2 + \frac{\varepsilon}{n} \Big) \,,
\end{split}
\end{equation}
for a positive constant~$C$ dependent only on~$h$ and on~$K$. Passing to the limit, in the order, as $\varepsilon \to 0$ and $n \to \infty$, and using~$(v_{1})$ and~$(\T_{1})$, we deduce from~\eqref{e:first-moment-2} that
\begin{equation}\label{e:first-moment-3}
m_{1}(\Psi_{t}) \leq m_{1}(\widehat{\Psi}_{0}) +  4( M_{v} + M_{\T}) \int_{0}^{t} (1 + m_{1}(\Psi_{\tau})) \,\di \tau + CT\,.
\end{equation}
Since $\Psi_{t} \in \sP(Y)$ for every $t \in [0, T]$, applying Gr\"onwall inequality to~\eqref{e:first-moment-3} we infer that 
\begin{equation}
\label{e:moment-bdd-12}
\sup_{t \in [0, T]} \,  m_{1}(\Psi_{t}) \leq \big(  m_{1}(\widehat{\Psi}_{0}) + \overline{C} T \big) e^{4T (M_{v} + M_{\T})}  \,,
\end{equation}
for some positive constant~$\overline{C}$ only depending on~$h$ and on~$K$.

We now prove the uniform bound of the support of~$\Psi_{t}$. To do this, we will apply the superposition principle~\cite[Theorem~5.2]{AmbForMorSav18}. The curve~$\Psi \in AC([0,T]; (\sP_{1}(Y); W_{1}))$ solves the continuity equation
\begin{align}
\label{e:conteq-2}
\partial_{t} \Psi_{t} + \dive ( b (t, \cdot) \Psi_{t}) = 0 \qquad \text{with $\Psi_{0}= \widehat{\Psi}_{0}$}\,,
\end{align}
where the velocity field~$b \colon [0,T]\times \overline{Y} \to \overline{Y}$ is defined as 
\begin{align}\label{e:flux-2}
b (t, y) \coloneq b_{\Psi_{t}}(y) + \left( \begin{array}{cc}
h_{\Psi_{t}} (y) w (t, y) \\
0
\end{array} \right) \qquad \text{for $y \in Y$}
\end{align} 
and is extended to $0$ in~$\overline{Y} \setminus Y$. By~\eqref{e:moment-bdd-12},~$(v_{1})$,~$(\T_{1})$, and~$(h_{1})$, and by the fact that $w (t, y) \in K$, we can estimate
\begin{equation}
\label{e:limsup1-2}
\begin{split}
\int_{0}^{T} \int_{\overline{Y}} \,\| b (t, y)\|_{\overline{Y}} \, \di \Psi_{t}(y) \, \di t  \leq & \ (M_{v} + M_{\T}) \int_{0}^{T} \int_{\overline{Y}} (1 + \| y \|_{\overline{Y}} + m_{1}(\Psi_{t}) ) \, \di \Psi_{t}(y) \, \di t 
\\
&
 + \int_{0}^{T} \int_{\overline{Y}} |h_{\Psi_{t}} (y)w (t, y)  |\, \di \Psi_{t}(y) \, \di t < +\infty\,. 
\end{split}
\end{equation}
We are therefore in a position to apply~\cite[Theorem~5.2]{AmbForMorSav18} with velocity field~$b$. Hence, there exists $\pi \in \sP(C([0, T]; \overline{Y}))$ such that 
\begin{equation}\label{e:superpos}
\Psi_{t} = (\mathrm{ev}_{t})_{\#} \pi \qquad \text{for every $t \in [0, T]$},
\end{equation}
where $\mathrm{ev}_{t} (y) \coloneqq y(t)$ for every $y \in  C([0, T]; \overline{Y})$ and every $t \in [0, T]$. Moreover, $\pi$ is concentrated on solutions of the Cauchy problems
\begin{equation}\label{e:another-Cauchy}
\left\{
\begin{array}{ll}
\dot{y}(t) = b (t, y(t)) \,,\\
y(0) = y_{0} \in \spt( \widehat{\Psi}_{0}) \,.
\end{array}
\right.
\end{equation}
For every $y \in C([0, T]; \overline{Y})$ solution of~\eqref{e:another-Cauchy}, for $t \in [0, T]$ we have, by $(v_{2})$,~$(\T_{2})$, and~$(h_{2})$, that
\begin{equation}\label{e:sm}
\| y(t) \|_{\overline{Y}} \leq \| y_{0}\|_{\overline{Y}} + (M_{v} + M_{\T})  \int_{0}^{t} (1 + \| y(\tau) \|_{\overline{Y}} + m_{1}(\Psi_{\tau})) \, \di \tau + CT\,,  
\end{equation}
where $C$ is as in~\eqref{e:first-moment-2}. Again by Gr\"onwall inequality, since $y_{0} \in \spt(\widehat{\Psi}_{0}) $ and~\eqref{e:moment-bdd-12} holds, we deduce from~\eqref{e:sm} that there exists $R>0$ independent of~$t$ such that every solution~$t \mapsto y(t)$ of the Cauchy problem~\eqref{e:another-Cauchy} takes values in~$\B^{Y}_{R}$, so that $\spt(\Psi_{t}) \subseteq \B^{Y}_{R}$ by~\eqref{e:superpos}. This implies, together with~$(v_{1})$,~$(\T_{1})$, and $(h_{1})$, that
\begin{displaymath}
\| b(t, y) \|_{\overline{Y}} \leq (M_{v} + M_{\T}) (1 + 2R) + C
\end{displaymath}
for every $t \in [0, T]$ and every $y \in \spt{\Psi}_{t}$. Since $\Psi$ solves~\eqref{e:conteq-2}, we deduce that $t \mapsto \Psi_{t}$ is Lipschitz continuous, with Lipschitz constant~$L$ only depending on~$R$. In particular, all the above computations are independent of the choice of~$(\Psi, \nnu) \in \mathcal{S}(\widehat{\Psi}_{0})$. This concludes the proof of the proposition.
\end{proof}

\begin{proposition}
For every $\widehat{\Psi}_{0} \in \sP_{c}(Y)$ the minimum problem~\eqref{e:min-mean-field} admits a solution.
\end{proposition}
\begin{proof}
The proof of existence follows from the Direct Method. Let $(\Psi_{k}, \nnu_{k}) \in \mathcal{S}(\widehat{\Psi}_{0})$ be a minimizing sequence for~\eqref{e:min-mean-field}. For every $k$, we may write $\nnu_{k} = w_{k} h_{\Psi_{k}} \Psi_{k}$ for $w_{k} \in L^{1}_{ h_{\Psi_{k}} \Psi_{k}} ([0, T]; K)$ such that
\begin{displaymath}
\Phi_{\rm min} (\Psi_{k}, \nnu_{k}) = \fint_{0}^{T} \int_{Y} \phi(w_{k}(t, y) ) \, \di \Psi_{k, t} (y) \, \di t\,.
\end{displaymath} 
Without loss of generality we may suppose $w_{k}(t, y) = 0$ in $\{ (t, y) \in [0, T] \times Y: \,  h_{\Psi_{k}} (t, y) = 0\}$.

By Proposition~\ref{p:bound-support}, $\Psi_{k, t}$ have a uniformly bounded support in~$Y$ and is equi-Lipschitz continuous. By Ascoli-Arzel\`a theorem, there exists $\Psi \in AC([0, T]; (\sP_{1}(Y); W_{1}))$ such that, up to a subsequence,~$\Psi_{k}$ converges to~$\Psi$ uniformly in $C([0, T]; (\sP_{1}(Y); W_{1}))$.

Since $\nnu_{k} = w_{k} h_{\Psi_{k}}\Psi_{k}$, we have that, up to a subsequence, $\nnu_{k} \rightharpoonup \nnu$ weakly$^*$ in~$\mathcal{M}([0, T]\times \overline{Y}; \R^{d})$. Let us define the auxiliary measure $\mmu_{k} \coloneqq w_{k} \Psi_{k} \in \mathcal{M}([0, T]\times \overline{Y}; \R^{d})$. In particular, we may assume that $\mmu_{k} \rightharpoonup \mmu$ weakly$^{*}$ in~$\mathcal{M}([0, T] \times \overline{Y} ; \R^{d})$. Thus, thanks to~$(h_{1})$, to the uniform convergence of~$\Psi_{k}$ to~$\Psi$, and to the fact that $\spt(\Psi_{k, t}) \subseteq \B^{Y}_{R}$, we also have that $\nnu = h_{\Psi} \mmu$. By definition of~$\Phi_{min}$ and of~$\overline{\Phi}$ (see~\eqref{e:limit-phi1}--\eqref{e:limit-phi2} and~\eqref{e:phi-bar}), we have that for every~$k$
\begin{displaymath}
\Phi_{min} (\Psi_{k}, \nnu_{k}) = \Phi(w_{k}, \Psi_{k}) = \overline{\Phi} (\Psi_{k}, \mmu_{k})\,,
\end{displaymath}
so that
\begin{displaymath}
\sup_{k} \,  \overline{\Phi} (\Psi_{k}, \mmu_{k})<+\infty\,.
\end{displaymath}
Applying~~\cite[Lemma~9.4.3]{AmbGigSav08} we infer that $\mmu \ll \Psi$ and
\begin{equation}\label{e:sm-liminf}
 \overline{\Phi} (\Psi, \mmu) \leq \liminf_{k\to\infty}\,  \overline{\Phi} (\Psi_{k}, \mmu_{k})\,.
\end{equation}
Since $\nnu = h_{\Psi} \mmu$, we also have that $\nnu \ll h_{\Psi} \Psi$. Moreover, since $K$ in convex and compact with~$0 \in K$, we have that $w \coloneqq \frac{\di \mmu}{\di \Psi} \in K$ for $\Psi$-a.e.~$(t, y) \in [0, T] \times \overline{Y}$ and
\begin{displaymath}
\frac{\di \nnu}{\di (h_{\Psi}\Psi)}(t, y) =  \frac{\di (h_{\Psi}\mmu)}{\di (h_{\Psi}\Psi)} (t, y)= w(t, y) \in K \qquad \text{for $h_{\Psi}\Psi$-a.e.~$(t, y) \in [0,T] \times  \overline{Y}$}.
\end{displaymath}
Thus, $ (\Psi, \nnu) \in \mathcal{S}(\widehat{\Psi}_{0})$ and, by~\eqref{e:sm-liminf},
\begin{equation}\label{e:sm-liminf-2}
\Phi_{min} (\Psi, \nnu) \leq \Phi(w, \Psi) = \overline{\Phi} (\Psi, \mmu) \leq  \liminf_{k\to\infty}\,  \overline{\Phi} (\Psi_{k}, \mmu_{k}) =  \liminf_{k\to\infty}\, \Phi_{min} (\Psi_{k}, \nnu_{k})\,.
\end{equation}
Finally, by $(\mathcal{L}_{3})$, by the uniform convergence of~$\Psi_{k}$ to~$\Psi$, and by the uniform inclusion~$\spt(\Psi_{k, t}) \subseteq \B^{Y}_{R}$, we get that
\begin{equation}\label{e:sm-liminf-3}
\fint_{0}^{T} \int_{Y} \mathcal{L} (y, \Psi_{t}) \, \di \Psi_{t} (y) \, \di t = \lim_{k \to \infty} \fint_{0}^{T} \int_{Y} \mathcal{L} (y, \Psi_{k, t}) \, \di \Psi_{k, t} (y) \, \di t\,. 
\end{equation} 
Combining~\eqref{e:sm-liminf-2} and~\eqref{e:sm-liminf-3} we infer that
\begin{displaymath}
\E^{\widehat{\Psi}_{0}} (\Psi, \nnu) \leq \liminf_{k \to \infty} \, \E^{\widehat{\Psi}_{0}} (\Psi_{k}, \nnu_{k})\,, 
\end{displaymath}
which concludes the proof of the proposition.
\end{proof}

We are now in a position to state our main $\Gamma$-convergence result.

\begin{theorem}
\label{t:Gamma-convergence}
Let $\widehat{\Psi}_{0} \in \sP_{c}(Y)$. Then the following facts hold:

\noindent \emph{\textbf{($\Gamma$-liminf inequality)}} for every sequence $(\y^{N}, \uu^{N}) \in AC([0,T];  Y^{N}) \times L^{1}([0,T]; K^{N})$ and $\y^{N}_{0} \in Y^{N}$, let $(\Psi^{N}, \nnu^{N})\in AC([0,T]; (\sP_{1}(Y); W_{1})) \times \mathcal{M}([0,T]\times  \overline{Y}; \R^{d})$ be the pair generated by $(\y^{N}, \uu^{N})$ and let $\Psi_{0}^{N}$ be the measure generated by $\y_{0}^{N}$.
Assume that $\Psi^{N}$ converges to $\Psi$ in $C([0,T]; (\sP_{1}(Y); W_{1}))$, that $\nnu^{N}$ converges weakly$^{*}$ to $\nnu$ in~$ \mathcal{M}([0,T]\times  \overline{Y}; \R^{d})$, and that $W_1(\Psi^{N}_{0},\widehat\Psi_{0})\to0$  as $N\to\infty$.
Then
\begin{equation}\label{gamma_liminf}
\E^{\widehat{\Psi}_{0}} (\Psi, \nnu) \leq \liminf_{N\to \infty} \, \E_{N}^{\y^{N}_{0}} ( \y^{N}, \uu^{N})\,.
\end{equation}

\noindent \emph{\textbf{($\Gamma$-limsup inequality)}}
for every  $(\Psi, \nnu) \in \mathcal{S}(\widehat{\Psi}_{0})$ and every sequence of initial data $\y_0^N \in Y^{N}$ such that the generated measures~$\Psi_0^N$ satisfy $W_1(\Psi_0^N,\widehat\Psi_0)\to0$, there exists a sequence $(\y^{N}, \uu^{N}) \in \mathcal{S}(\y^{N}_{0})$ with generated pairs $(\Psi^{N},\nnu^{N})\in AC([0,T]; (\sP_{1}(Y); W_{1})) \times \mathcal{M}([0,T]\times  \overline{Y} ; \R^{d})$ such that $\Psi^{N} \to \Psi$ in $C([0,T]; (\sP_{1}(Y); W_{1}))$, $\nnu^{N} \rightharpoonup \nnu$ weakly$^{*}$ in $\mathcal{M}([0,T] \times  \overline{Y}; \R^{d})$, as $N\to\infty$, and
\begin{equation}\label{gamma_limsup}
\E^{\widehat{\Psi}_{0}} ( \Psi, \nnu) \geq  \limsup_{N \to \infty} \, \E_{N}^{\y^{N}_{0}} ( \y^{N}, \uu^{N}) \,.
\end{equation} 
 \end{theorem}
 
As a corollary of  Theorem~\ref{t:Gamma-convergence}, we obtain the convergence of minima and minimizers.
 
 \begin{corollary}
 \label{c:minima}
Let $\widehat{\Psi}_{0} \in \sP_{c}(Y)$ and let $\y^{N}_{0} \in Y^{N}$ a fixed sequence of initial data with generated measure~$\Psi^{N}_{0} \in \sP^{N}(Y)$ satisfying~$W_{1}(\Psi^{N}_{0}, \widehat{\Psi}) \to 0$ as~$N \to\infty$. Then for every sequence $( \y^{N} , \uu^{N} ) \in \mathcal{S} (\y^{N}_{0})$ of solutions to~\eqref{e:min-N-particles} with generated pairs~$(\Psi^{N}, \nnu^{N})$, there exists $(\Psi, \nnu) \in \mathcal{S}(\widehat{\Psi}_{0})$ solution to~\eqref{e:min-mean-field} such that, up to a subsequence, $\Psi^{N} \to \Psi$ in $C([0, T]; (\sP_{1}(Y); W_{1}))$, $\nnu^{N} \rightharpoonup \nnu$ weakly$^{*}$ in $\mathcal{M}([0, T]\times \overline{Y}; \R^{d})$, and
\begin{displaymath}
\E^{\widehat{\Psi}_{0}} (\Psi, \nnu) = \lim_{N \to \infty} \E_{N}^{\y^{N}_{0}} (\y^{N}, \uu^{N})\,.
\end{displaymath}
\end{corollary}

\begin{proof}
The result is standard in $\Gamma$-convergence theory (see, \emph{e.g.},~\cite{Braides2002,DM1993}) and follows from the compactness result in Proposition~\ref{p:mean-field-dynamics} and from Theorem~\ref{t:Gamma-convergence}.
\end{proof}


Before proving Proposition~\ref{p:mean-field-dynamics} and Theorem~\ref{t:Gamma-convergence}, we state two lemmas relating the control part of the cost functional~$\E_{N}^{\y^{N}_{0}}$ and the functionals~$ \overline{\Phi}$ and~$\overline{\phi}$ defined in~\eqref{e:phi-bar}.

\begin{lemma}
\label{l:control1}
Let $N \in \mathbb{N}$, let $(\y^{N}, \uu^{N}) \in AC([0,T]; Y^{N}) \times L^{1}([0,T]; K^{N})$, and let~$(\Psi^{N},\nnu^{N}) \in AC([0,T]; (\sP^{N}(Y); W_{1}))\times \mathcal{M}([0,T]\times  \overline{Y};\R^d)$ be the pair generated by~$(\y^{N},\uu^{N})$; finally, let
\begin{equation}
\label{e:mu}
\mmu^{N}_{t} \coloneq \frac{1}{N} \sum_{i=1}^{N} u_{i} (t) \delta_{y_{i}(t)} \,, \qquad \mmu^{N} \coloneq \mmu^{N}_{t} \otimes \mathcal{L}^{1}\res[0,T] \,.
\end{equation}
Then, for a.e.~$t \in [0,T]$ we have
\begin{equation}
\label{e:aux1}
\frac{1}{N} \sum_{i=1}^{N} \phi(u_{i}(t)) \geq \overline{\phi} ( \Psi^{N}_{t}, \mmu^{N}_{t}) \,.
\end{equation}
If $\y_{0}^{N} \in Y^{N}$ and the pair $(\y^{N}, \uu^{N}) \in \mathcal{S} (\y_{0}^{N})$ is such that $u_{i}(t) = 0$ if $h_{\Psi^{N}_{t}} (y_{i}(t)) = 0$ for $i=1, \ldots, N$ and $t \in [0,T]$, then for a.e.~$t \in [0,T]$, we have that
\begin{equation}
\label{e:aux2}
\frac{1}{N} \sum_{i=1}^{N} \phi(u_{i}(t)) = \overline{\phi} (\Psi^{N}_{t}, \mmu^{N}_{t}) \,.
\end{equation} 
\end{lemma}

\begin{proof}
The proof of~\eqref{e:aux1} can be found in~\cite[Lemma~6.2, formula~(6.2)]{Flos}. Arguing as in the proof of~\cite[Lemma~6.2, formula~(6.3)]{Flos} we may also prove~\eqref{e:aux2}. 
Referring to the notation in \cite[Lemma~6.2]{Flos}, the only modification we have to make is that, whenever $y_{i}(t) = y_{j}(t)$ for $t \in S \subseteq[0,T]$ and for some $i\neq j$, the equality~$\dot{y}_{i}(t) = \dot{y}_{j}(t)$ for a.e.~$t \in S$ only implies that $h_{\Psi^{N}_{t}} (y_{i}(t) ) u_{i}(t) = h_{\Psi^{N}_{t}} (y_{j}(t)) u_{j}(t)$ for a.e.~$t \in S$. 
Therefore, for a.e.~$t \in S \cap \{ h_{\Psi^{N}_{t}} (y_{i}(t)) \neq0\}$ we have~$u_{i}(t) = u_{j}(t)$. Instead, for a.e.~$t \in S \cap \{ h_{\Psi^{N}_{t}} (y_{i}(t)) = 0\}$ we have $u_{i}(t) = u_{j}(t) = 0$ by assumption. This implies that $u_{i}(t) = u_{j}(t)$ a.e.~in~$S$, and the proof can be concluded as in~\cite[Lemma~6.2]{Flos}.
\end{proof}

\begin{lemma}
\label{l:AGS}
Let $\widehat{\Psi}_{0} \in \sP_{c}(Y)$ and $\y^{N}_{0} \in Y^{N}$ be such that the generated measure~$\Psi_{0}^{N}$ converges to~$\widehat{\Psi}_{0}$ in the 1-Wasserstein distance. Let~$(\y^{N}, \uu^{N}) \in \mathcal{S}(\y^{N}_{0})$ and let $(\Psi^{N}, \nnu^{N}) \in AC([0,T]; (\sP_{1}(Y); W_{1})) \times \mathcal{M}([0,T]\times  \overline{Y}; \R^{d})$ be the corresponding generated measures, according to Definition~\ref{d:generate}. Assume that
\begin{displaymath}
\sup_{N \in \mathbb{N}} \, \E_{N}^{\y^{N}_{0}} ( \y^{N}, \uu^{N}) < +\infty
\end{displaymath}
and that $\Psi^{N} \to \Psi$ uniformly in $C([0,T]; (\sP_{1}(Y); W_{1}))$ and $\nnu^{N} \rightharpoonup \nnu$ weakly$^{*}$ in~$\mathcal{M}([0,T]\times  \overline{Y} ; \R^{d})$. Then, $\nnu \ll h_{\Psi} \Psi$, $\frac{\di \nnu}{\di h_{\Psi} \Psi} \in K$ for $h_{\Psi} \Psi$-a.e.~$(t, y) \in [0,T]\times  \overline{Y}$, and
\begin{equation}
\label{e:lsc-phi}
\Phi_{\rm min} (\Psi, \nnu) \leq \liminf_{N\to\infty} \, \frac{1}{N} \sum_{i=1}^{N} \fint_{0}^{T} \phi(u_{i}(t)) \, \di t \,.
\end{equation}
\end{lemma}

\begin{proof}
We define the auxiliary measures~$\mmu^{N}$ and~$\mmu^{N}_{t}$ as in~\eqref{e:mu} and we notice that~$\nnu^{N} = h_{\Psi^{N}} \mmu^{N}$, $\nnu^{N}_{t} = h_{\Psi^{N}_{t}} \mmu^{N}_{t}$ for $t \in [0,T]$. In view of Proposition~\ref{p:bound-N-particles}, both~$\mmu^{N}$ and~$\nnu^{N}$ are supported on a compact subset of~$[0,T] \times Y$ and are bounded in $\mathcal{M}([0,T]\times  \overline{Y};\R^{d})$. In particular, we deduce that there exists~$\mmu \in \mathcal{M}([0,T]\times  \overline{Y}; \R^{d})$ such that, up to a not relabelled subsequence, $\mmu^{N} \rightharpoonup \mmu$ weakly$^{*}$ in~$\mathcal{M}([0,T]\times  \overline{Y}; \R^{d})$. Since~$(h_{2})$ holds, $\Psi^{N} \to \Psi$ uniformly in $C([0,T]; (\sP_{1}(Y); W_{1}))$, and~$\mmu^{N}$ and~$\nnu^{N}$ have uniformly compact support, in the limit it holds $\nnu = h_{\Psi} \mmu$.

By Lemma~\ref{l:control1} and by the boundedness of the energy~$\E_{N}^{\y^{N}_{0}}$, it is clear that
\begin{displaymath}
\sup_{N \in \mathbb{N}} \overline{\Phi}( \Psi^{N}, \mmu^{N}) <+\infty\,.
\end{displaymath}
Hence, we can apply~\cite[Lemma~9.4.3]{AmbGigSav08} to infer that, in the limit, $\mmu \ll \Psi$ and
\begin{equation}
\label{e:mu2}
\overline{\Phi}( \Psi, \mmu) \leq \liminf_{N\to\infty} \, \overline{\Phi}(  \Psi^{N}, \mmu^{N})\,.
\end{equation}
Furthermore, being~$K$ a convex and compact set with~$0 \in K$, we have that $\frac{\di \mmu}{\di \Psi}(t, y) \in K$ for $\Psi$-a.e.~$(t, y) \in [0,T] \times  \overline{Y}$, which implies that $\nnu \ll h_{\Psi}\Psi$ and, denoting $w \coloneq \frac{\di \mmu}{\di \Psi}$, 
\begin{displaymath}
\frac{\di \nnu}{\di (h_{\Psi}\Psi)}(t, y) =  \frac{\di (h_{\Psi}\mmu)}{\di (h_{\Psi}\Psi)} (t, y)= w(t, y) \in K \qquad \text{for $h_{\Psi}\Psi$-a.e.~$(t, y) \in [0,T] \times  \overline{Y}$},
\end{displaymath}
so that $\nnu = w h_{\Psi} \Psi$ and $\Phi(w, \Psi) = \overline{\Phi} ( \Psi, \mmu)$. Finally, by definition of~$\Phi_{\rm min}$ in~\eqref{e:limit-phi1} we get~\eqref{e:lsc-phi}.
\end{proof}

We now prove Proposition~\ref{p:mean-field-dynamics}.

\begin{proof}[Proof of Proposition~\ref{p:mean-field-dynamics}]
Let $\y^{N}_{0}$, $(\y^{N}, \uu^{N})$, $(\Psi^{N}, \nnu^{N})$, and~$\Psi^{N}_{0}$ be as in the statement of the proposition. Since~$W_1(\Psi^{N}_{0},\widehat{\Psi}_{0})\to0$ as $N\to\infty$, by Proposition~\ref{p:bound-N-particles} we obtain that for every~$t\in [0,T]$ the probability measure~$\Psi^{N}_{t}$ has support contained in the compact set~$\B^{Y}_{R}$ for a suitable~$R>0$ independent of~$t$ and~$N$. This implies that the curve~$\Psi^{N}$ takes values in a compact subset of~$\sP_{1}(Y)$ with respect to the 1-Wasserstein distance. Let us now show that the sequence~$\Psi^{N}$ is equi-continuous. Thanks to the assumptions~$(v_{1})$, $(\T_{1})$, and~$(h_{1})$, to the fact that~$\uu^{N}(t) \in K^{N}$ and $\spt ( \Psi^{N}_{t}) \subseteq \B^{Y}_{R}$  for $t \in [0,T]$, for every~$s<t \in [0,T]$ we estimate
\begin{equation}
\label{e:estimate-W1}
\begin{split}
W_{1}(\Psi^{N}_{s}, \Psi^{N}_{t}) & = \sup\bigg\{ \int_{Y} \eta(y) \, \di (\Psi^{N}_{t} - \Psi^{N}_{s})(y): \eta \in \Lip_{1}(Y) \bigg\}
\\
&
= \sup \bigg\{ \frac{1}{N} \sum_{i=1}^{N} \big( \eta(y_{i}(t)) - \eta(y_{i}(s)) \big) : \eta \in \Lip_{1}(Y) \bigg\} 
\\
&
\leq \frac{1}{N} \sum_{i=1}^{N} \int_{s}^{t} \| b_{\Psi^{N}_{\tau}} (y_{i} (\tau) ) \|_{\overline Y} \, \di \tau + \frac{1}{N} \sum_{i=1}^{N} \int_{s}^{t} | h_{\Psi^{N}_{\tau}} (y_{i} (\tau) ) u_{i}(\tau)| \, \di \tau
\\
&
\leq (2(M_{v} + M_{\T}) (1 + R) + C) |t - s|\,,
\end{split}
\end{equation}
for a positive constant~$C$ independent of~$t$, $s$, and~$N$. Therefore,~$\Psi^{N}$ is equi-continuous in~$C([0,T]; (\sP_{1}(Y); W_{1}))$ and, by Ascoli-Arzel\`a Theorem, it converges, up to a subsequence, to a limit curve~$\Psi$ in~$C([0,T]; (\sP_{1}(Y); W_{1}))$. By~\eqref{e:estimate-W1}, $\Psi$ is also Lipschitz continuous.

Since~$\uu^{N}$ takes values in~$K^{N}$ with~$K$ compact and~$h_{\Psi^{N}}$ is bounded by~$(h_{1})$, we have that, up to a further subsequence,~$\nnu^{N} \rightharpoonup \nnu$ weakly$^*$ in~$\mathcal{M}([0,T] \times  \overline{Y}; \R^{d})$. Since the cost functional~$\E_{N}^{\y^{N}_{0}} (\y^{N}, \uu^{N})$ is bounded, we deduce from Lemma~\ref{l:AGS} that~$\nnu \ll h_{\Psi}\Psi$ and $\frac{\di \nnu}{\di h_{\Psi}\Psi} (t, y) \in K$ for $h_{\Psi}\Psi$-a.e.~$(t, y) \in [0,T]\times  \overline{Y}$.

We finally show that~$(\Psi, \nnu)$ solves the corresponding continuity equation in the sense of distributions. By the uniform convergence of~$\Psi^{N}$ to $\Psi$, we have that $\Psi_{0}= \widehat{\Psi}_{0}$. For every test function~$\varphi \in C^{\infty}_{c}((0,T) \times \overline{Y})$, since $(\y^{N}, \uu^{N}) \in \mathcal{S}(\y^{N}_{0})$, we have that, for every $t \in [0,T]$,
\begin{equation}
\label{e:lim1}
\begin{split}
\int_{ \overline{Y}} \varphi(t, y) \, \di \Psi_{t}^{N}(y) & = \int_{0}^{t} \int_{ \overline{Y}} \partial_{t} \varphi(\tau,  y) \, \di \Psi^{N}_{\tau}(y) \, \di \tau + \int_{0}^{t} \int_{ \overline{Y}} \nabla \varphi (\tau, y) \cdot b_{\Psi^{N}_{\tau}} (y) \, \di \Psi^{N}_{\tau}(y) \, \di \tau
\\
&
\qquad  + \int_{0}^{t} \int_{ \overline{Y}} \nabla_{x} \varphi (\tau, y) \, \di \nnu^{N}_{t}( y)\, \di t 
\end{split}
\end{equation}
Since $\Psi^{N} \to \Psi$ in~$C([0,T]; (\sP_{1}(Y); W_{1}))$ and $\nnu^{N} \rightharpoonup \nnu$ weakly$^{*}$ in~$\mathcal{M}([0,T]\times  \overline{Y}; \R^{d})$, we only have to determine the limit of the second integral on the right-hand side of~\eqref{e:lim1}. To do this, we estimate
\begin{equation}
\label{e:lim2}
\begin{split}
\bigg| \int_{0}^{t}   \bigg( \int_{ \overline{Y}} &  \nabla \varphi (\tau, y) \cdot b_{\Psi^{N}_{\tau}} (y) \, \di \Psi^{N}_{\tau}(y) - \int_{ \overline{Y}} \nabla \varphi (\tau, y) \cdot b_{\Psi_{\tau}} (y) \, \di \Psi_{\tau}(y) \bigg) \, \di \tau \bigg|
\\
&
\leq \int_{0}^{t} \int_{ \overline{Y}} \big | \nabla \varphi (\tau, y) \cdot \big( b_{\Psi^{N}_{\tau}} (y) - b_{\Psi_{\tau}} (y)\big) \big| \, \di \Psi^{N}_{\tau}(y) \, \di \tau 
\\
&
\qquad +  \bigg|\int_{0}^{t} \int_{ \overline{Y}}  \nabla \varphi (\tau, y) \cdot b_{\Psi_{\tau}} (y) \, \di ( \Psi^{N}_{\tau} - \Psi_{\tau}) (y) \, \di \tau \bigg|= : I^{N}_{1} + I^{N}_{2}\,. 
\end{split}
\end{equation}
By the regularity of the test function~$\varphi$, by assumptions~$(v_{2})$ and~$(\T_{2})$, and by the uniform inclusion $\spt(\Psi^{N}_{t}) \subseteq \B^{Y}_{R}$, we may estimate~$I_{1}^{N}$ with
\begin{align*}
I^{N}_{1} \leq L_{R}\| \varphi\|_{C^{\infty}_{c}((0,T)\times \overline{Y})}  \int_{0}^{t} W_{1}( \Psi^{N}_{\tau} , \Psi_{\tau}) \, \di \tau\,,
\end{align*}
for a positive constant~$L_{R}$ depending only on~$R$. Since $\Psi^{N} \to \Psi$ in $C([0,T]; (\sP_{1}(Y); W_{1}))$, we deduce from the previous inequality that~$I_{1}^{N} \to 0$ as $N\to \infty$. Again by~$(v_{2})$ and~$(\T_{2})$, the function $y \mapsto \nabla \varphi(t, y) b_{\Psi_{t}}(t)$ is Lipschitz continuous on~$\B^{Y}_{R}$ for every $t \in [0,T]$, with Lipschitz constant~$C_{R}>0$ uniformly bounded in time. Since $\spt(\Psi_{t}), \spt(\Psi^{N}_{t}) \subseteq \B^{Y}_{R}$ for $t \in [0,T]$, we estimate~$I^{N}_{2}$ with
\begin{align*}
I^{N}_{2} \leq C_{R} \int_{0}^{t} W_{1}(\Psi^{N}_{\tau}, \Psi_{\tau}) \, \di \tau\,,
\end{align*}
and $I^{N}_{2} \to 0$ as $N\to\infty$. We can now pass to the limit in~\eqref{e:lim2} to obtain that
\begin{displaymath}
\lim_{N\to\infty} \int_{0}^{t}   \int_{ \overline{Y}}   \nabla \varphi (\tau, y) \cdot b_{\Psi^{N}_{\tau}} (y) \, \di \Psi^{N}_{\tau}(y)\, \di \tau = \int_{0}^{t}    \int_{ \overline{Y}}   \nabla \varphi (\tau, y) \cdot b_{\Psi_{\tau}} (y) \, \di \Psi_{\tau}(y) \, \di \tau\,,
\end{displaymath}
which in turn implies, by passing to the limit in~\eqref{e:lim1}, that
\begin{align*}
\int_{ \overline{Y}} \varphi(t, y) \, \di \Psi_{t} (y) & = \int_{0}^{t} \int_{Y} \partial_{t} \varphi(\tau,  y) \, \di \Psi_{\tau}(y) \, \di \tau + \int_{0}^{t} \int_{ \overline{Y}} \nabla \varphi (\tau, y) \cdot b_{\Psi_{\tau}} (y) \, \di \Psi_{\tau}(y) \, \di \tau
\\
&
\qquad  + \int_{0}^{t} \int_{ \overline{Y}} \nabla_{x} \varphi (\tau, y) \, \di \nnu_{t}( y)\, \di t \,.
\end{align*}
By the arbitrariness of~$\varphi \in C^{\infty}_{c}((0,T)\times \overline{Y})$, we conclude that~$(\Psi, \nnu) \in \mathcal{S}(\widehat{\Psi}_{0})$. This completes the proof.
\end{proof}

Eventually, we prove the $\Gamma$-convergence result.

\begin{proof}[Proof of Theorem~\ref{t:Gamma-convergence}]
We divide the proof into two steps.

\noindent\textbf{Step~1: $\Gamma$-liminf inequality}. Let $(\Psi, \nnu)$,~$(\y^{N}, \uu^{N})$,~$\y^{N}_{0}$, $(\Psi^{N},\nnu^{N})$, and~$\Psi^{N}_{0}$ be as in the statement. If $\liminf_{N\to\infty} \E_{N}^{\y^{N}_{0}} (\y^{N}, \uu^{N}) =+\infty$ there is nothing to show. Without loss of generality we may therefore assume that
\begin{displaymath}
\sup_{N \in \mathbb{N}} \, \E_{N}^{\y^{N}_{0}} (\y^{N}, \uu^{N}) <+\infty\,,
\end{displaymath}
which implies, by definition~\eqref{e:cost-N-particles} of~$\E_{N}^{\y^{N}_{0}}$, that $(\y^{N}, \uu^{N}) \in \mathcal{S}(\y^{N}_{0})$ for every~$N$. Furthermore, by Proposition~\ref{p:bound-N-particles} there exists $R>0$ independent of~$N$ and~$t$ such that $\spt(\Psi^{N}_{t}) \subseteq \B^{Y}_{R}$. By Proposition~\ref{p:mean-field-dynamics} we have that the limit pair~$(\Psi, \nnu)$ belongs to~$\mathcal{S}(\widehat{\Psi}_{0})$ and~$\spt(\Psi_{t}) \subseteq \B^{Y}_{R}$ for every $t \in [0,T]$. Applying Lemma~\ref{l:AGS} we infer that
\begin{equation}
\label{e:liminf1}
\Phi_{\rm min}( \Psi, \nnu) \leq \liminf_{N \to \infty} \frac{1}{N} \sum_{i=1}^{N} \fint_{0}^{T} \phi(u_{i}^{N} (t)) \, \di t \,.
\end{equation}
Since $\cL_{N}$ $\sP_{1}$-converges to~$\cL$ uniformly on compact sets and~$\spt(\Psi^{N}_{t}) , \spt(\Psi_{t}) \subseteq \B^{Y}_{R}$, we get that
\begin{equation}
\label{e:liminf2}
\begin{split}
&\fint_{0}^{T} \int_{Y} \cL(y, \Psi_{t}) \, \di \Psi_{t}(y) \, \di t  = \fint_{0}^{T} \int_{\B^{Y}_{R}}  \cL(y, \Psi_{t}) \, \di \Psi_{t}(y) \, \di t \\
=& \lim_{N\to \infty} \fint_{0}^{T} \int_{\B^{Y}_{R}} \mathcal{L}_{N}(y, \Psi^{N}_{t}) \, \di \Psi^{N}_{t}(y) \, \di t 
= \lim_{N\to \infty} \fint_{0}^{T} \int_{Y}\mathcal{L}_{N}(y, \Psi^{N}_{t}) \, \di \Psi^{N}_{t}(y) \, \di t 
\\
= &
 \lim_{N\to \infty} \frac{1}{N} \sum_{i=1}^{N} \fint_{0}^{T} \mathcal{L}_{N}(y_{i}^{N}(t), \Psi^{N}_{t}) \, \di t \,. 
\end{split}
\end{equation}
Combining~\eqref{e:liminf1} and~\eqref{e:liminf2} we conclude that
\begin{displaymath}
\E^{\widehat{\Psi}_{0}} ( \Psi, \nnu) \leq \liminf_{N \to \infty} \, \E^{\y^{N}_{0}}_{N} ( \y^{N}, \uu^{N}) \,,
\end{displaymath}
which is \eqref{gamma_liminf}.

\noindent\textbf{Step~2: $\Gamma$-limsup inequality}. 
We will construct a sequence $(\y^{N},\uu^{N})$ such that
\begin{equation}\label{gamma_limit}
\E^{\widehat\Psi_0}(\Psi,\nnu)=\lim_{N\to\infty} \E_{N}^{\y_{0}^{N}}( \y^{N}, \uu^{N})
\end{equation}
and we recall that this condition is equivalent to 
\eqref{gamma_limsup}.

 Let~$(\Psi, \nnu) \in \mathcal{S}(\widehat{\Psi}_{0})$ be such that~$\E^{\widehat{\Psi}_{0}}(\Psi, \nnu) <+\infty$, and let $w \in L^{1}_{h_{\Psi} \Psi}([0,T]\times Y; K)$ be such that $\nnu= w h_{\Psi} \Psi$ and $\Phi_{\rm min}(\Psi, \nnu) = \Phi(w, \Psi, \nnu)$. In particular, we may assume that $w = 0$ on the set~$\{(t, y) \in [0,T] \times \overline{Y}: \, h_{\Psi_{t}} (y) = 0 \}$. Indeed, we notice  that the function~$\overline{w} (t, y)\coloneq w (t, y)\mathbf{1}_{ \{h_{\Psi} \neq 0\}}(t, y)$ still belongs to~$L^{1}_{h_{\Psi}\Psi} ([0,T]\times Y; K)$ and satisfies $\nnu = \overline{w} h_{\Psi} \Psi$ and, $\Phi(\overline{w},  \Psi, \nnu) = \Phi(w, \Psi, \nnu)$, by the minimality of $w$.

As in~\cite[Theorem~3.2]{Flos}, the construction of a recovery sequence is based on the superposition principle~\cite[Theorem~5.2]{AmbForMorSav18}. The curve~$\Psi \in AC([0,T]; (\sP_{1}(Y); W_{1}))$ solves indeed the continuity equation
\begin{align}
\label{e:conteq}
\partial_{t} \Psi_{t} + \dive ( b(t, \cdot) \Psi_{t}) = 0 \qquad \text{with $\Psi_{0}= \widehat{\Psi}_{0}$}\,,
\end{align}
where the velocity field~$b \colon [0,T]\times \overline{Y} \to \overline{Y}$ is defined by 
\begin{align}\label{e:flux}
b(t, y) \coloneq b_{\Psi_{t}}(y) + \left( \begin{array}{cc}
h_{\Psi_{t}} (y)w(t, y) \\
0
\end{array} \right) \qquad \text{for $y \in Y$}
\end{align} 
and is extended to $0$ in~$\overline{Y} \setminus Y$. By Proposition~\ref{p:bound-support}, there exists $R>0$ such that $\spt (\Psi_{t}) \subseteq \B^{Y}_{R}$ for every $t \in [0, T]$. Thus, by~$(v_{1})$,~$(\T_{1})$, and~$(h_{1})$, and by the fact that $w(t, y) \in K$, we can estimate
\begin{equation}
\label{e:limsup1}
\begin{split}
\int_{0}^{T} \int_{\overline{Y}} \,\| b(t, y)\|_{\overline{Y}} \, \di \Psi_{t}(y) \, \di t  \leq & \ (M_{v} + M_{\T}) \int_{0}^{T} \int_{\overline{Y}} (1 + \| y \|_{\overline{Y}} + m_{1}(\Psi_{t}) ) \, \di \Psi_{t}(y) \, \di t 
\\
&
 + \int_{0}^{T} \int_{\overline{Y}} |h_{\Psi_{t}} (y)w(t, y)  |\, \di \Psi_{t}(y) \, \di t < +\infty\,. 
\end{split}
\end{equation}
We are therefore in a position to apply~\cite[Theorem~5.2]{AmbForMorSav18} with velocity field~$b$. Setting
\begin{align*}
& \Gamma \coloneq C([0,T]; \overline{Y})\,,\\
& \Delta \coloneq \bigg\{ y \in \Gamma: \, \text{$y (t) \in Y$ for $t \in [0,T]$ and solves}
\\
& \hphantom{------}\,\,\, \,\text{$\dot{y}(t) = b_{\Psi_{t}}(y(t)) + \left( \begin{array}{cc}
h_{\Psi_{t}} (y(t)) w(t, y(t)) \\
0
\end{array}\right)$ in $[0,T]$ with $y(0) \in \spt(\widehat{\Psi}_{0})$} \bigg\}\,,
\end{align*} 
we infer that there exists a probability measure~$\pi \in \sP(\Gamma)$ concentrated on~$\Delta$ such that for every $t \in [0,T]$~$\Psi_{t} = (\ev_{t})_{\#}\pi$, where $\ev_{t}\colon \Gamma \to \overline{Y}$ denotes the evaluation map defined as $\ev_{t}(y) \coloneq y(t)$ for every $y \in \Gamma$.

We define the auxiliary functional
\begin{equation}
\label{e:limsup2}
\F( y) \coloneq \fint_{0}^{T} \phi(w(t, y(t))) \, \di t \qquad \text{for every $y \in \Delta$}.
\end{equation}
We notice that by Fubini Theorem
\begin{align*}
\int_{\Gamma} \F(y) \, \di \pi(y) & = \int_{\Gamma} \fint_{0}^{T} \phi(w(t, y(t))) \, \di t \, \di \pi(y) =  \fint_{0}^{T}\int_{\Gamma} \phi(w(t, \ev_{t}(y))) \, \di \pi(y)\, \di t 
\\
&
=  \fint_{0}^{T}\int_{Y}  \phi(w(t,y) ) \, \di \Psi_{t}(y) \, \di t = \Phi_{\rm min} (\Psi, \nnu)\,.
\end{align*}
Furthermore,~$\F$ is lower semicontinuous in~$\Delta$. Indeed, if~$y_{k}, y \in \Delta$ are such that $y_{k} \to y$ with respect to the uniform convergence in~$\Gamma$, since~$w$ takes values in the compact set~$K$ we immediately deduce that~$w(\cdot, y_{k}(\cdot))$ is bounded in~$L^{\infty}([0,T]; \R^{d})$, and therefore converges weakly$^*$, up to a subsequence, to some $g \in L^{\infty}([0,T]; \R^{d})$ and, by convexity of~$\phi$,
\begin{displaymath}
\fint_{0}^{T} \phi(g(t)) \, \di t \leq \liminf_{k\to \infty}  \, \F(y_{k}) \,.
\end{displaymath}
Since~$y_{k} \in \Delta$ for every $k$, for $s<t \in [0,T]$ we can write
\begin{displaymath}
y_{k}(t) - y_{k}(s) = \int_{s}^{t} \left( b_{\Psi_{\tau}} (y_{k}(\tau)) + \left( \begin{array}{cc}
h_{\Psi_{\tau}} (y_{k}(\tau)) w(\tau, y_{k}(\tau)) ) \\
0
\end{array} \right)\right) \di \tau\,.
\end{displaymath}
Passing to the limit in the previous equality we deduce, thanks to~$(v_{2})$, $(\T_{2})$, and~$(h_{2})$,
\begin{displaymath}
y(t) - y(s) = \int_{s}^{t} \left(b_{\Psi_{\tau}} (y(\tau)) +  \left( \begin{array}{cc}
h_{\Psi_{\tau}} (y(\tau)) g(\tau)  \\
0 \end{array}\right)\right) \di \tau\,.
\end{displaymath}
On the other hand, being~$y \in \Delta$ we have that
\begin{displaymath}
y(t) - y(s) = \int_{s}^{t} \left(b_{\Psi_{\tau}} (y(\tau)) +  \left( \begin{array}{cc}
h_{\Psi_{\tau}} (y(\tau)) w(\tau, y(\tau)) \\
0
\end{array} \right)\right)  \di \tau\,,
\end{displaymath}
which implies, by the arbitrariness of~$s$ and~$t$, that $h_{\Psi_{\tau}} (y(\tau))  g(t) = h_{\Psi_{\tau}} (y(\tau))  w(t, y(t))$ for a.e.~$t \in [0,T]$. Hence, $g(t) = w(t, y(t))$ for a.e.~$t \in \{ s \in [0,T] : \, h_{\Psi_{s}} (y(s))  \neq 0\}$, while $w(t, y(t)) = 0$ for $t \in \{ s \in [0,T]: \, h_{\Psi_{s}} (y(s))  = 0 \}$. Since~$\phi \geq 0$ and $\phi (0) = 0$, we finally obtain
\begin{displaymath}
\F(y) \leq \fint_{0}^{T} \phi(g(t)) \, \di t \leq \liminf_{k\to \infty} \F(y_{k}) \,.
\end{displaymath}

By Lusin theorem, we can select an increasing sequence of compact sets $\Delta_{k} \Subset \Delta_{k+1} \Subset \Delta$ such that $\pi( \Delta \setminus \Delta_{k}) <\frac{1}{k}$ and $\F$ is continuous on~$\Delta_{k}$. Setting
\begin{displaymath}
\overline{\pi}_{k} \coloneq \frac{1}{\pi(\Delta_{k}) } \, \pi\lfloor \Delta_{k} \in \sP(\Gamma)\,,
\end{displaymath}
we have that 
\begin{equation}
\label{e:limsup3}
\lim_{k\to\infty} \, W_{1}(\pi, \overline{\pi}_{k}) = 0\,, \qquad  \lim_{k\to\infty} \int_{\Gamma} \F(y) \, \di \overline{\pi}_{k}(y) = \int_{\Gamma} \F(y) \, \di \pi(y)\,.
\end{equation}

Let us fix a countable dense set $D \coloneq \{\varphi_{\ell} \}_{\ell \in \mathbb{N}}$ in~$C_{c}([0,T]\times \overline{Y}; \R^{d})$. Since~$\Delta_{k}$ is compact, we can select a sequence of curves $\{(y_k)_{i}^{m}: \, i =1, \ldots, m, \, m \in \mathbb{N}\} \subseteq \Delta_{k} $ such that for every $k$ the measures
\begin{displaymath}
\overline{\pi}_{k}^{m} \coloneq \frac{1}{m} \sum_{i=1}^{m} \delta_{(y_k)_{i}^{m}} \in \sP(\Gamma)
\end{displaymath}
satisfy
\begin{equation}
\label{e:limsup4}
\lim_{m\to\infty} \, W_{1}( \overline{\pi}_{k}^{m}, \overline{\pi}_{k}) = 0\,, \qquad \lim_{m\to\infty} \int_{\Gamma} \F(y) \, \di \overline{\pi}_{k}^{m}(y)=  \int_{\Gamma} \F(y) \, \di \overline{\pi}_{k}(y)\,,
\end{equation}
where the second equality is due to the fact that~$\F$ is continuous and bounded on~$\Delta_{k}$.

We recall that, by construction, on the set~$\Delta_{k}$ the function $y \mapsto \F(y)$ is continuous. Since~$\phi$ is superlinear, this implies that $w(\cdot, \gamma_{j} (\cdot) ) \to w(\cdot, \gamma(\cdot))$ in $L^{p}([0,T]; \R^{d})$ for every $p<+\infty$ whenever $\gamma_{j}, \gamma \in \Delta_{k}$ with $\gamma_{j} \to \gamma$. Hence, also the map
\begin{displaymath}
y \mapsto  \int_{0}^{T}  \varphi_{\ell} (t, y(t)) w(t, y(t)) h_{\Psi_{t}} (y(t)) \, \di t
\end{displaymath}
is continuous in~$\Delta_{k}$ for every $\ell \in \mathbb{N}$. Combining this fact with~\eqref{e:limsup3} and~\eqref{e:limsup4}, we are able to select a suitable strictly increasing sequence~$m(k)$ such that for every $m \geq m(k)$ it holds
\begin{align}
\label{e:limsup4.1}
&W_{1}(\overline{\pi}_{k}^{m}, \overline{\pi}_{k}) <\frac{1}{k} \,,\\
&
\label{e:limsup4.2} \bigg| \int_{\Gamma} \F(y) \, \di \overline{\pi}_{k}^{m} - \int_{\Gamma} \F(y) \, \di \overline{\pi}_{k} (y) \bigg| <\frac{1}{k}\,,
\\
&
\label{e:limsup4.3} \bigg| \int_{\overline{Y}} \int_{0}^{T}  \varphi_{\ell} (t, y(t)) w(t, y(t)) h_{\Psi_{t}} (y(t)) \, \di t \, \di (\overline{\pi}_{k}^{m} - \overline{\pi}_{k} ) (y) \bigg| \leq \frac{1}{k} \qquad \text{for $\ell \leq k$}\,,
\end{align}
where in the last inequality we have used that $\overline{\pi}_{k}^{m}$ converges narrowly to $\overline{\pi}_{k}$ as $m\to\infty$ and that~$\overline{\pi}_{k}^{m}$ is concentrated on curves belonging to~$\Delta_{k}$.

Therefore we set $\pi_{N} \coloneq \overline{\pi}_{k}^{N}$ for $m(k) \leq N < m(k+1)$ and obtain that 
\begin{equation}\label{101}
\lim_{N\to\infty} W_1(\pi_{N},\pi)=0,
\end{equation}
so that
\begin{equation}
\label{e:limsup5}
\lim_{N\to \infty} \int_{\Gamma} \F(y) \, \di \pi_{N}(y) = \Phi_{\rm min} (\Psi, \nnu)\,.
\end{equation}

We now construct the recovery sequence~$(\y^{N}, \uu^{N})$. First, we define the auxiliary curves~$\Lambda_{t}^{N} \coloneq (\ev_{t})_{\#} \pi_{N} \in AC([0,T]; (\sP_{1}(Y); W_{1}))$ and the corresponding curves~$\z^{N} = (z_{1}, \ldots, z_{N}) \in AC([0,T]; Y^{N})$ so that $\Lambda_{t}^{N}=\frac{1}{N}\sum_{i=1}^{N} \delta_{z_{i}(t)}$. Then, we set  $\overline{u}_{i} (t) \coloneq w(t, z_{i}(t))$ for every $t \in [0,T]$, every $i=1, \ldots, N$, and every $N \in \mathbb{N}$, and $\overline{\uu}^{N} \coloneqq (\overline{u}_{1}, \ldots, \overline{u}_{N}) \in L^{1}([0, T]; K^{N})$. In particular, each component of~$\z^{N}$ solves the ODE
\begin{equation}
\label{e:limsup6}
\dot{z}_{i} (t) = b_{\Psi_{t}} (z_{i}(t)) + \left( \begin{array}{cc}
 h_{\Psi_{t}} (z_{i}(t)) \overline{u}_{i}(t) \\
 0
 \end{array}\right)
\end{equation}
with initial point~$z_{i}(0) \in \spt( \widehat{\Psi}_{0})$. The curves~$\z^{N}$ have to be further modified, since in the ODE~\eqref{e:limsup6} the velocity field~$b_{\Psi_{t}}$ still contains the state of the limit system~$\Psi_{t}$ rather than~$\Lambda^{N}$, and the initial data~$\z^{N}_{0} = (z_{1} (0), \ldots, z_{N}(0))$ do not coincide with~$\y^{N}_{0}$.

Being~$\Psi^{N}_{0}$ and $\Lambda^{N}_{0}$ two empirical measures, we can find a sequence of permutations~$\sigma^{N} \colon Y^{N} \to Y^{N}$ such that
\begin{equation}
\label{e:permutation}
W_{1}(\Psi^{N}_{0}, \Lambda^{N}_{0}) = \frac{1}{N} \sum_{i=1}^{N} \| \big( \sigma^{N}(\y^{N}_{0})\big)_{i} - z_{i}(0)\|\,.
\end{equation}
Let us further denote by~$\sigma^{N}_{\R^{d}}\colon (\R^{d})^{N} \to (\R^{d})^{N}$ the spatial component of~$\sigma^{N}$. We set $\overline{\y}^{N}_{0} \coloneqq \sigma^{N} (\y^{N}_{0})$ and denote by $\overline{y}_{0, i}$ its $i$-th component. We define $\overline{\y}^{N} = (\overline{y}_{ 1}, \ldots, \overline{y}_{ N}) \in AC([0,T]; Y^{N})$ by solving for $i=1, \ldots, N$ the Cauchy problems
\begin{equation}
\label{e:limsup7}
\left\{ \begin{array}{ll}
\dot{\overline{y}}_{i} (t) = b_{\Psi^{N}_{t}} (\overline{y}_{i}(t)) +  \left( \begin{array}{cc}
h_{\Psi^{N}_{t}} (\overline{y}_{i}(t)) \overline{u}_{i}(t) \\
0
\end{array}\right)\,,\\[2mm]
\overline{y}_{i}(0) = \overline{y}_{0, i}\,,
\end{array}\right.
\end{equation}
where, as for the Cauchy problem in~\eqref{e:Cauchy}, we have set $\Psi^{N}_{t} \coloneq \frac{1}{N} \sum_{i=1}^{N} \delta_{\overline{y}_{i}(t)} \in \sP^N (Y)$.
By~\cite[Corollary~2.3]{MorSol19}  system~\eqref{e:limsup7} admits a unique solution and $(\overline{\y}^{N}, \overline{\uu}^{N})  \in \mathcal{S}(\overline{\y}^{N}_{0})$. Finally, we set $(\y^{N}, \uu^{N}) \coloneqq ( (\sigma^{N})^{-1}( \overline{\y}^{N}), (\sigma^{N}_{\R^{d}})^{-1} (\overline{\uu}^{N})) \in \mathcal{S}(\y^{N}_{0})$.

We denote by~$(\Psi^{N}, \nnu^{N})$ and $(\Lambda^{N}, \eeta^{N})$ the pairs generated by~$(\overline{\y}^{N}, \overline{\uu}^{N})$ and by~$(\z^{N}, \overline{\uu}^{N})$, respectively, and notice that, by invariance with respect to permutations, $(\Psi^{N}, \nnu^{N})$ coincides with the pair generated by~$(\y^{N}, \uu^{N})$. We want to show that 
\begin{equation}\label{100}
\text{$\Psi^{N} \to \Psi$ in~$C([0,T]; (\sP_{1}(Y); W_{1}))$ and~$\nnu^{N} \rightharpoonup \nnu$ weakly$^*$ in~$\mathcal{M}( [0,T] \times  \overline{Y}; \R^{d})$.}
\end{equation}
To do this, we will prove that 
\begin{equation}\label{4.33}
\lim_{N\to\infty} \sup_{t\in[0,T]} W_1(\Psi^{N}_{t},\Lambda^{N}_{t})=0\quad\text{and}\quad \lim_{N\to\infty} \sup_{t\in[0,T]} W_1(\Lambda^{N}_{t},\Psi_t)=0
\end{equation}
and that
\begin{equation}\label{4.35}
\nnu^{N}-\eeta^{N} \rightharpoonup 0 \quad\text{and} \quad \eeta^N \rightharpoonup \nnu \quad \text{weakly$^*$ in~$\mathcal{M}([0,T]\times  \overline{Y}; \R^{d})$,}
\end{equation}
so that \eqref{100} follows by triangle inequality.

Let us consider the pair~$(\Lambda^{N}, \eeta^{N})$. Since~$z_{i}(0)\in\spt( \widehat{\Psi}_{0})$ for every $i=1,\ldots,N$ and~$\widehat{\Psi}_{0} \in \sP_{c}(Y)$, Proposition~\ref{p:bound-N-particles} yields the existence of~$R>0$ independent of~$N$ and~$t$ such that~$\spt (\Lambda^{N}_{t}) \subseteq \B^{Y}_{R}$ for every $t\in[0,T]$. Repeating the computations performed in~\eqref{e:estimate-W1} we obtain that~$\Lambda^{N}$ is equi-Lipschitz continuous with respect to~$t$.
The convergence in \eqref{101} implies that $W_1(\Lambda^{N}_{t},\Psi_{t})\to0$ for every $t\in[0,T]$ as $N\to\infty$, so that and application of Ascoli-Arzel\`a Theorem yields that  $\Lambda^{N} \to \Psi$ in~$C([0,T]; (\sP_{1}(Y); W_{1}))$. This proves the second convergence in \eqref{4.33}. 

To prove the first convergence in \eqref{4.33}, we estimate the distance between~$\overline{\y}^{N}$ and~$\z^{N}$. First we notice that, up to possibly taking a larger~$R$, we have that~$\| \overline{y}_{i}(t)\|_{\overline{Y}} \leq R$ for every~$i=1,\ldots,N$ for every $N\in\N$ and for every~$t \in [0,T]$, so that $\spt(\Psi^{N}_{t}) \subseteq \B^{Y}_{R}$. For every $t \in [0,T]$ and every $i=1, \ldots, N$ we have, by definition of~$\overline{y}_{i}$ and~$z_{i}$ and by assumptions~$(v_{2})$, $(\T_{2})$, and~$(h_{2})$,
\begin{equation}
\label{e:limsup8}
\begin{split}
\| z_{i}(t) - \overline{y}_{i}(t)\|_{\overline{Y}}  \leq & \| z_{i}(0) -  \overline{y}_{0, i} \| +  \int_{0}^{t} \| b_{\Psi_{\tau}} (z_{i}(\tau)) - b_{\Psi^{N}_{\tau}}(\overline{y}_{i} (\tau) ) \|_{\overline{Y}} \, \di \tau \\
& + \int_{0}^{t} | h_{\Psi_{\tau}} (z_{i}(\tau)) - h_{\Psi^{N}_{\tau}} (\overline{y}_{i}(\tau))| \, | w(\tau, z_{i}(\tau))| \, \di \tau
\\
\leq &\,
 \| z_{i}(0) -  \overline{y}_{0, i} \| + L_{R} \int_{0}^{t} \big( \| z_{i}(\tau) - \overline{y}_{i}(\tau) \|_{\overline{Y}} + W_{1}( \Psi_{\tau}, \Psi^{N}_{\tau}) \big) \, \di \tau \,,
\end{split}
\end{equation}
for some positive constant~$L_{R}$ independent of~$N$. Hence, by Gr\"onwall inequality we deduce from~\eqref{e:limsup8} that
\begin{equation}
\label{e:conv-min-1}
\| \overline{y}_{i}(t) - z_{i}(t)\|_{\overline{Y}} \leq  e^{L_{R} T} \Big( \| z_{i}(0) - \overline{y}_{0, i}\|_{\overline{Y}} + L_{R} \int_{0}^{t} W_{1} ( \Psi^{N}_{\tau} , \Lambda^{N}_{\tau})  \, \di \tau \Big).  
\end{equation}
Summing~\eqref{e:conv-min-1} over $i=1, \ldots, N$ and recalling~\eqref{e:permutation}, we infer that for every $t \in [0, T]$
\begin{equation}
\label{e:limsup8-1}
W_{1}(\Psi^{N}_{t} , \Lambda^{N}_{t}) \leq \frac{1}{N} \sum_{i=1}^{N} \| z_{i}(t) - \overline{y}_{i}(t) \| \leq   e^{L_{R}T} W_{1}( \Psi^{N}_{0}, \Lambda^{N}_{0}) +  L_{R}e^{L_{R}T}  \int_{0}^{t} W_{1} ( \Psi^{N}_{\tau} , \Lambda^{N}_{\tau})  \, \di \tau \,.
\end{equation}
Applying once again Gr\"onwall inequality to~\eqref{e:limsup8-1} we obtain for every $t \in [0, T]$
\begin{equation}
\label{e:conv-min-3}
W_{1}(\Psi^{N}_{t}, \Lambda^{N}_{t}) \leq e^{L_{R} T(1+ e^{L_{R}T})}\, W_{1} (\Psi^{N}_{0}, \Lambda^{N}_{0}) \,.
\end{equation}
Since~$W_{1}( \Lambda^{N}_{0}, \widehat{\Psi}_{0}) \to 0$ and~$W_{1}(\Psi^{N}_{0}, \widehat{\Psi}_{0}) \to 0$, from~\eqref{e:conv-min-3} we conclude~\eqref{4.33} and the convergence of~$\Psi^{N}$ to~$\Psi$ in~$C([0, T]; (\sP_{1}(Y); W_{1}))$.

We now turn our attention to~\eqref{4.35}. The second convergence in \eqref{4.35} is a matter of a direct computation.
Indeed, for every $\varphi \in C_{c}([0,T] \times \overline Y; \R^{d})$ and every $\varepsilon >0$ we can fix $\varphi_{\ell} \in D$ such that $\| \varphi - \varphi_{\ell}\|_{C([0,T]\times \overline{Y})} \leq \varepsilon$ and estimate
\begin{equation}\label{e:limsup7.1}
\begin{split}
& \bigg| \int_{0}^{T} \int_{\overline{Y}}  \varphi(t, y) \, \di (\eeta^{N} - \nnu) (t, y) \bigg| 
\\
&
 \leq  \int_{0}^{T} \int_{\overline{Y}}| \varphi(t, y) - \varphi_{\ell}(t, y) |  \, \di |\eeta^{N} - \nnu| (t, y)  + \bigg| \int_{0}^{T} \int_{\overline{Y}} \varphi_{\ell}(t, y) \, \di (\eeta^{N} - \nnu) (t, y) \bigg| 
\\
&
\leq C \varepsilon + \bigg|  \frac{1}{N} \sum_{i=1}^{N} \int_{0}^{T} \varphi_{\ell}(t, z_{i}(t)) w(t, z_{i}(t)) h_{\Psi_{t}}(z_{i}(t))  \, \di t 
\\
& \qquad\qquad 
- \int_{0}^{T}\int_{\overline{Y}} \varphi_{\ell}(t, y) w(t, y)  h_{\Psi_{t}} (y) \, \di \Psi_{t}(y) \, \di t \bigg| 
\\
&
= C \varepsilon + \bigg|  \int_{0}^{T} \int_{\overline{Y}} \varphi_{\ell} (t, y) w(t, y) h_{\Psi_{t}}(y) \, \di \Lambda^{N}_{t}(y) \,\di t 
\\
& \qquad\qquad 
- \int_{0}^{T}\int_{\overline{Y}} \varphi_{\ell}(t, y) w(t, y) h_{\Psi_{t}}(y)  \, \di \Psi_{t}(y) \, \di t \bigg| 
\\
&
= C \varepsilon + \bigg| \int_{\Gamma} \int_{0}^{T} \varphi_{\ell}(t, y(t)) w(t, y(t)) h_{\Psi_{t}} (y(t)) \, \di t \, \di \pi_{N}(y) 
\\
&
\qquad\qquad- \int_{\Gamma} \int_{0}^{T} \varphi_{\ell}(t, y(t)) w(t, y(t)) h_{\Psi_{t}} (y(t)) \, \di t \, \di \pi (y)  \bigg| \,,
\end{split}
\end{equation}
for some positive constant~$C$ independent of~$\varepsilon$. We now estimate the right-hand side of~\eqref{e:limsup7.1}. By definition of~$\pi_{N}$ and by~\eqref{e:limsup4.1} and~\eqref{e:limsup4.3}, for every $N \in [m(k), m(k+1))$ with $k \geq \ell$ we have that
\begin{align*}
& \bigg| \int_{\Gamma} \int_{0}^{T}  \varphi_{\ell}(t, y(t)) w(t, y(t)) h_{\Psi_{t}} (y(t)) \, \di t \, \di \pi_{N}(y) 
\\
&
\qquad - \int_{\Gamma} \int_{0}^{T} \varphi_{\ell}(t, y(t)) w(t, y(t)) h_{\Psi_{t}} (y(t))  \, \di t \, \di \pi (y)  \bigg| 
\\
\leq &
\bigg| \int_{\Gamma} \int_{0}^{T} \varphi_{\ell}(t, y(t)) w(t, y(t)) h_{\Psi_{t}} (y(t)) \, \di t \, \di \overline{\pi}_{k}^{N}(y) 
\\
& 
\qquad - \int_{\Gamma} \int_{0}^{T} \varphi_{\ell}(t, y(t)) w(t, y(t)) h_{\Psi_{t}} (y(t)) \, \di t \, \di \overline\pi_{k} (y)  \bigg| 
\\
&
\quad + \bigg| \int_{\Gamma} \int_{0}^{T} \varphi_{\ell}(t, y(t)) w(t, y(t)) h_{\Psi_{t}} (y(t)) \, \di t \, \di \overline{\pi}_{k}(y) 
\\
&
\qquad - \int_{\Gamma} \int_{0}^{T} \varphi_{\ell}(t, y(t)) w(t, y(t)) h_{\Psi_{t}} (y(t)) \, \di t \, \di \pi (y)  \bigg| 
\\
\leq &
\frac{1}{k} + \bigg| \int_{\Gamma} \int_{0}^{T} \varphi_{\ell}(t, y(t)) w(t, y(t)) h_{\Psi_{t}} (y(t)) \, \di t \, \di \overline{\pi}_{k}(y) 
\\
&
\qquad - \int_{\Gamma} \int_{0}^{T} \varphi_{\ell}(t, y(t)) w(t, y(t)) h_{\Psi_{t}} (y(t))  \, \di t \, \di \pi (y)  \bigg| 
\\
= &
\frac{1}{k} + \bigg| \frac{1}{\pi(\Delta_{k})} \int_{\Delta_{k}} \int_{0}^{T} \varphi_{\ell}(t, y(t)) w(t, y(t)) h_{\Psi_{t}} (y(t))  \, \di t \, \di \pi (y) \\
&
\qquad -  \int_{\Gamma} \int_{0}^{T} \varphi_{\ell}(t, y(t)) w(t, y(t)) h_{\Psi_{t}} (y(t))  \, \di t \, \di \pi (y) \bigg|\,.
\end{align*}
Passing to the limit as~$N\to\infty$ in the previous inequality we get by the boundedness of~$w$, $\varphi_{\ell}$, and~$h$, that
\begin{displaymath}
\lim_{N\to\infty}\! \int_{\Gamma} \! \int_{0}^{T} \!\!  \varphi_{\ell}(t, y(t)) w(t, y(t)) h_{\Psi_{t}} (y(t)) \, \di t \, \di \pi_{N}(y) =\! \int_{\Gamma} \!  \int_{0}^{T}  \!\! \varphi_{\ell}(t, y(t)) w(t, y(t)) h_{\Psi_{t}} (y(t)) \, \di t \, \di \pi(y)\,. 
\end{displaymath}
Therefore, passing to the limsup as~$N\to\infty$ in~\eqref{e:limsup7.1} we obtain
\begin{align*}
\limsup_{N\to\infty}\bigg| \int_{0}^{T} \int_{\overline{Y}} &  \varphi(t, y) \, \di (\eeta^{N} - \nnu) (t, y) \bigg| \leq C\varepsilon\,.
\end{align*}
By the arbitrariness of~$\varepsilon$ and~$\varphi$ we infer that~$\eeta^{N} \rightharpoonup \nnu$ weakly$^*$ in $\mathcal{M}([0,T]\times  \overline{Y}; \R^{d})$.


To prove the first convergence in \eqref{4.35}, we need to estimate, for every $\varphi \in C_{c}([0,T]\times  \overline{Y}; \R^{d})$, and using the definition of~$\nnu^{N}$, of~$\eeta^{N}$, and of the controls~$\uu^{N}$,
\begin{equation}
\label{e:limsup131}
\begin{split}
&\, \bigg|\int_{0}^{T}  \int_{ \overline{Y}} \varphi (t, y) \, \di \nnu^{N}(t, y) - \int_{0}^{T} \int_{ \overline{Y}} \varphi(t, y) \, \di \eeta^{N}(t, y) \bigg|
\\
= &\,  \bigg| \frac{1}{N} \sum_{i=1}^{N} \int_{0}^{T} \big( \varphi(t, y_{i}(t)) h_{\Psi^{N}_{t}} (y_{i}(t)) - \varphi(t, z_{i}(t)) h_{\Psi_{t}} (z_{i}(t)) \big) \, w(t, z_{i}(t)) \, \di t \bigg|
\\
\leq &\,
  \frac{1}{N} \sum_{i=1}^{N} \int_{0}^{T} \big| \varphi(t, y_{i}(t)) - \varphi(t, z_{i}(t)) \big| \cdot  \big| h_{\Psi^{N}_{t}} (y_{i}(t)) w(t, z_{i}(t)) \big| \, \di t 
\\
&
\, +  \frac{1}{N} \sum_{i=1}^{N} \int_{0}^{T} \big| h_{\Psi^{N}_{t}} (y_{i}(t)) - h_{\Psi_{t}} (z_{i}(t)) \big| \cdot \big|  \varphi(t, z_{i}(t)) w(t, z_{i}(t)) \big| \, \di t \,.
\end{split}
\end{equation}
In order to continue in~\eqref{e:limsup131} let us fix a modulus of continuity~$\omega_{\varphi}$ for the function $\varphi$. Notice that, without loss of generality, we may assume~$\omega_{\varphi}$ to be increasing and concave. Thus, by~$(h_{1})$, $(h_{2})$, by the fact that~$w(t, z_{i}(t)) \in K$ and $\overline{y}_{i}, z_{i} \in \B^{Y}_{R}$ for every~$t \in [0,T]$ and every $i=1, \ldots, N$, and by the inequalities~\eqref{e:limsup8-1}, \eqref{e:conv-min-3}, we can further estimate~\eqref{e:limsup131} with
\begin{equation*}
\begin{split}
 \bigg|& \int_{0}^{T}  \int_{ \overline{Y}} \varphi (t, y) \, \di \nnu^{N}(t, y) - \int_{0}^{T} \int_{ \overline{Y}} \varphi(t, y) \, \di \eeta^{N}(t, y) \bigg|
\\
& \leq   \frac{C}{N} \sum_{i=1}^{N} \int_{0}^{T} \! \big[ \omega_{\varphi} ( \| z_{i}(t) - \overline{y}_{i}(t) \|_{\overline{Y}} ) 
+ 
\| \varphi\|_{C([0,T] \times \overline{Y})} 
\big(\| z_{i}(t) - \overline{y}_{i}(t) \|_{\overline{Y}} + W_{1}(\Psi_{t}, \Psi^{N}_{t}) \big)\big]\, \di t 
\\
&
\leq C \int_{0}^{T} \omega_{\varphi} \bigg( \frac{1}{N} \sum_{i=1}^{N} \| z_{i}(t) - \overline{y}_{i}(t) \|_{\overline{Y}}\bigg) \di t 
\\
&
\qquad + C \| \varphi\|_{C([0,T] \times \overline{Y})} \bigg(e^{L_{R}T} T W_{1}( \Psi^{N}_{0}, \Lambda^{N}_{0}) + ( 1 +  L_{R} T e^{L_{R}T})  \int_{0}^{T} W_{1} ( \Psi^{N}_{t} , \Lambda^{N}_{t})  \, \di t \bigg) 
\\
&
\leq  C T \omega_{\varphi} \bigg( e^{L_{R}T} W_{1}( \Psi^{N}_{0}, \Lambda^{N}_{0}) +  L_{R}e^{L_{R}T}  \int_{0}^{T} W_{1} ( \Psi^{N}_{t} , \Lambda^{N}_{t})  \, \di t \bigg)
\\
&
\qquad + C \| \varphi\|_{C([0,T] \times \overline{Y})} \bigg(e^{L_{R}T} T W_{1}( \Psi^{N}_{0}, \Lambda^{N}_{0}) + ( 1 +  L_{R} T e^{L_{R}T})  \int_{0}^{T} W_{1} ( \Psi^{N}_{t} , \Lambda^{N}_{t})  \, \di t \bigg) \,,
\end{split}
\end{equation*}
where $C>0$ is a constant independent of~$N$. Therefore, by~\eqref{4.33} we conclude that
\begin{displaymath}
\lim_{N \to \infty}  \bigg| \int_{0}^{T}  \int_{ \overline{Y}} \varphi (t, y) \, \di \nnu^{N}(t, y) - \int_{0}^{T} \int_{ \overline{Y}} \varphi(t, y) \, \di \eeta^{N}(t, y) \bigg| = 0\,,
\end{displaymath}
which yields the first convergence in~\eqref{4.35}. 

Finally, we prove \eqref{gamma_limit}. As already observed, $(\y^{N}, \uu^{N}) \in \mathcal{S}(\y^{N}_{0})$ by construction, so that
\begin{equation}
\label{e:limsup13}
\E^{\y^{N}_{0}}_{N} ( \y^{N}, \uu^{N}) = \frac{1}{N} \sum_{i=1}^{N} \fint_{0}^{T} \cL_{N} ( y_{i}(t) , \Psi^{N}_{t}) \, \di t + \frac{1}{N} \sum_{i=1}^{N} \fint_{0}^{T} \phi(u_{i}(t)) \, \di t \,.
\end{equation}
Since $\spt (\Psi^{N}_{t}) , \, \spt( \Psi_{t}) \subseteq \B^{Y}_{R}$ for every $t \in [0,T]$ and, by~$(\cL_{1})$ and~$(\cL_{2})$,~$\cL$ is continuous and~$\cL_{N}$ $\sP_{1}$-converges to~$\cL$ uniformly on compact sets, we have that
\begin{equation}
\label{e:limsup14}
\begin{split}
\lim_{N \to \infty}  \frac{1}{N} \sum_{i=1}^{N} \fint_{0}^{T} \cL_{N} ( y_{i}(t) , \Psi^{N}_{t}) \, \di t & = \lim_{N \to \infty}  \fint_{0}^{T} \int_{Y} \cL_{N} ( y , \Psi^{N}_{t}) \, \di \Psi^{N}_{t} \,  \di t 
\\
&
= \lim_{N \to \infty} \fint_{0}^{T} \int_{\B^{Y}_{R}} \cL_{N}( y, \Psi^{N}_{t}) \, \di \Psi^{N}_{t}\, \di t
\\
&
 =   \fint_{0}^{T} \int_{\B^{Y}_{R}} \cL ( y , \Psi_{t}) \, \di \Psi_{t} \,  \di t=  \fint_{0}^{T} \int_{Y} \cL ( y , \Psi_{t}) \, \di \Psi_{t} \,  \di t  \,. 
\end{split}
\end{equation}
As for the second term on the right-hand side of~\eqref{e:limsup13}, we recall that~$\uu^{N} = (\sigma^{N}_{\R^{d}})^{-1}(\overline{\uu}^{N})$ with $\overline{u}_{i}(t) = w(t, z_{i}(t))$ and that $\Lambda^{N}_{t} = ( \ev_{t})_{\#}\pi_{N}$, so that we can write
\begin{align*}
 \frac{1}{N} \sum_{i=1}^{N} \fint_{0}^{T} \phi(u_{i}(t)) \, \di t & = \frac{1}{N} \sum_{i=1}^{N} \fint_{0}^{T} \phi(\overline{u}_{i}(t)) \, \di t =  \fint_{0}^{T} \int_{Y} \phi(w(t, y)) \, \di \Lambda^{N}_{t} (y) \, \di t 
 \\
 &
 = \int_{\Gamma}  \fint_{0}^{T} \phi(w(t, y(t))) \, \di t \,  \di \pi_{N} (y) = \int_{\Gamma} \F(y) \, \di \pi_{N}(y) \,. \nonumber
\end{align*}
In view of~\eqref{e:limsup5}, we infer that
\begin{displaymath}
\lim_{N \to \infty}  \frac{1}{N} \sum_{i=1}^{N} \fint_{0}^{T} \phi(u_{i}(t)) \, \di t = \Phi_{\rm min} (  \Psi, \nnu)\,,
\end{displaymath}
which implies, together with~\eqref{e:limsup14}, that
\begin{displaymath}
\lim_{N \to \infty} \, \E^{\y^{N}_{0}}_{N} ( \y^{N}, \uu^{N}) = \E^{\widehat{\Psi}_{0}} ( \Psi, \nnu) \,,
\end{displaymath}
which is \eqref{gamma_limit}.
This concludes the proof of the theorem.
\end{proof}


\section{Numerical experiments}\label{s:appl}
In this section we consider specific applications of our model in the context of opinion dynamics. In Section~\ref{sec:LF}, we discuss the effects of controlling a single population of leaders. 
In Section~\ref{sec:2LG}, instead, two competing populations of leaders and a residual population of followers are considered, 
but the policy maker favors only one of the populations of leaders towards their goal.

In both cases, for the continuity equation \eqref{e:conteq-2-intro} we use a finite volume scheme with dimensional splitting for the state space discretization, following a similar approach to the one employed in~\cite{ABRS}. Introducing a suitable discretization of the density $\Psi^n_i=\Psi(t_n,y_i)$ on uniform grid with parameters $\Delta x,\Delta \lambda$ in the state space, and $\Delta t$ in time,
the resulting scheme reads
\begin{subequations}
	\begin{align*}
	\tilde \Psi^n_{i}&=\Psi^{n}_{i} + \frac{\Delta t}{\Delta \lambda}\left( {\mathcal T_{i+1/2}[\Psi^n](y_i)-\mathcal T_{i-1/2}[\Psi^n](y_i)}\right),\\
	\Psi^{n+1}_{i}&=\tilde \Psi^n_{i} + \frac{\Delta t}{\Delta x}\left( {\mathcal{V}_{i+1/2}[\Psi^n,w^n](y_i)-\mathcal{V}_{i-1/2}[\Psi^n,w^n](y_i)}\right),\
	\end{align*}
\end{subequations}
where $\mathcal T_{i\pm1/2},\mathcal V_{i\pm1/2}$ are suitable discretizations of the transition operator and the non-local velocity flux. Notice that the update of $\Psi$ follows a two-step approximation, first in $\lambda$ then in~$x$, of the continuity equation \eqref{e:conteq-2-intro} (see also \cite{AlmMorSol21} for a rigorous convergence result).

The realization of the control is approximated using a nonlinear Model Predictive Control (MPC) tecnique. Hence, an open-loop optimal control action is synthesized over a prediction horizon $[0,T_p]$, by solving the optimal control problem \eqref{e:limit-cost-intro}--\eqref{e:conteq-2-intro}.
 Having prescribed the system dynamics and the running cost, this optimization problem depends on the initial state and the horizon $T_p$ only. The control $w^*$, which is obtained for the whole horizon $[0,T_p]$, is implemented over a possibily shorter control horizon $[0,T_c]$. At $t=T_c$ the initial state of the system is re-initialized to $\Psi(T_c)$ and the optimization is repeated. In this setting, to comply with an efficient solution of the dynamics, we perform the MPC optimization selecting  $T_p=T_c=\Delta t$. This choice of the horizons correponds to a instantaneous relaxation towards the target state. For further discussion on MPC literature we refer to \cite{albi2018selective,degond2017meanfield,grune2017nonlinear} and references therein.

\subsection{A leader-follower dynamics}\label{sec:LF}
In this setting, the set $U$ consists of two elements, that is $U:=\{F, L\}$ and is endowed with a two-valued distance
\begin{equation*}
0=\mathrm{dist }(F,F)=\mathrm{dist }(L,L)\,,\quad 1=\mathrm{dist }(F,L)=\mathrm{dist }(L,F)\,.
\end{equation*}
The space $\cP_1(\{F,L\})$ is identified with the interval $[0,1]$; accordingly, in the discrete model, $\lambda_i$ is a scalar value describing the probability of the $i$-th particle of being a follower. 

In order to tune the influence of the control, the simplest possible choice is to fix a function $h_\Psi(x,\lambda)=h(\lambda)$ in \eqref{e:Cauchy} for a suitable bounded non-negative Lipschitz function $h\colon[0,1]\to\R$. In the applications, where the policy maker aims at controlling only the population of leaders, the ideal function~$h$ should be non-increasing and equal to zero when $\lambda$ is close to~$1$. As shown in Proposition~\ref{p:N-particles}, if the cost function $\phi$ satisfies $\{\phi=0\}=\{0\}$, the optimal control will steer only agents with small~$\lambda$.

It is natural to partition the total population into leaders and followers, according to~$\lambda$. Given $\Psi\in\cP(\R^d\times[0,1])$, and for a fixed Lipschitz function $g\colon [0,1]\to [0,1]$, we define the followers and leaders distributions as
\begin{equation}\label{I7}
\mu^F_\Psi(B):=\int_{B \times[0,1]} g(\lambda) \,\mathrm{d}\Psi(x, \lambda),\qquad \mu^L_\Psi(B):=\int_{B \times [0,1]} (1-g(\lambda)) \,\mathrm{d}\Psi(x, \lambda),
\end{equation}
for each Borel set $B \subset \mathbb{R}^d$. 
In particular, the sum $\mu_\Psi^F(B)+\mu_\Psi^L(B)$ coincides with the first marginal of $\Psi$ and therefore it counts the total population contained in $B$.
In the discrete setting, the leaders and followers distributions in \eqref{I7} are given by
\begin{equation}\label{I7-discrete}
\mu_{\Psi^N}^F(B)=\frac1N\sum_{i\,:\,x_i\in B}g(\lambda_i),\quad
 \mu_{\Psi^N}^L(B)=\frac1N\#\{i:x_i\in B\}-\mu_{\Psi^N}^F(B)=\frac1N\sum_{i\,:\,x_i\in B}(1-g(\lambda_i)).
\end{equation}
A typical choice for $g$ is any Lipschitz regularization of the indicator function of the set $\{\lambda \ge m\}$, with $m\ge 0$ a small given threshold. 
Doing so amounts to classifying agents with small~$\lambda$ (and therefore high influence) as leaders and the remaining ones as followers. However, different and softer choices for $g$ are possible. For instance, the choice $g(\lambda)=\lambda$ allows one to measure the average degree of influence of an agent sitting in the region $B$ on the remaining ones.

It is a common feature of many-particle models to assume that 
each agent experiences a velocity which combines the action of the overall followers and leaders distribution. 
Hence, these velocities are an \emph{average} velocity of the system, weighted by the probability $\lambda$ that an agent located at $x$ has of being a 
follower, and have the general form
\begin{equation}\label{e:vel-ex}
\begin{split}
v_{\Psi}(x, \lambda)& =   g_1(\lambda)\int_{\R^d\times[0,1]} \big[K^{FF}(x-x')g_2(\lambda')+K^{LF}(x-x')(1-g_2(\lambda'))\big]\,\de \Psi(x',\lambda')\cr
&+(1-g_1(\lambda))\int_{\R^d\times[0,1]} \big[K^{FL}(x-x')g_2(\lambda')+ K^{LL}(x-x')(1-g_2(\lambda'))\,\de \Psi(x',\lambda'),
\end{split}
\end{equation}
where the functions $g_i\colon[0,1]\to\R$ (for $i=1,2$) are given Lipschitz continuous functions. Let us remark that the choice $g_1=g_2=g$, so that the velocities actually depend on $\Psi$ through the distributions $ \mu^F_\Psi$ and  $\mu^L_\Psi$, is quite plausible in this kind of modeling.
In the discrete setting, a velocity field of this kind reads as 
\begin{equation*}
\begin{split}
v_{\Psi^N}(x_i, \lambda_i) = &\,  g_1(\lambda_i)\bigg(\frac{1}{N}\sum_{j=1}^N K^{FF}(x_i-x_j)g_2(\lambda_j)+\frac{1}{N}\sum_{j=1}^N K^{LF}(x_i-x_j)(1-g_2(\lambda_j))\bigg)\\
&\, + (1-g_1(\lambda_i))\bigg(\frac{1}{N}\sum_{j=1}^N K^{FL}(x_i-x_j)g_2(\lambda_j)+\frac{1}{N}\sum_{j=1}^N K^{LL}(x_i-x_j)(1-g_2(\lambda_j))\bigg).
\end{split}
\end{equation*}

Similar principles can be used for defining the transitions rates. According to the identification of $\cP_1(\{F,L\})$ with $[0,1]$, the transition operator $\cT_\Psi(x,\lambda)$ will be identified with a scalar (see~\eqref{transition} below), instead of taking values in the two-dimensional space $\cF(\{F,L\})$. Indeed, in this case $(\cT_0)$ uniquely determines the second component of $\cT_\Psi$ once the first one is known. For instance, one can consider
\begin{equation}\label{transition}
\cT_\Psi(x,\lambda)=-\alpha_F (x,\Psi)g_3(\lambda) + \alpha_L(x,\Psi)(1-g_3(\lambda)),
\end{equation}
with $\alpha_\bullet$ having the typical form
$$\alpha_\bullet(x,\Psi)=\int_{\R^d\times[0,1]} H_\bullet(x-x')\ell_\bullet(\lambda')\,\de\Psi(x',\lambda'),\quad\text{for $\bullet\in\{F,L\}$,}$$
and where $g_3\colon[0,1]\to[0,1]$, $H_\bullet\colon\R^d\to\R_+$, and $\ell_\bullet\colon[0,1]\to[0,1]$ are given Lipschitz functions. Notice that condition $(\mathcal T_3)$ amounts to requiring that the conditions
\begin{equation}\label{e:alpha}
\alpha_\bullet\ge 0\,,\quad g_3(0)=0\,,\quad g_3(1)=1
\end{equation}
are satisfied (equivalently, the evolution of  $\lambda$ is confined into $[0,1]$). If one chooses $g_3(\lambda)=\lambda$, for fixed $x$ and $\Psi$ the evolution of $\lambda$ is governed by a linear master equation. Instead, for $g_3=g$, the switching rates $\alpha_F$ and $\alpha_L$ are activated depending on the population to which an agent belongs.
The function $H_\bullet$ can be used to localize the effect of the overall distribution on the transition rates; within this model, an agent sitting at $x$ is able to interact only with agents in a small neighborhood around $x$.
Similarly, with a proper choice of $\ell_\bullet$, one can tune the influence of the surrounding agents according to their probability of belonging to the populations of followers or leaders. The choice $\ell_F=1-\ell_L=g$ corresponds to having rates which depend on~$\Psi$ through the distributions $\mu^F_\Psi$ and  $\mu^L_\Psi$. Let us however stress that, in general, also with these choices it is not possible to decouple equation \eqref{e:S2} into a system of equations for  $\mu^F_\Psi$ and  $\mu^L_\Psi$, which, on the contrary, can only be reconstructed after solving for~$\Psi$ first. Some particular cases where this is instead possible are discussed in \cite[Proposition 4.8]{MorSol19}.

With the arguments of \cite[Section~4]{MorSol19}, one can see that choices of $v_\Psi$ and $\cT_\Psi$ made in \eqref{e:vel-ex} and \eqref{transition} fit in our general framework. Let us remark that in \cite[Section~4]{MorSol19}, only the case $g(\lambda)=g_i(\lambda)=\lambda$, $i=1,2,3$ was discussed, but the adaption to the current, more general situation, is straightforward.


A typical Lagrangian that we may consider should penalize the distance of the leaders from a desired goal. This may be encoded by a function of the form
\begin{equation}
\label{e:lag-1}
\cL_1(x,\lambda)=\theta(\lambda)|x-\bar x|^2,
\end{equation}
where $\bar x\in\R^d$ is the position of the desired goal and $\theta\colon[0,1]\to[0,1]$ is zero when $\lambda$ is above a given threshold (a possible choice is even $\theta(\lambda)=1-g(\lambda)$).
Moreover, a competing effect, depending on the overall distribution of the population, can be taken into account: leaders should stay as close as possible to the population of followers, in order to influence their behavior.
This may be encoded by a function of the form
\begin{equation}
\label{e:lag-2}
\cL_2(x,\lambda,\Psi)=\theta(\lambda) \bigg|x-\fint_{\R^d} x'\,\de\mu_\Psi^F(x')\bigg|^2,
\end{equation}
which favors a leader agent to be close to the barycenter of the followers distribution. 
Notice that the function $\cL_2$ depends continuously on $\Psi$ as long as $\mu_\Psi^F(\R^d)>0$, which is always the case in practical situations.
Hence, the Lagrangian of the system is the sum
\begin{equation}
\label{e:lag-3}
\cL(x,\lambda,\Psi)\coloneq \alpha\cL_1(x,\lambda)+(1-\alpha)\cL_2(x,\lambda,\Psi)\,,
\end{equation}
for $\alpha\in[0,1]$ a given constant.

Finally, a very simple and natural family of cost functions is 
\begin{equation}
\label{e:cost-phi}
\phi_{p}(u) = \frac{\gamma}{p}| u|^{p}\,, \qquad \text{for $p \in (1, +\infty)$ and $u \in \R^{d}$.}
\end{equation}
In particular,~$\phi_{p}$ is strictly convex and $\{\phi_{p} = 0\} = \{0\}$, so that the conclusions of Proposition~\ref{p:N-particles} hold true in the case $h_{\Psi} = h$ mentioned above. Namely, the optimal control~$\uu \in L^{1}([0,T] ; (\R^{d})^{N})$ in the $N$-particle problem will actually act only on the population of leaders, while the evolution of the population of followers will be determined by the velocities and transitions rates detailed above.

%
%
%

\subsubsection{Test~1: Opinion dynamics with emerging leaders population} 
We study the setting proposed in \cite{BMPW2009,albi2014boltzmann} for opinion dynamics in presence of leaders influence, and we assume that $x\in[-1,1]$, where $\{\pm1\}$ identify two opposite opinions. 
The interaction field $v_{\Psi}$ \eqref{e:vel-ex} is characterized by bounded confidence kernels with the following structure
\begin{align}
K^{\star \bullet}(x-x') = \chi_\varepsilon(\{|x-x'|\leq \kappa_{\star\bullet}\}),\qquad\text{for $\star,\bullet\in\{F,L\}$,}
\end{align}
where $\eps\geq0$ is a regularization parameter for the characteristic function $\chi$ and $\kappa_{\bullet\star}$ represent the confidence intervals with the following numerical values,
\[
\kappa_{FF} = 0.25,\quad\kappa_{FL} = 0.5,\quad \kappa_{LF}= 0, \quad \kappa_{LL} = 0.2.
\]
The weighting functions $g_1,g_2$ are such that $ g_1(\lambda)\equiv g_2(\lambda)\equiv \ell(\lambda)$ with
\begin{align}
\ell(\lambda) = \dfrac{e^{C(\lambda-\bar \lambda)}}{1+e^{C(\lambda-\bar \lambda)}}, \quad C = 10^3,\quad \bar \lambda = 0.5.
\end{align}
The transition operator $\mathcal{T}_\Psi(x,\lambda)$ in \eqref{transition} is identified by the following quantities
\begin{align*}
\alpha_{F}(x,\Psi) = a_F\left(1-\mathcal D_L(x,\Psi)\right),\quad 
\alpha_{L}(x,\Psi) = a_L\left(1-\mathcal D_F(x,\Psi)\right),\quad
g_3(\lambda) = \lambda,
\end{align*}
where the functions $\mathcal D_F$ and $ \mathcal D_L$ represent the concentration of followers and leaders at position~$x$ and are defined by
\begin{align}\label{e:conc-ex1}
\mathcal D_\bullet(x,\Psi) = S_\bullet\int_{\R^d\times[0,1]}\exp\left\{{-\dfrac{(x-x')^2}{\sigma_\bullet^2}}\right\} G_\bullet(\lambda')\de\Psi(x',\lambda'),\qquad \bullet\in\left\{F,L\right\},
\end{align}
with $G_F(\lambda) = \ell(\lambda)$ and $G_L(\lambda) = 1-G_F(\lambda)$, and $S_\bullet$ normalization constants such that concentrations are bounded above by one, \emph{i.e.}, $\mathcal D_\bullet(x,\Psi) \in[0,1]$ to preserve the positivity of the rates~$\alpha_F$ and $\alpha_L$, and with the following parameters
\[
a_F=0.025,\quad a_L = 0.05,\quad \sigma_F=\sigma_L=0.1.
\]

Finally, the cost functional is defined by the Lagrangian defined in \eqref{e:lag-3}, which steers followers towards $\bar x=-0.5$ and keeps track of followers average position with $\alpha = 0.35$ and $\theta(\lambda) = 1-\ell(\lambda)$. We account for quadratic penalization of the control in \eqref{e:cost-phi} by choosing $\gamma = 2$.

In Figure~\ref{fig:init_data_Test1}, we report the choice of the initial data, and the marginals $\mu^{F}_\Psi(t,x)$, $\mu^{L}_\Psi(t,x)$ relative to the opinion space, and to the label space $\nu^F_\Psi(t,\lambda)$, $\nu^L_\Psi(t,\lambda)$. The structure of the initial data is a bimodal Gaussian distribution defined as follows
\[
\Psi_0(x,\lambda) \coloneqq C_0 \left(\exp\left\{-\frac{(x-x_F)^2}{\sigma_{x,F}^2}-\frac{(\lambda-\lambda_F)^2}{{\sigma_{\lambda,F}^2}}\right\}+\exp\left\{-\frac{(x-x_L)^2}{{\sigma_{x,L}^2}}-\frac{(\lambda-\lambda_L)^2}{{\sigma_{\lambda,L}^2}}\right\}\right),
\]
where ${\sigma_{\lambda,F}^2}=\sigma_{\lambda,L}^2=1/100$, ${\sigma_{x,L}^2}=1/50$, $\sigma_{x,F}^2 =1/30$, 
$\lambda_F=0.45$, $\lambda_L = -0.45$, $x_F=0.3$, $x_L=0.7$ and $C_0$ is normalizing constant.
\begin{figure}[!ht]
	\begin{center}
		\includegraphics[width=0.325\linewidth]{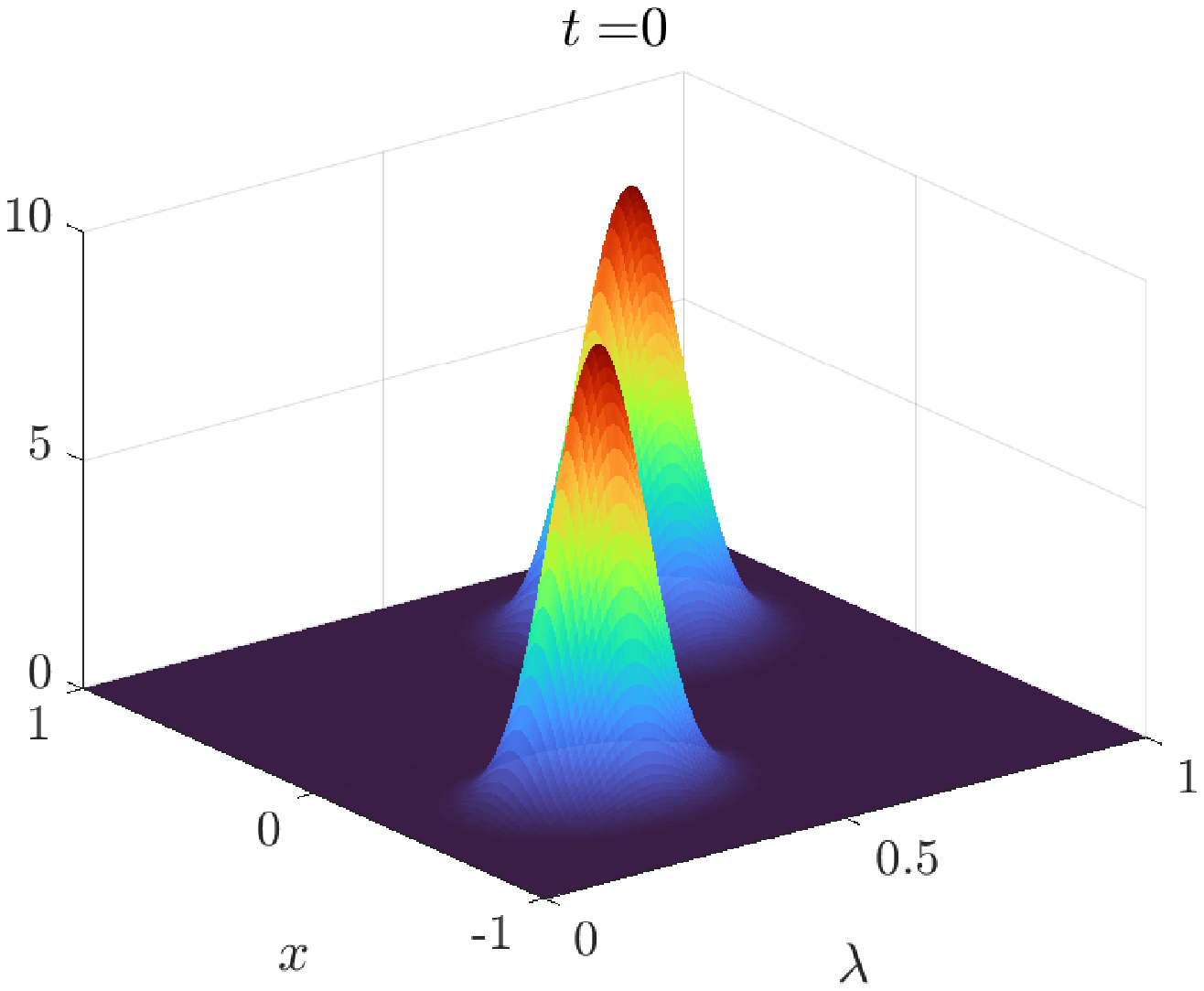}
		\includegraphics[width=0.325\linewidth]{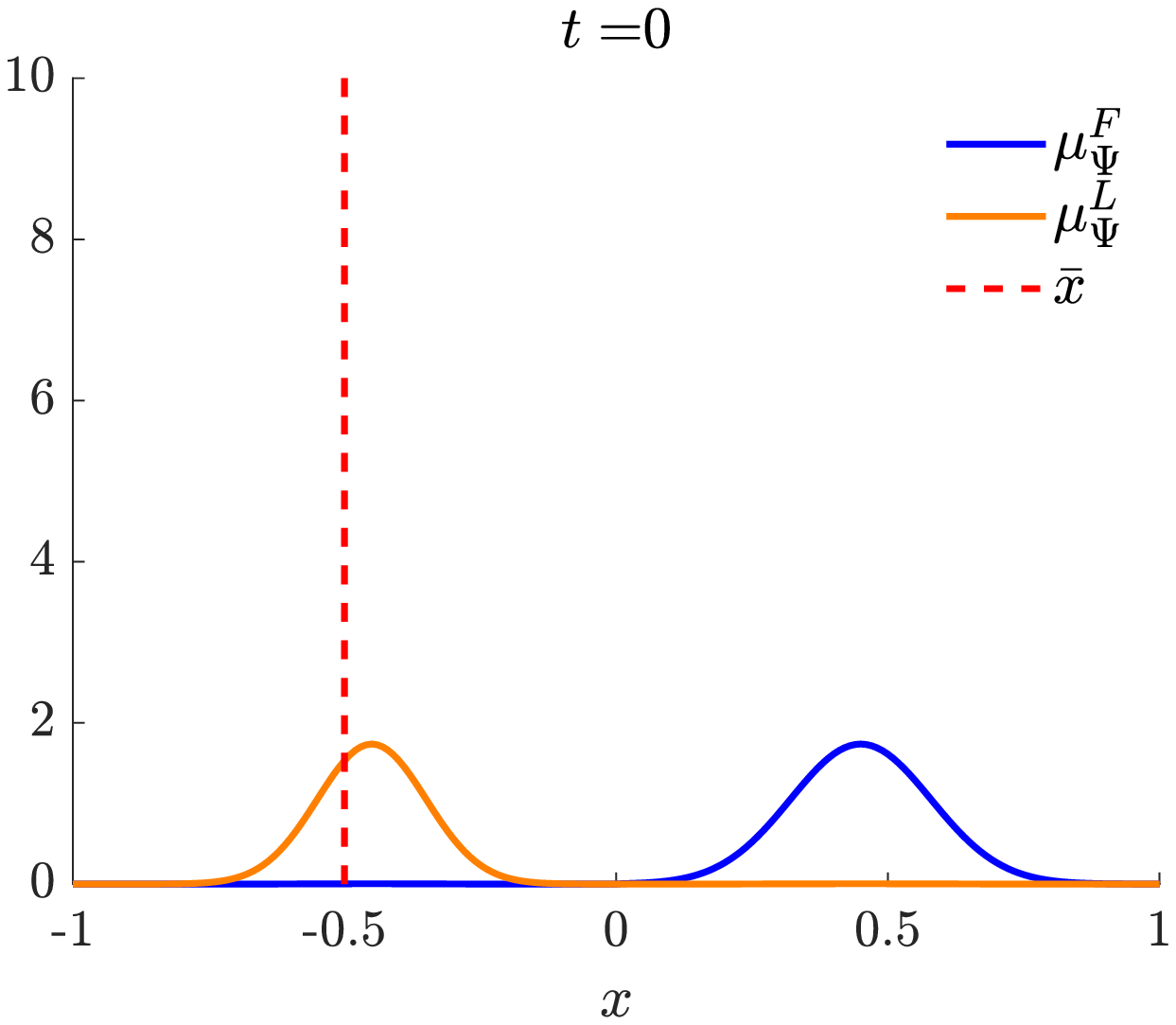}
		\includegraphics[width=0.325\linewidth]{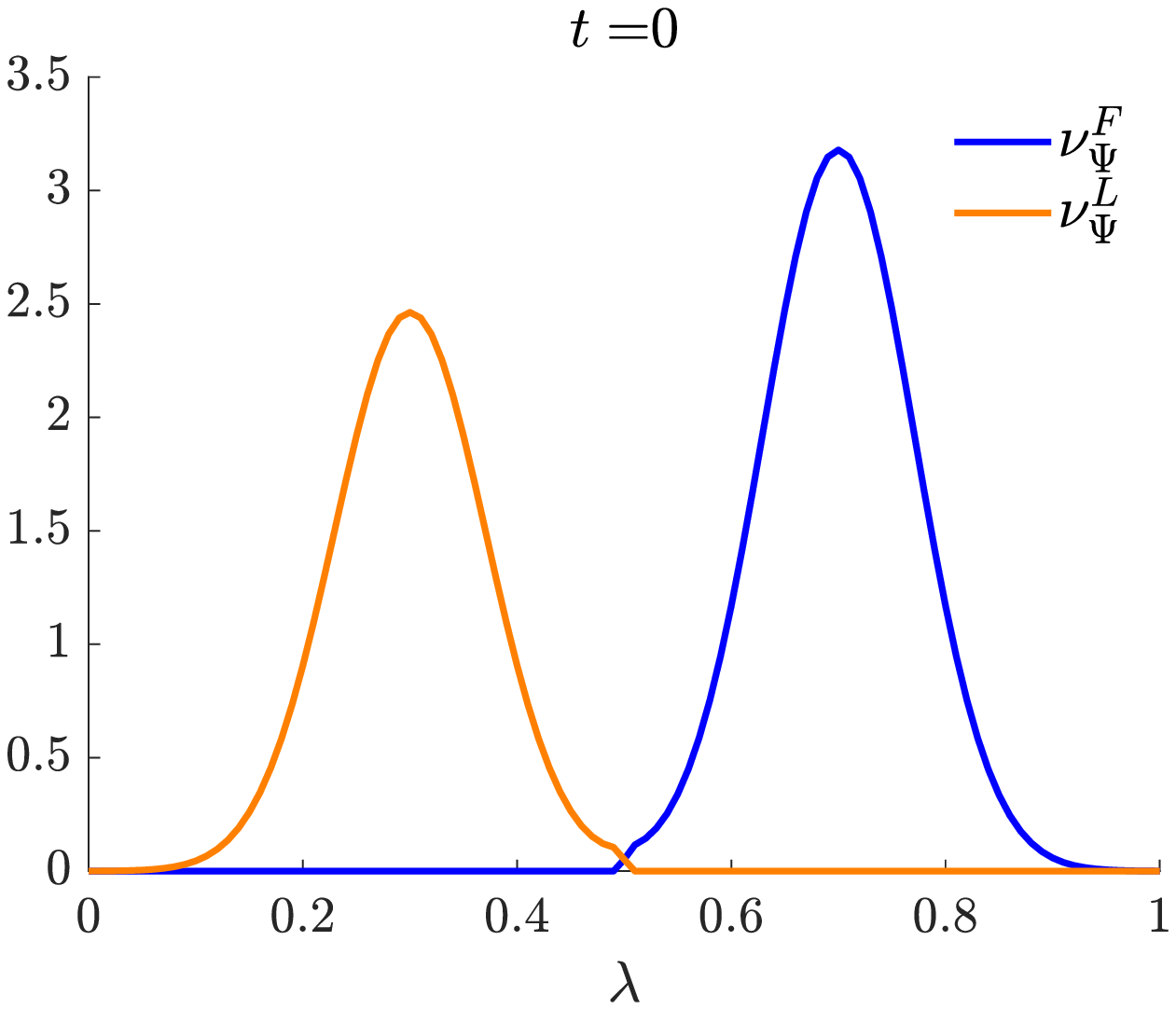}
		\caption{{\em Test 1.} Intial distribution $\Psi_0(x,\lambda)$ and the marginals associated with the opinion space $\mu^{F}_\Psi(t,x)$, $\mu^{L}_\Psi(t,x)$, and to the label space $\nu^F_\Psi(t,\lambda)$, $\nu^L_\Psi(t,\lambda)$. The red dashed line marks the target position $\bar x=-0.5$.} \label{fig:init_data_Test1}
	\end{center}
\end{figure}
Figure \ref{fig:evo_noctrl_Test1} reports from left to right four frames of the marginals up to time $t=10$, without control. We observe transition from leader to follower, and  viceversa, where, without the action of a policy maker, the initial clusters of opinions remain bounded away and no consensus is reached. In Figure~\ref{fig:evo_ctrl_Test1}, control is activated and in this case we observe the steering action of the leaders towards the target position $\bar x$.
\begin{figure}[!ht]
	\begin{center}
		\includegraphics[width=0.225\linewidth]{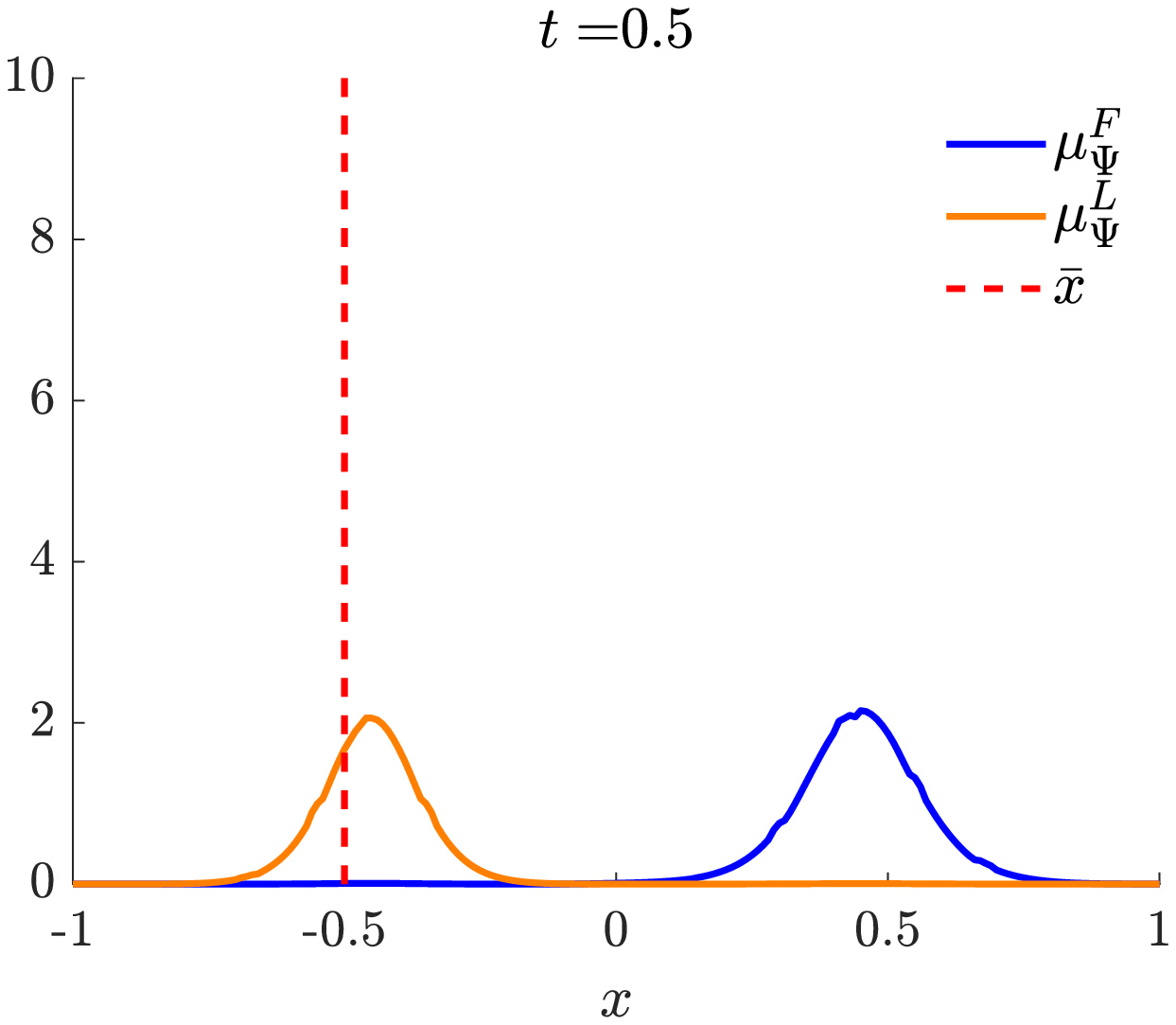}
		\includegraphics[width=0.225\linewidth]{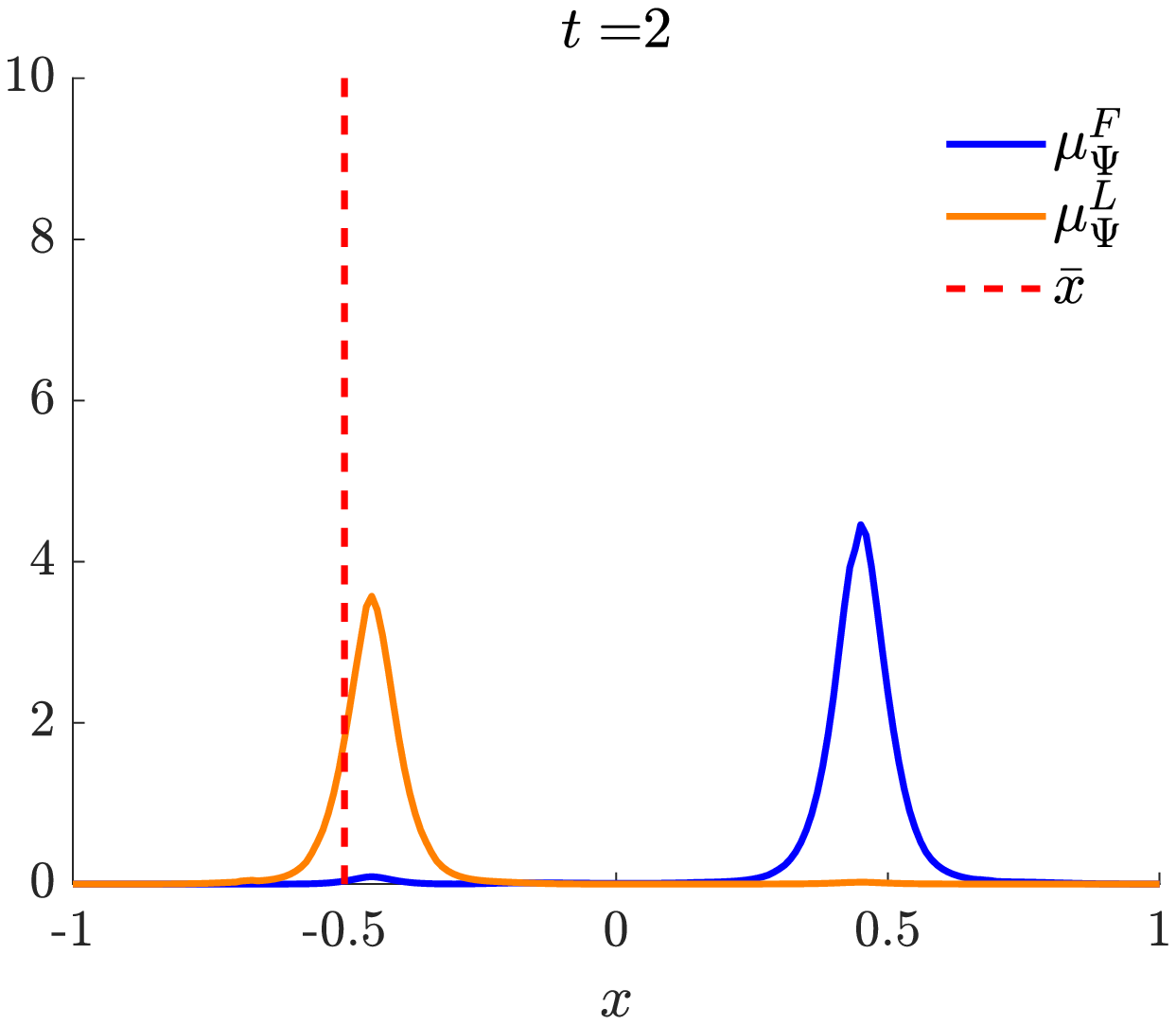}
		\includegraphics[width=0.225\linewidth]{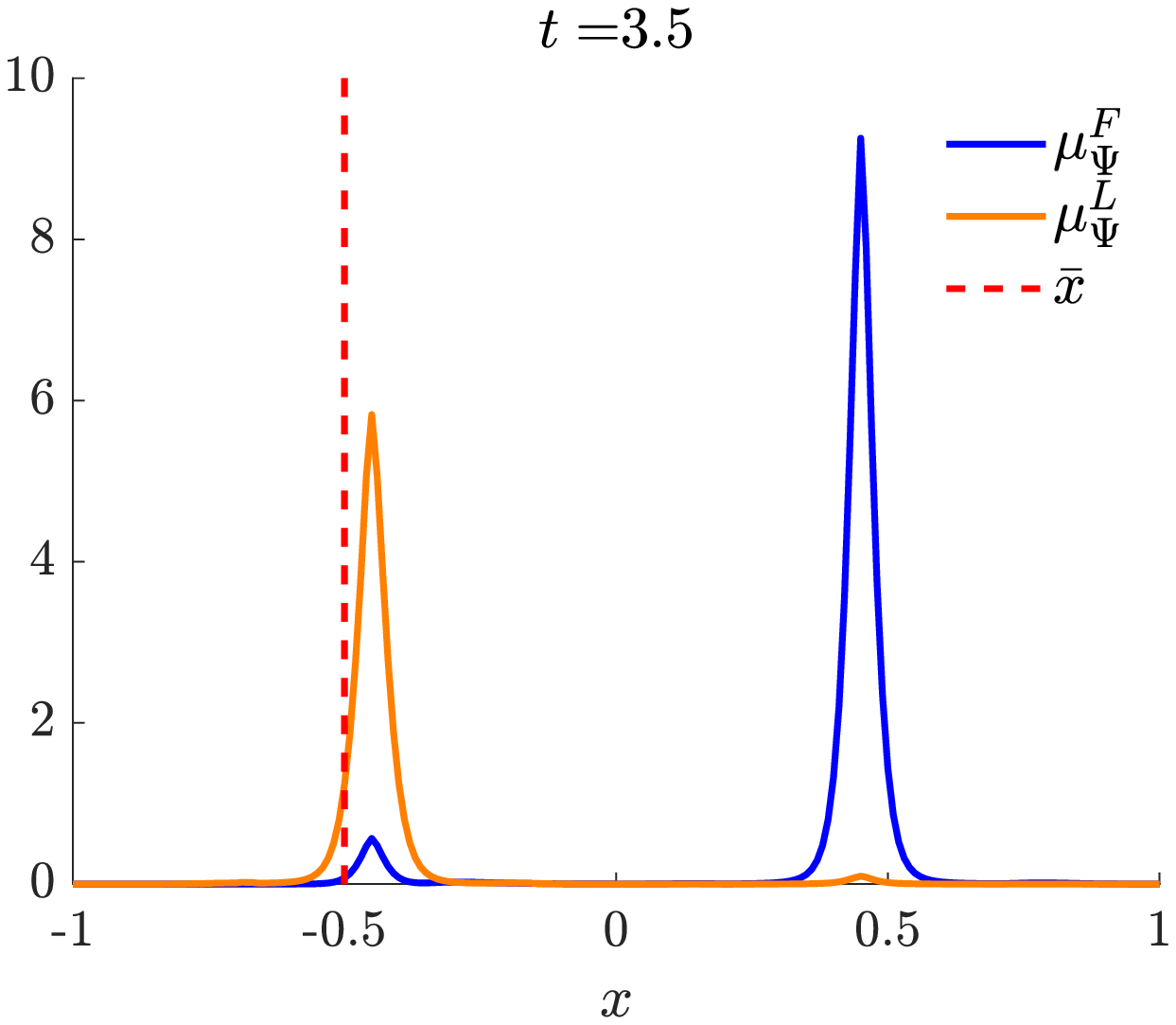}
		\includegraphics[width=0.225\linewidth]{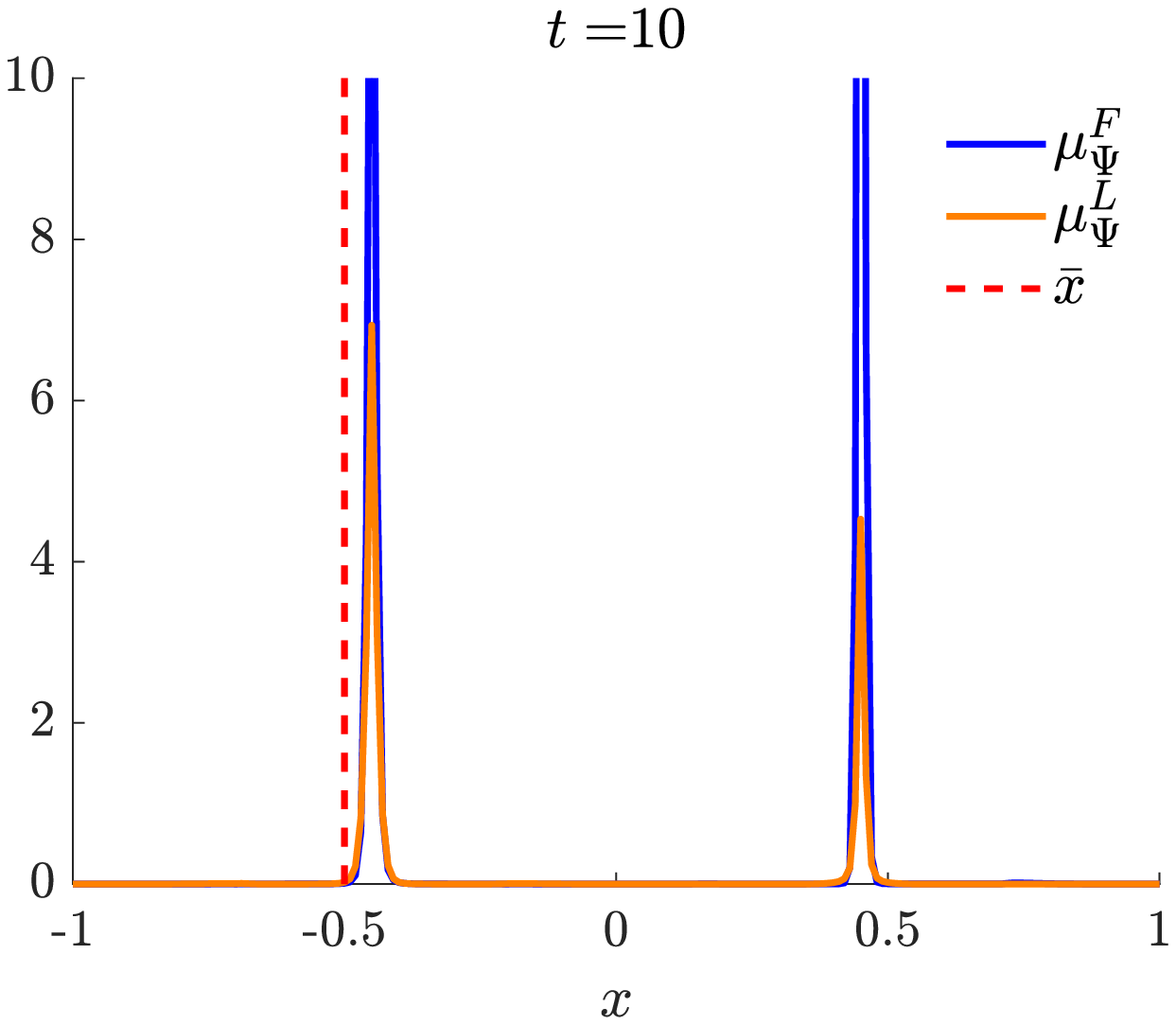}
		\\
		\includegraphics[width=0.225\linewidth]{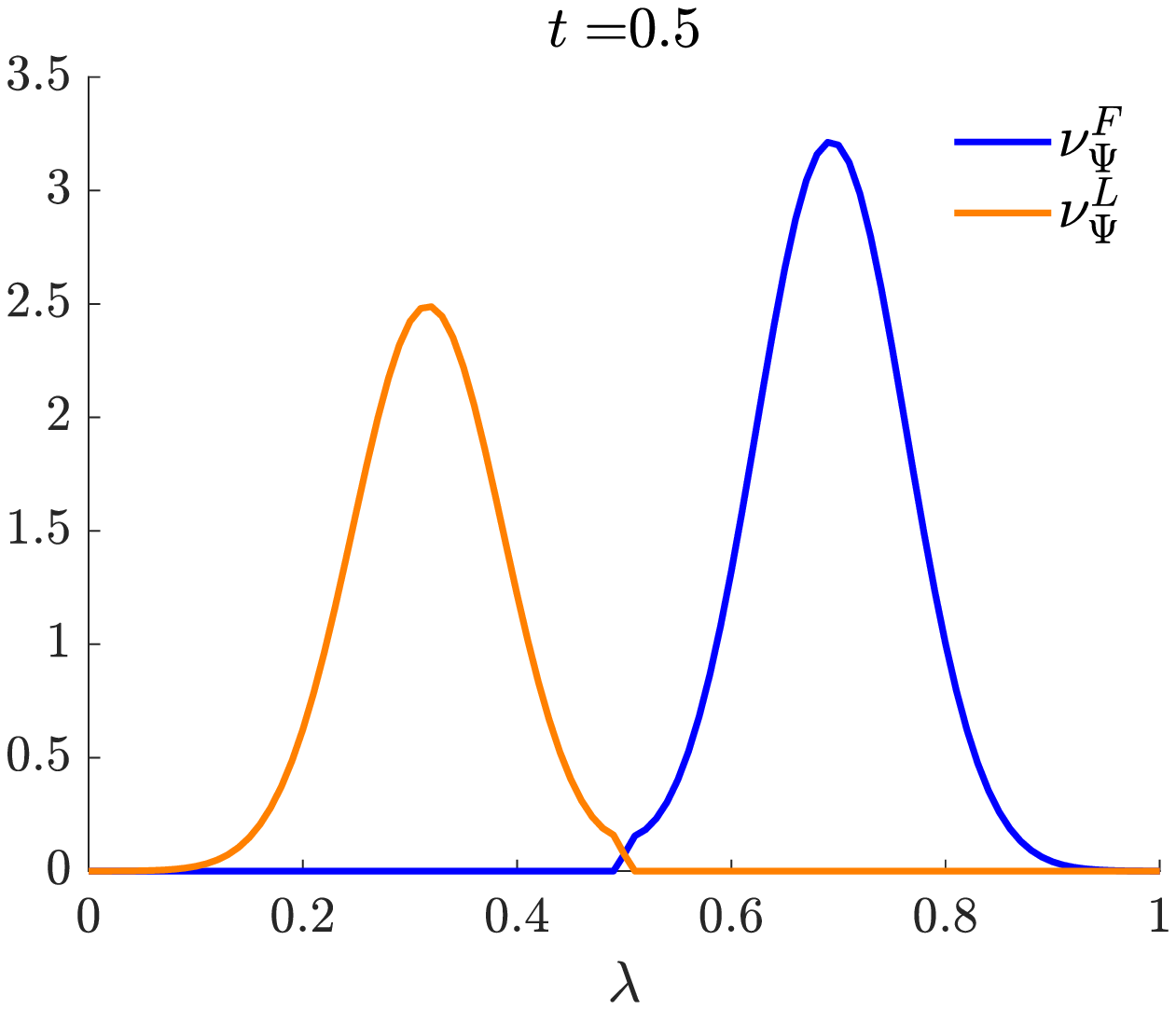}
		\includegraphics[width=0.225\linewidth]{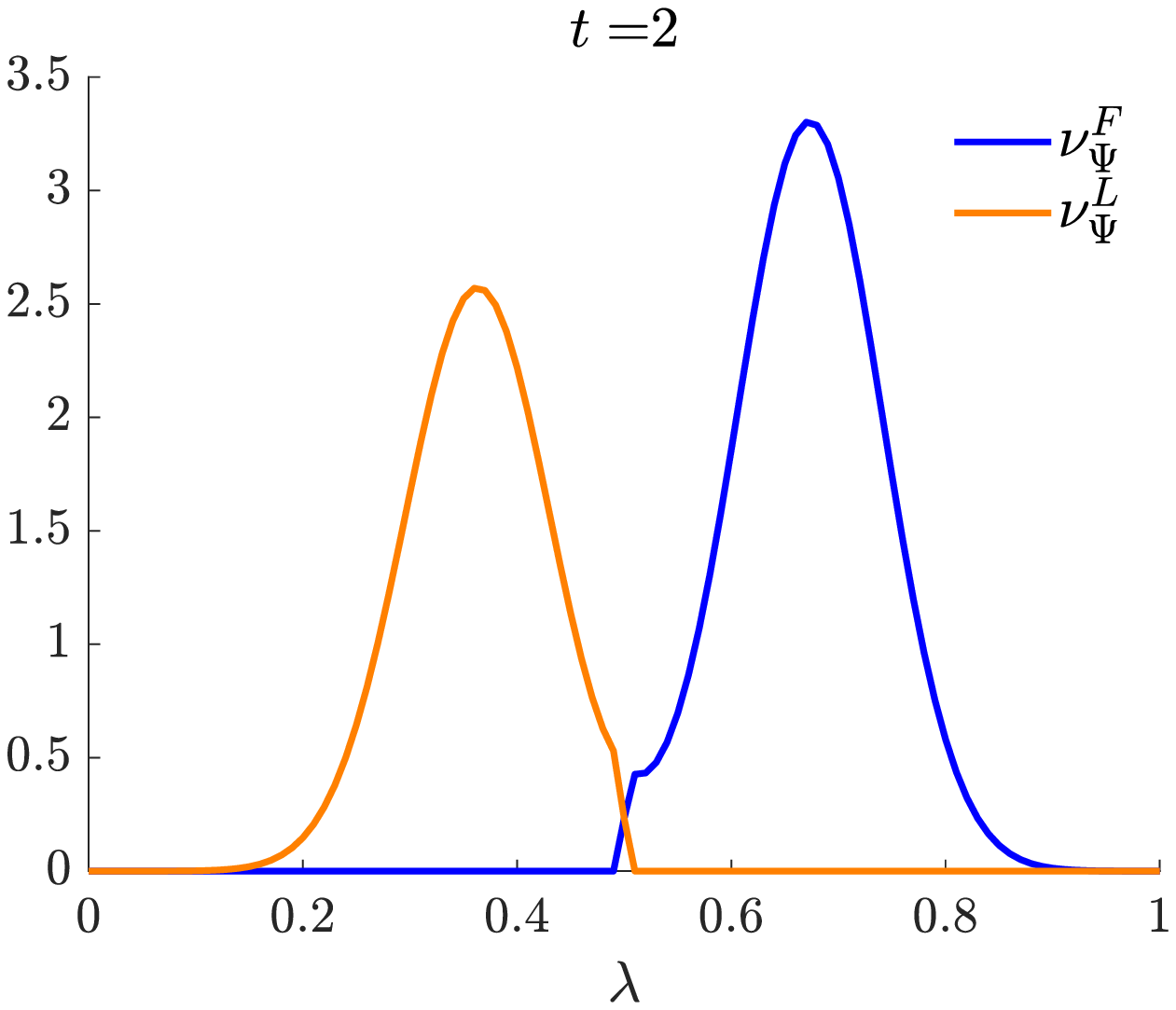}
		\includegraphics[width=0.225\linewidth]{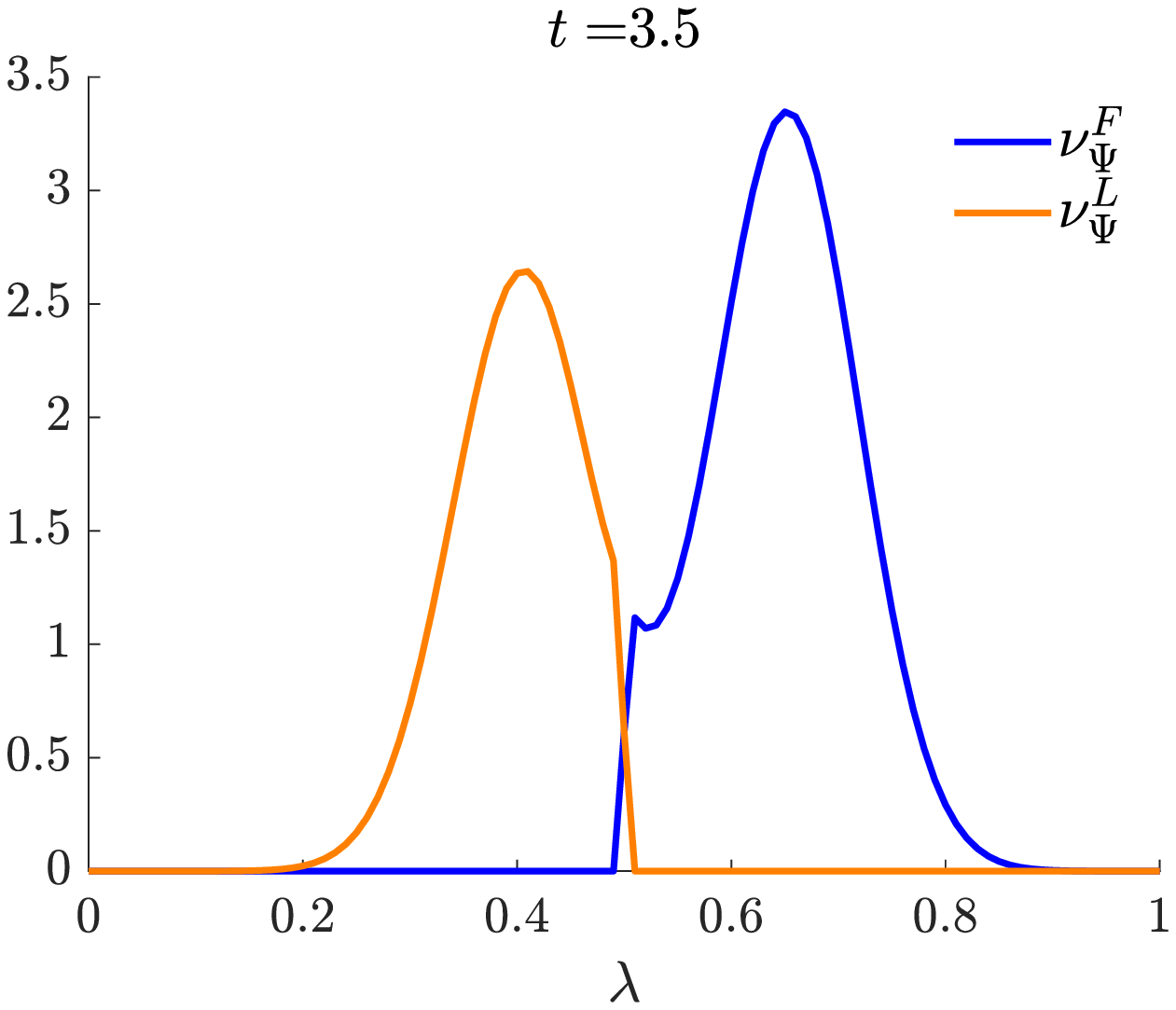}
		\includegraphics[width=0.225\linewidth]{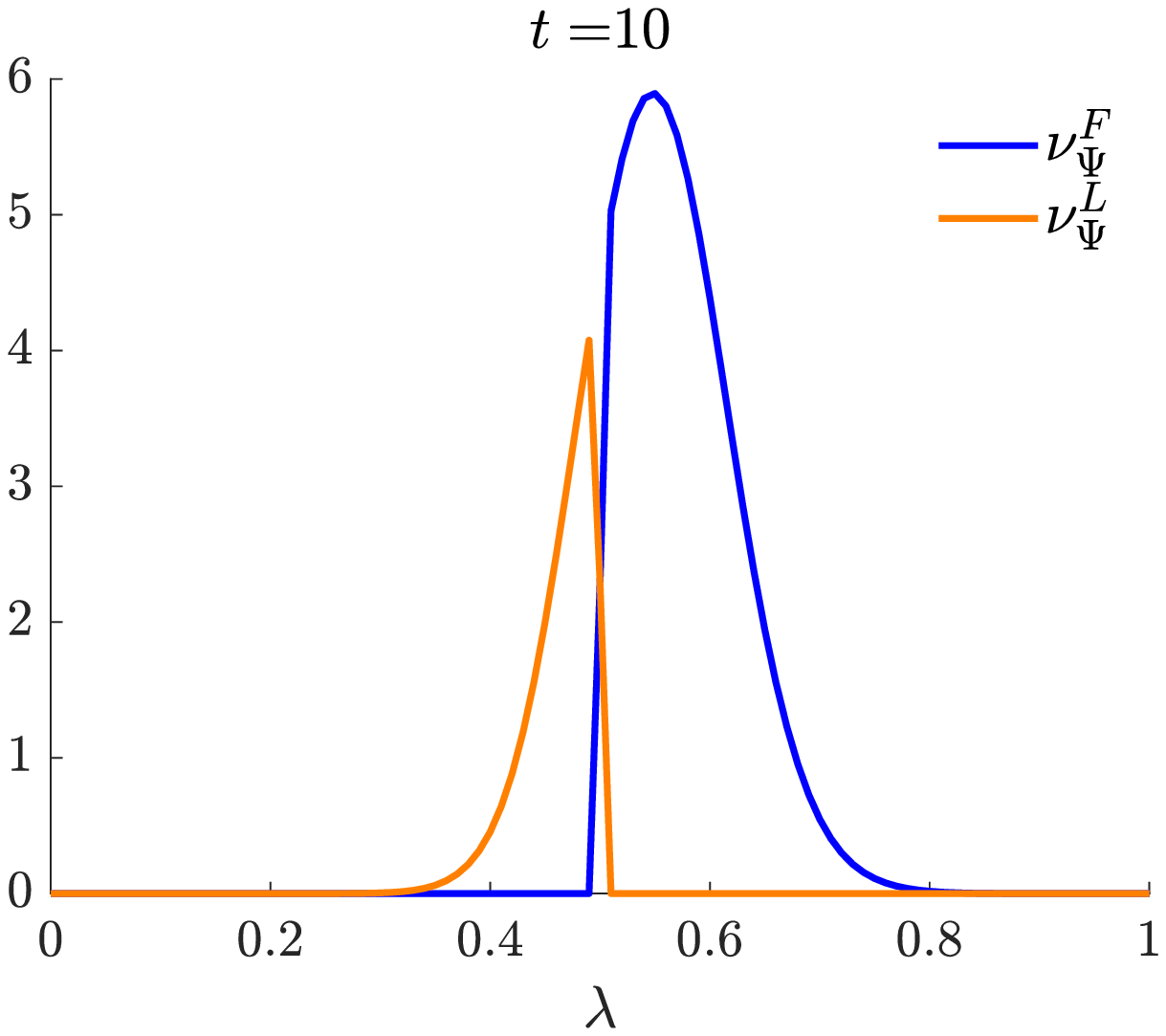}
		\caption{{\em Test 1.}~Evolution of the marginals without control. In the top row $\mu^{F}_\Psi(t,x)$ and $\mu^{L}_\Psi(t,x)$ are depicted; in the bottom row $\nu^F_\Psi(t,\lambda)$, $\nu^L_\Psi(t,\lambda)$ are depicted, both for time frames associated with $t=0.5,2, 3.5, 10$.} \label{fig:evo_noctrl_Test1}
	\end{center}
\end{figure}
\begin{figure}[!ht]
	\begin{center}
		\includegraphics[width=0.225\linewidth]{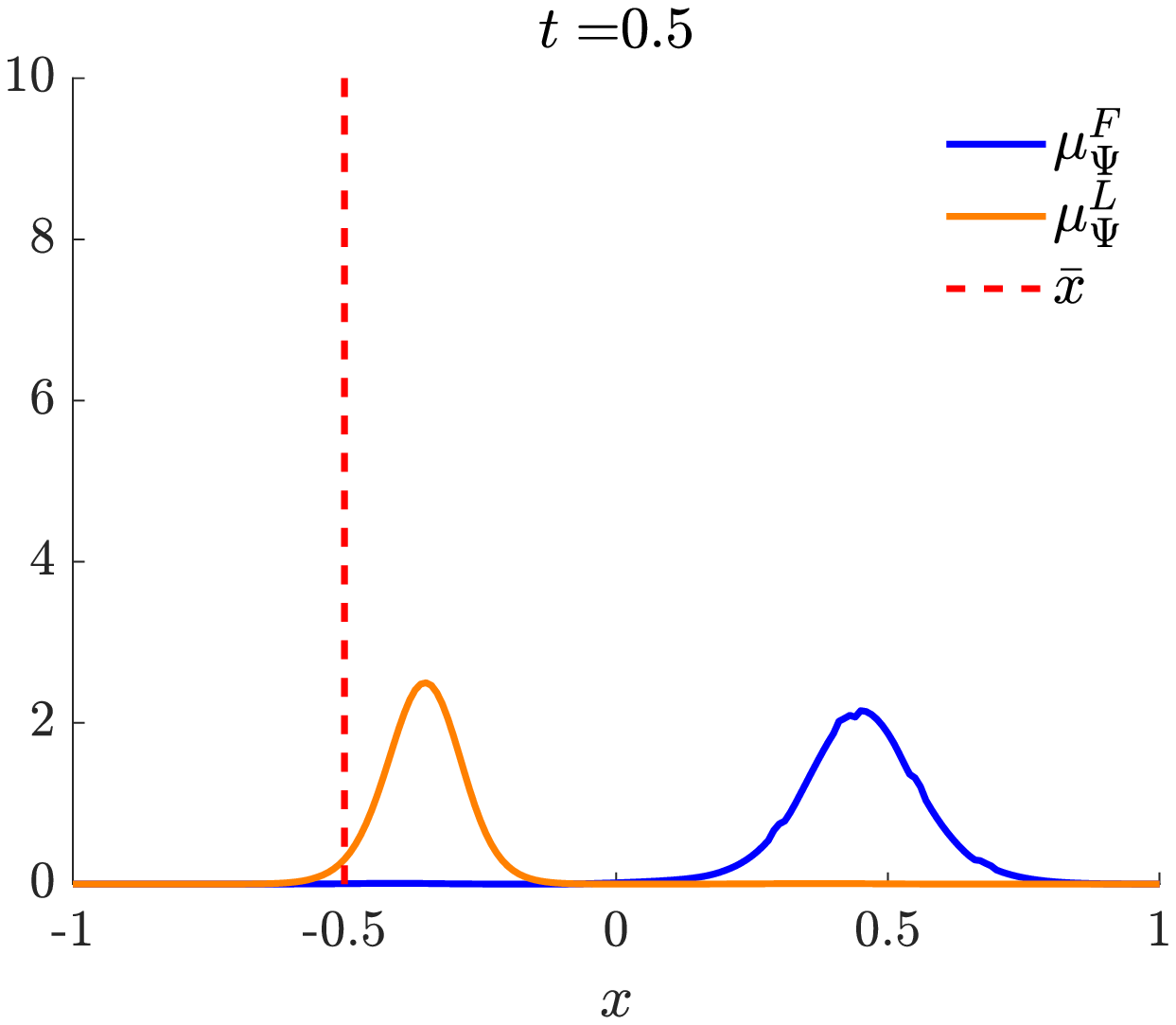}
		\includegraphics[width=0.225\linewidth]{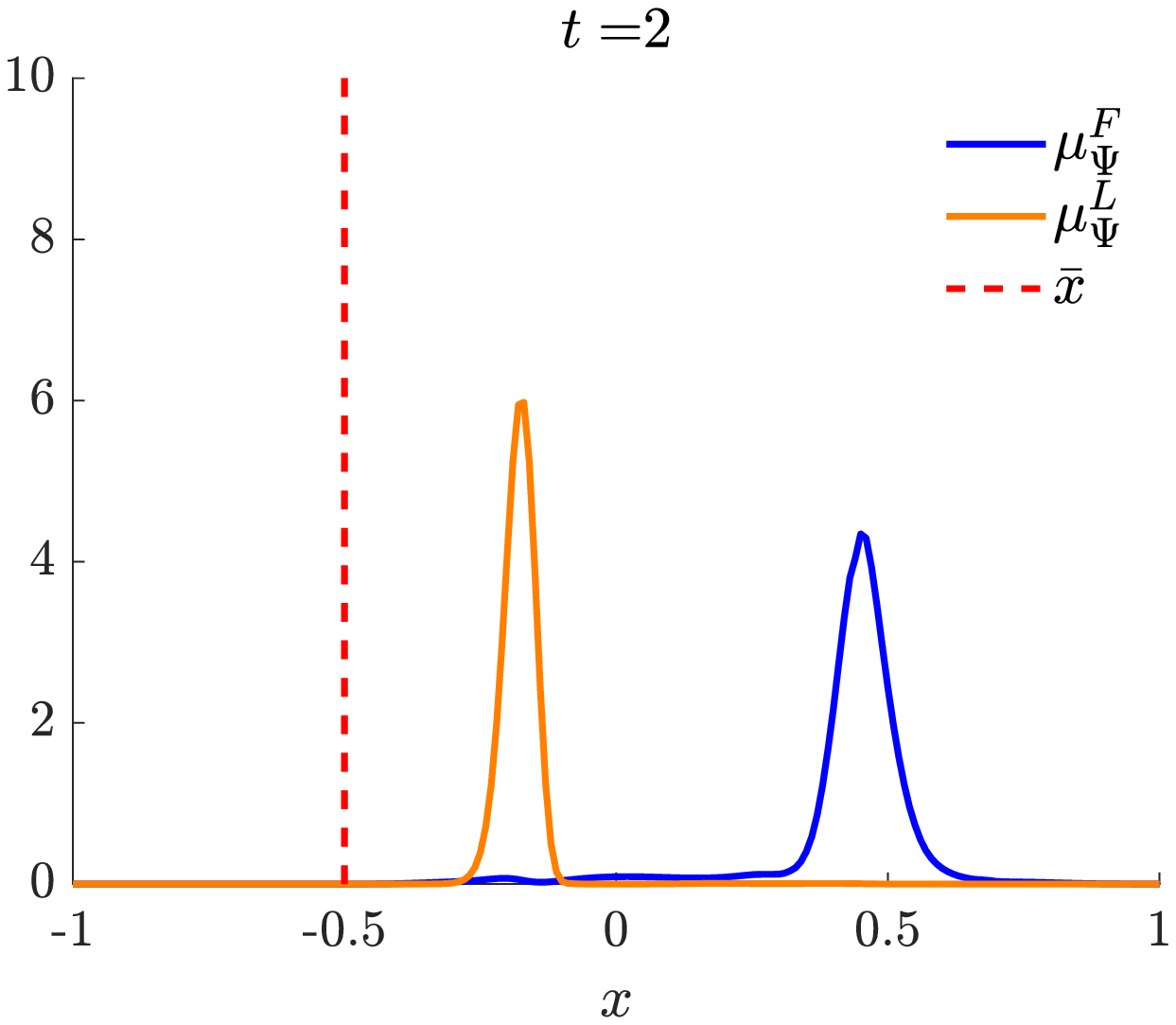}
		\includegraphics[width=0.225\linewidth]{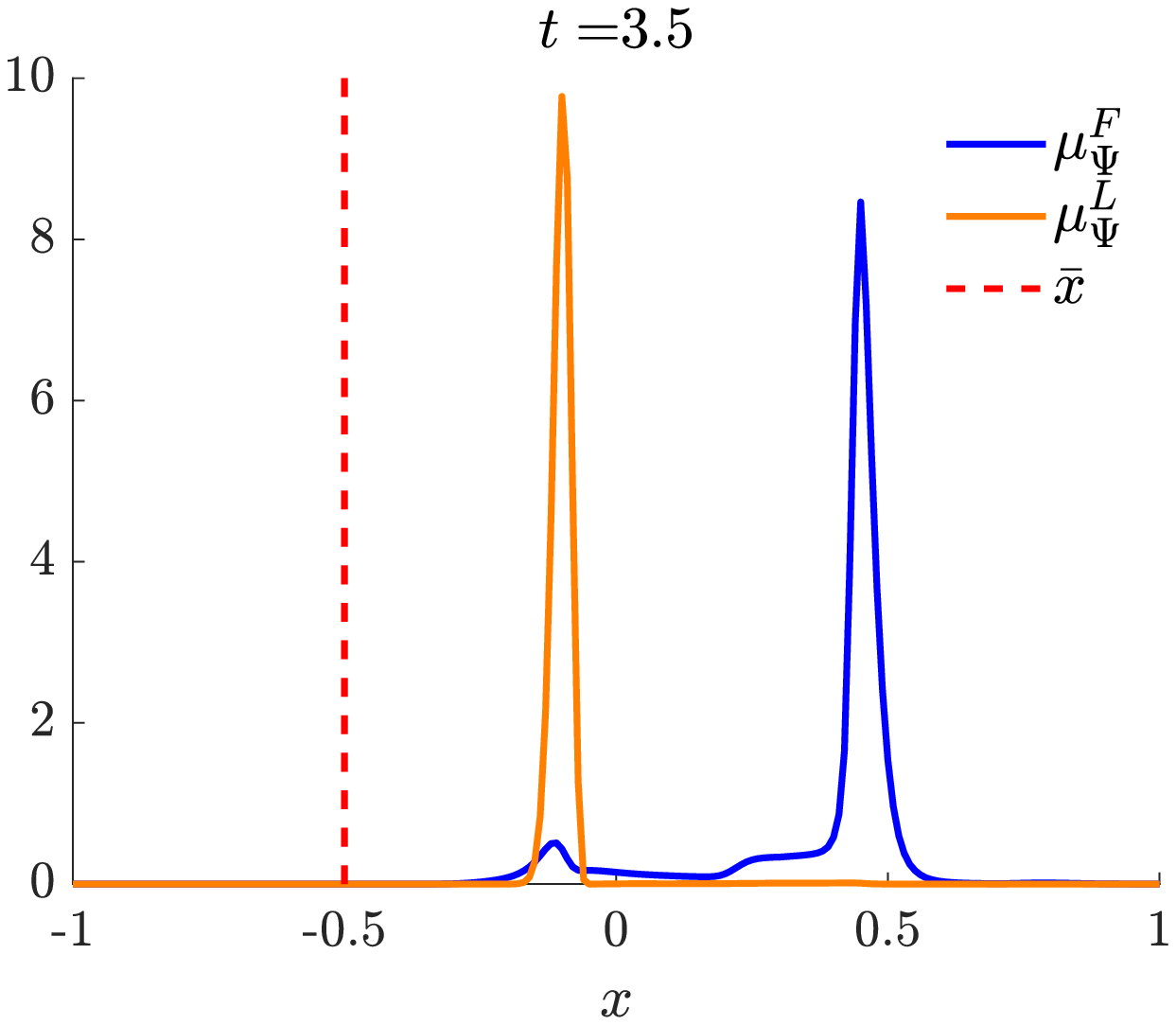}
		\includegraphics[width=0.225\linewidth]{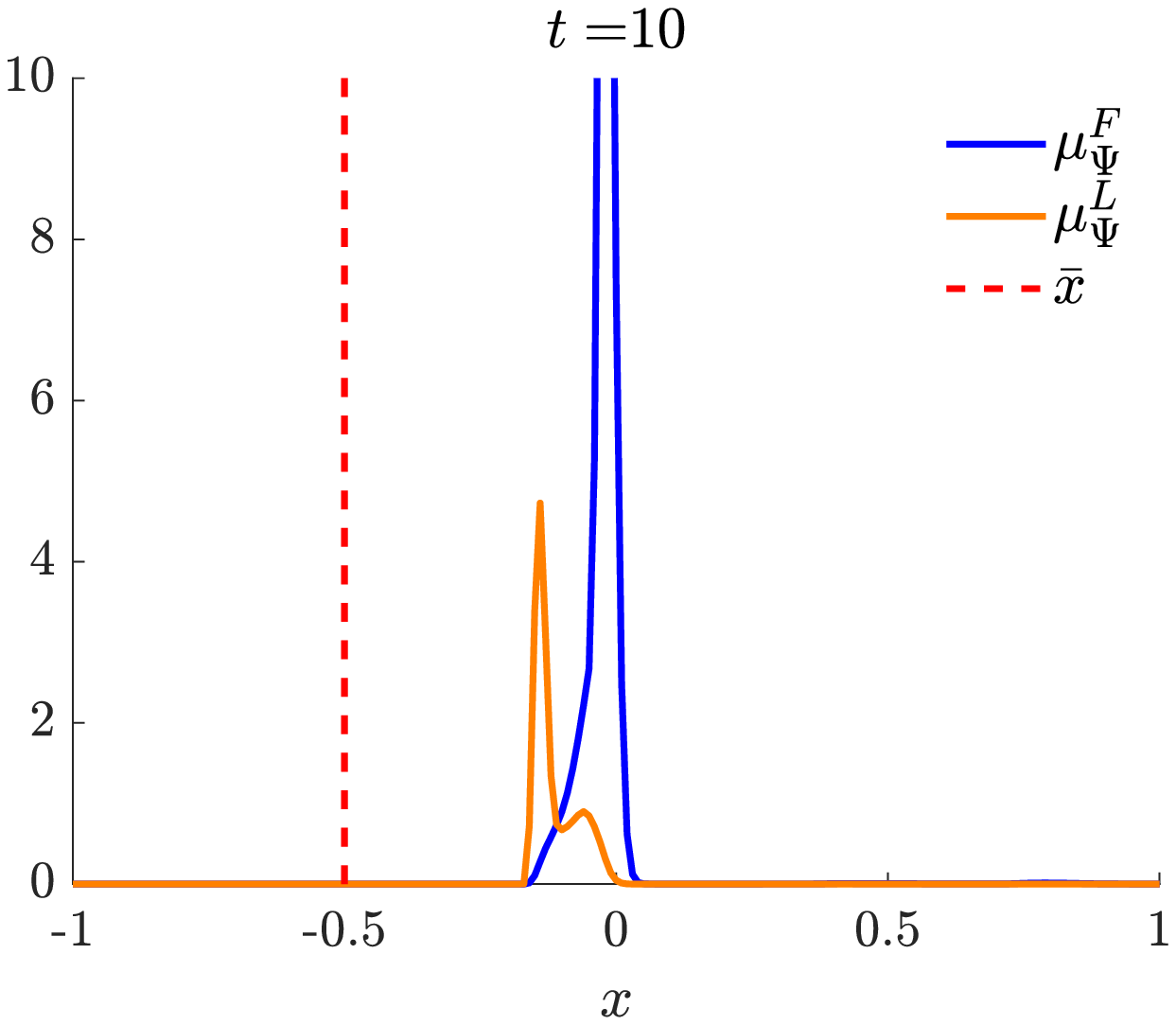}
		\\
		\includegraphics[width=0.225\linewidth]{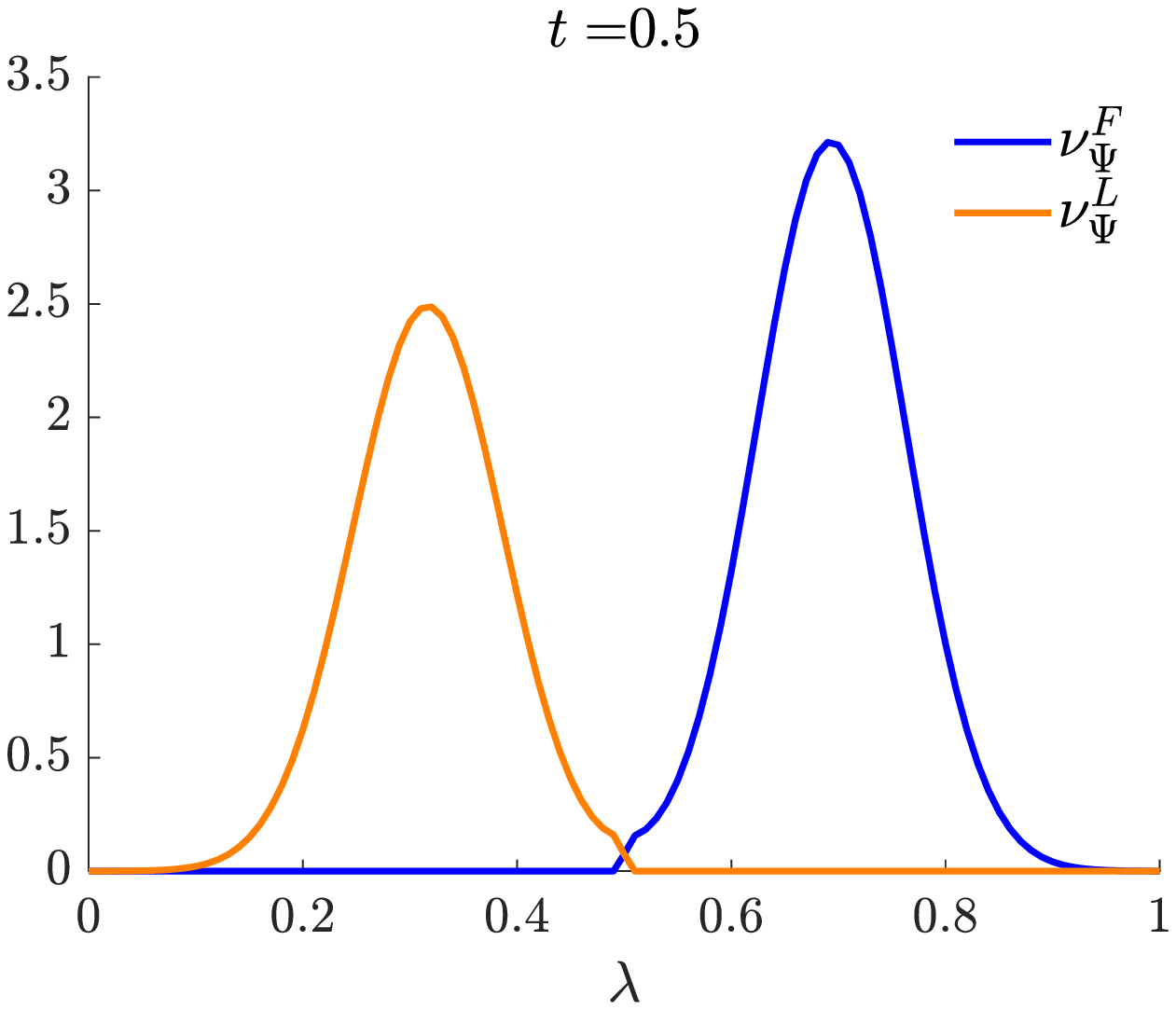}
		\includegraphics[width=0.225\linewidth]{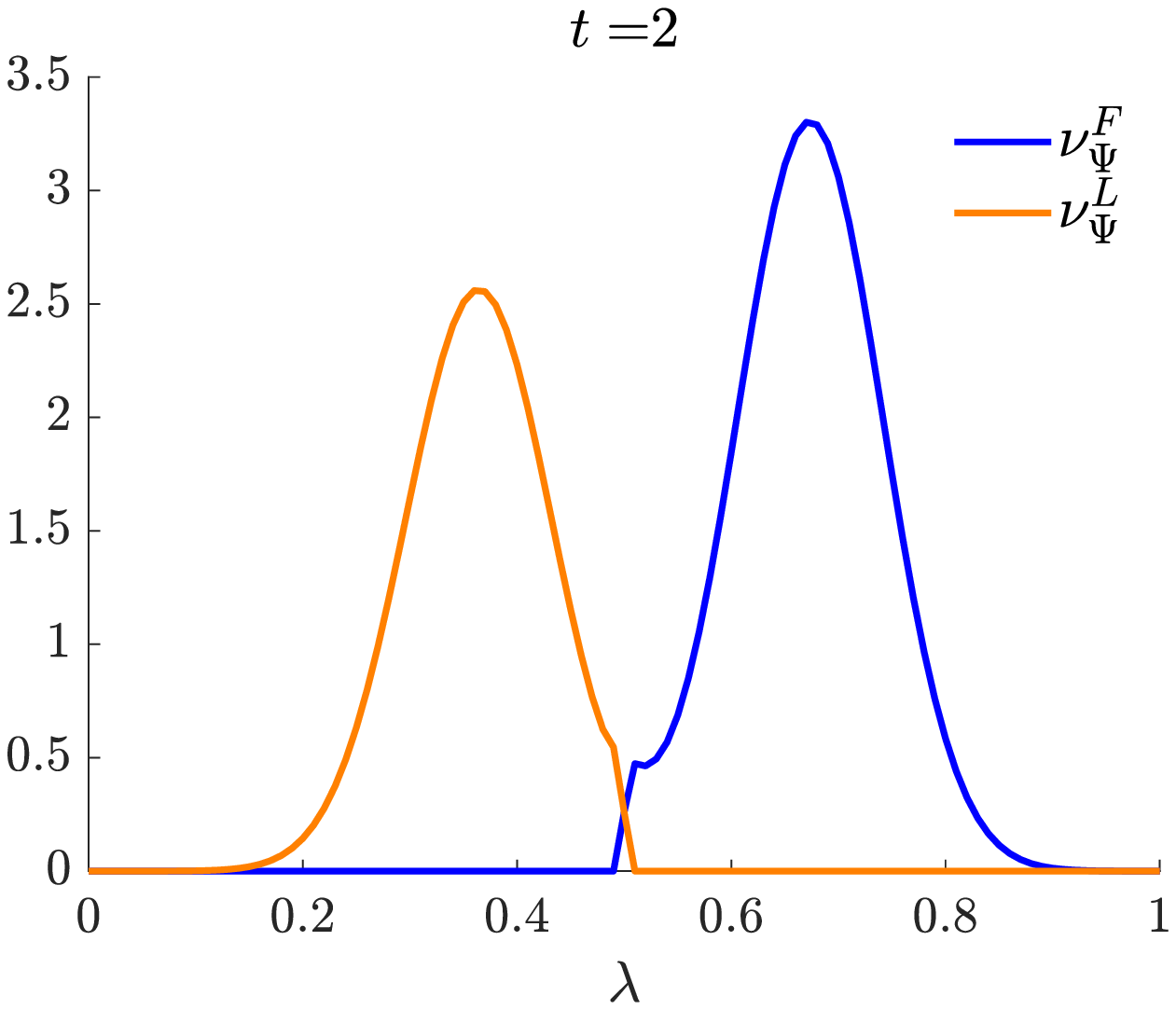}
		\includegraphics[width=0.225\linewidth]{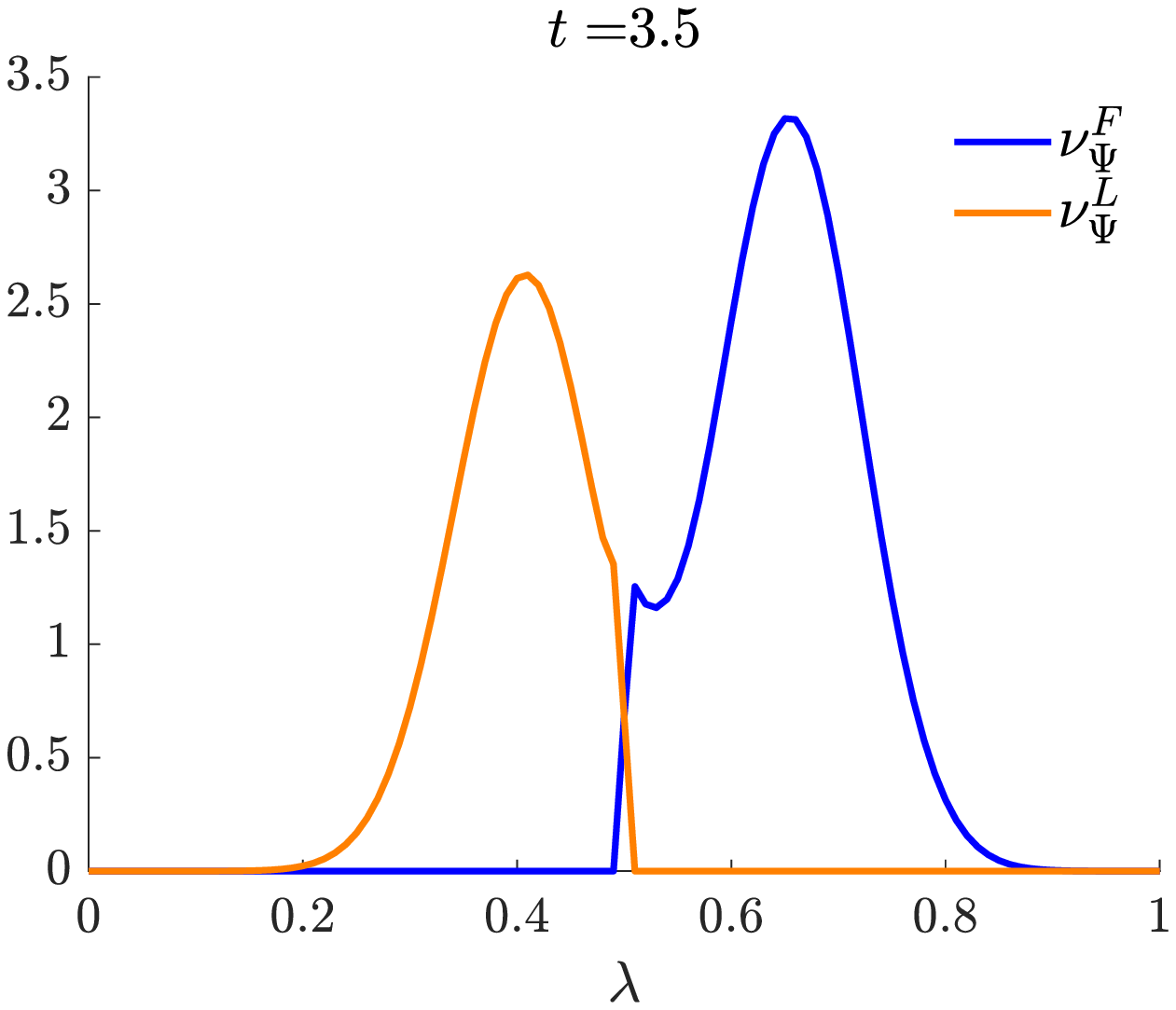}
		\includegraphics[width=0.225\linewidth]{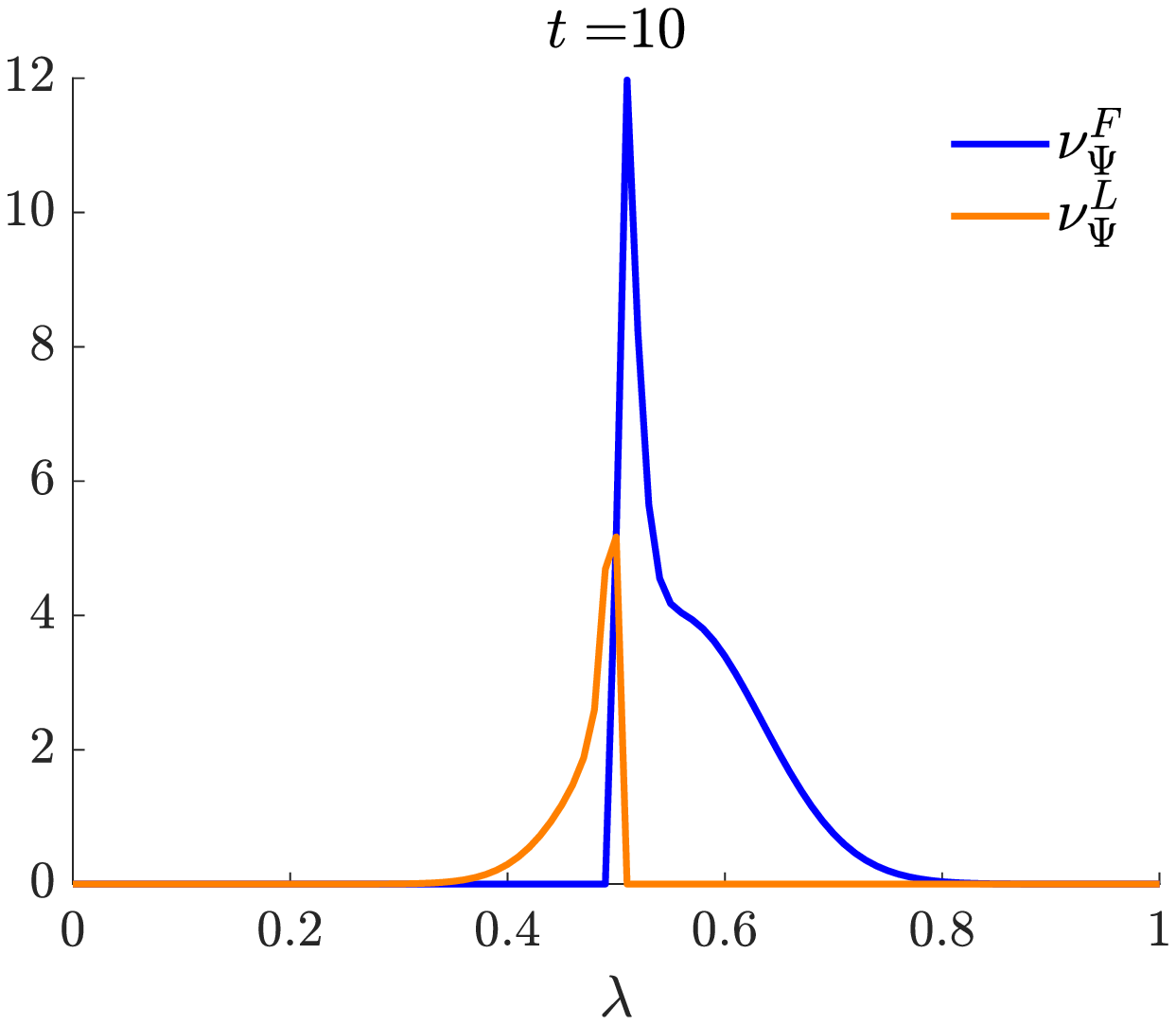}
		\caption{{\em Test 1.}~Evolution of the marginals with control. In the top row $\mu^{F}_\Psi(t,x)$ and $\mu^{L}_\Psi(t,x)$ are depicted; in the bottom row $\nu^F_\Psi(t,\lambda)$, $\nu^L_\Psi(t,\lambda)$ are depicted, both for time frames associated with $t=0.5,2, 3.5, 10$.} \label{fig:evo_ctrl_Test1}
	\end{center}
\end{figure}
We summarize the evolution of controlled and uncontrolled case up to final time  $T=50$ in Figure~\ref{fig:evo_Test1}, comparing the control and uncontrolled cases, respectively. We compare marginals $\mu^F_\Psi$ and $\mu^L_\Psi$ and the percentage of followers and leaders as  functions of time.
\begin{figure}[!ht]
	\begin{center}
		\includegraphics[width=0.325\linewidth]{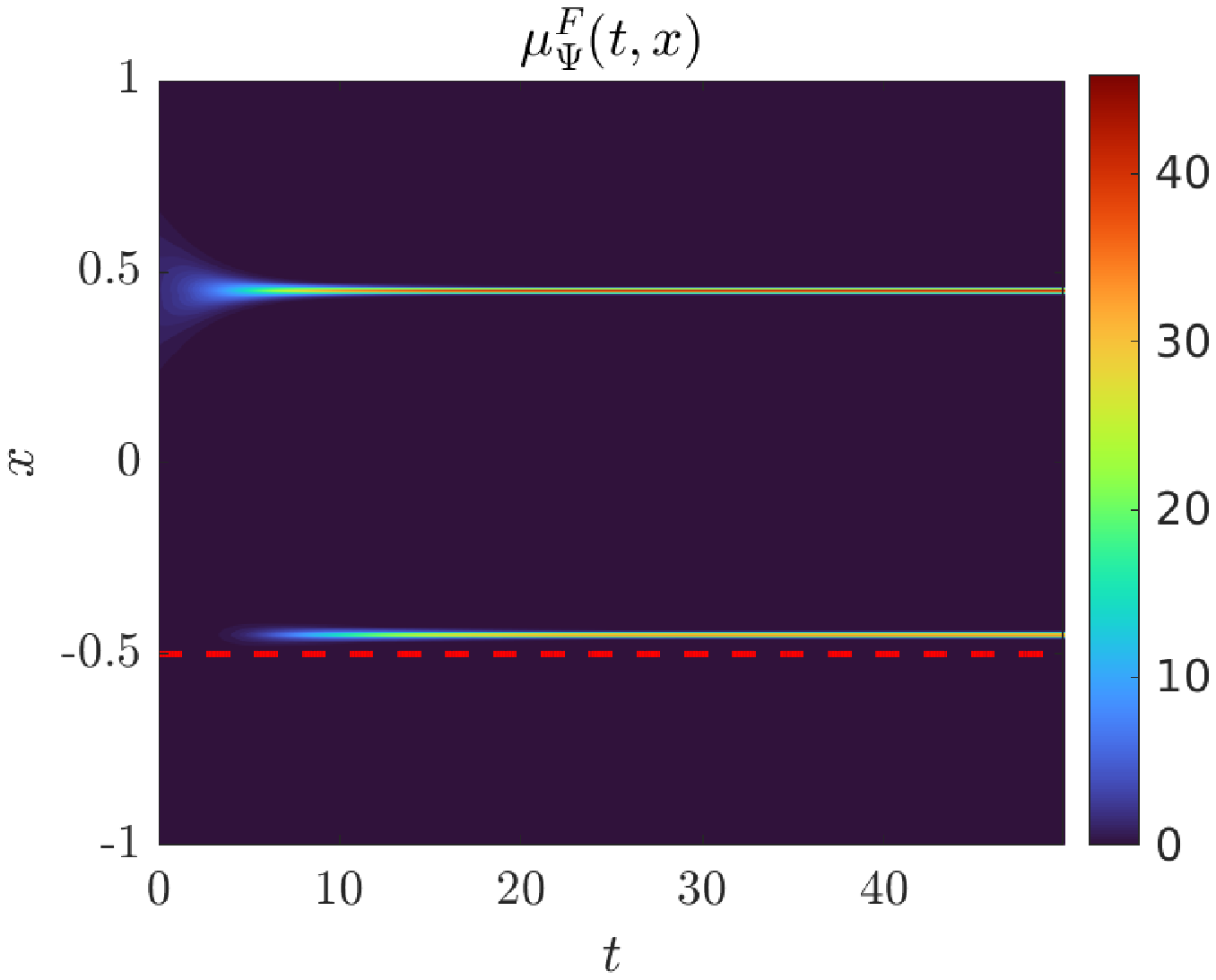}
		\includegraphics[width=0.325\linewidth]{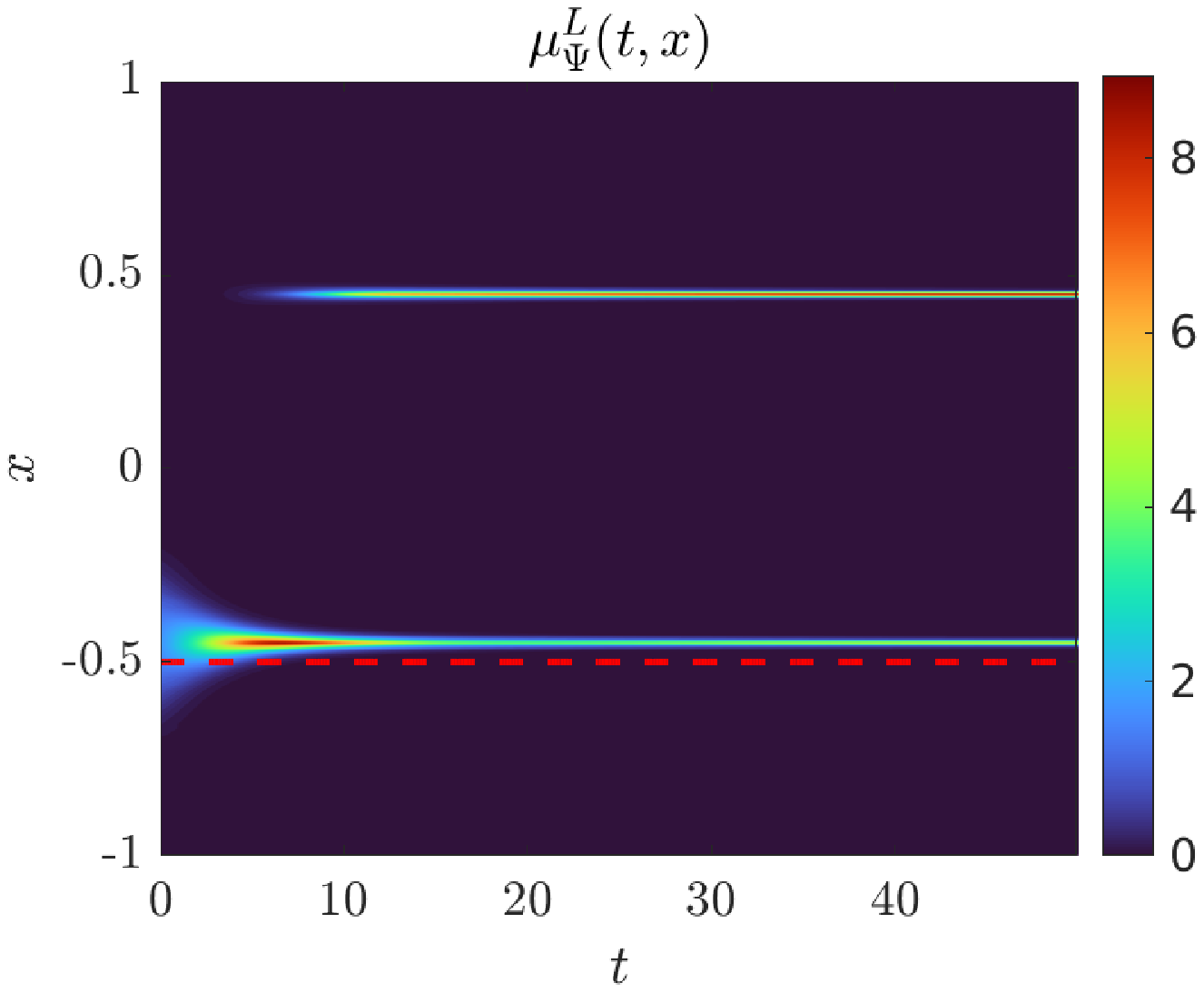}
		\includegraphics[width=0.275\linewidth]{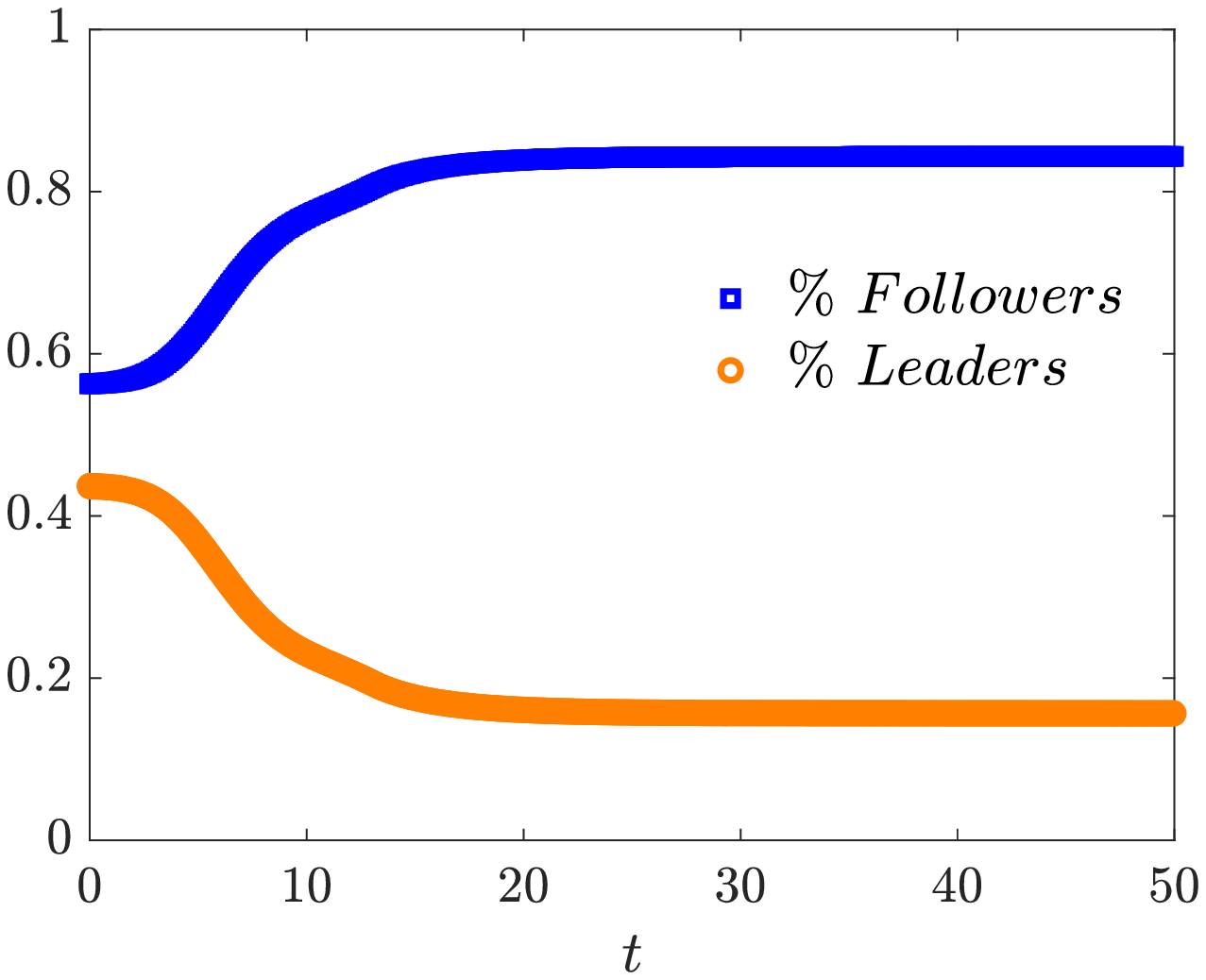}
		\\
		\includegraphics[width=0.325\linewidth]{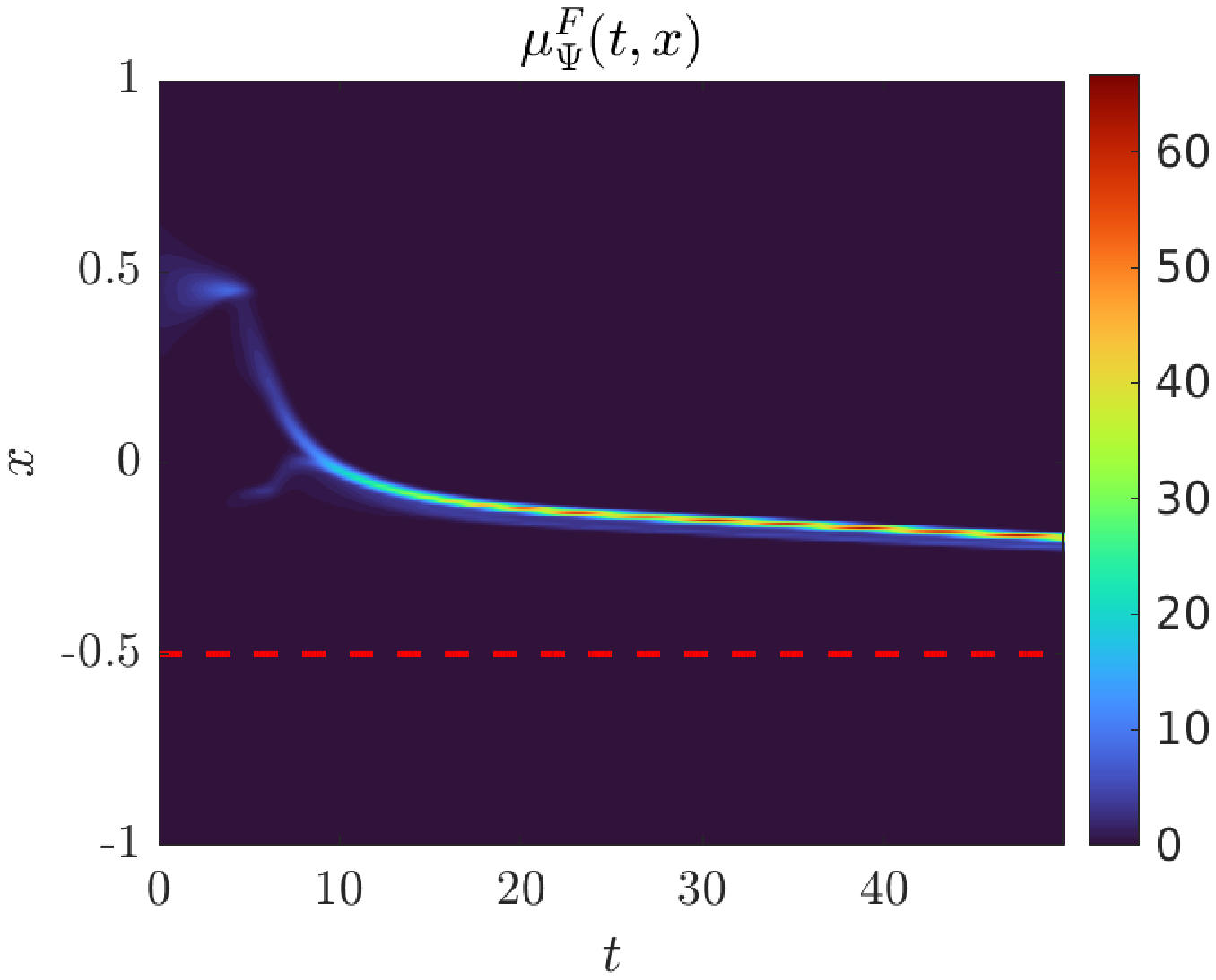}
		\includegraphics[width=0.325\linewidth]{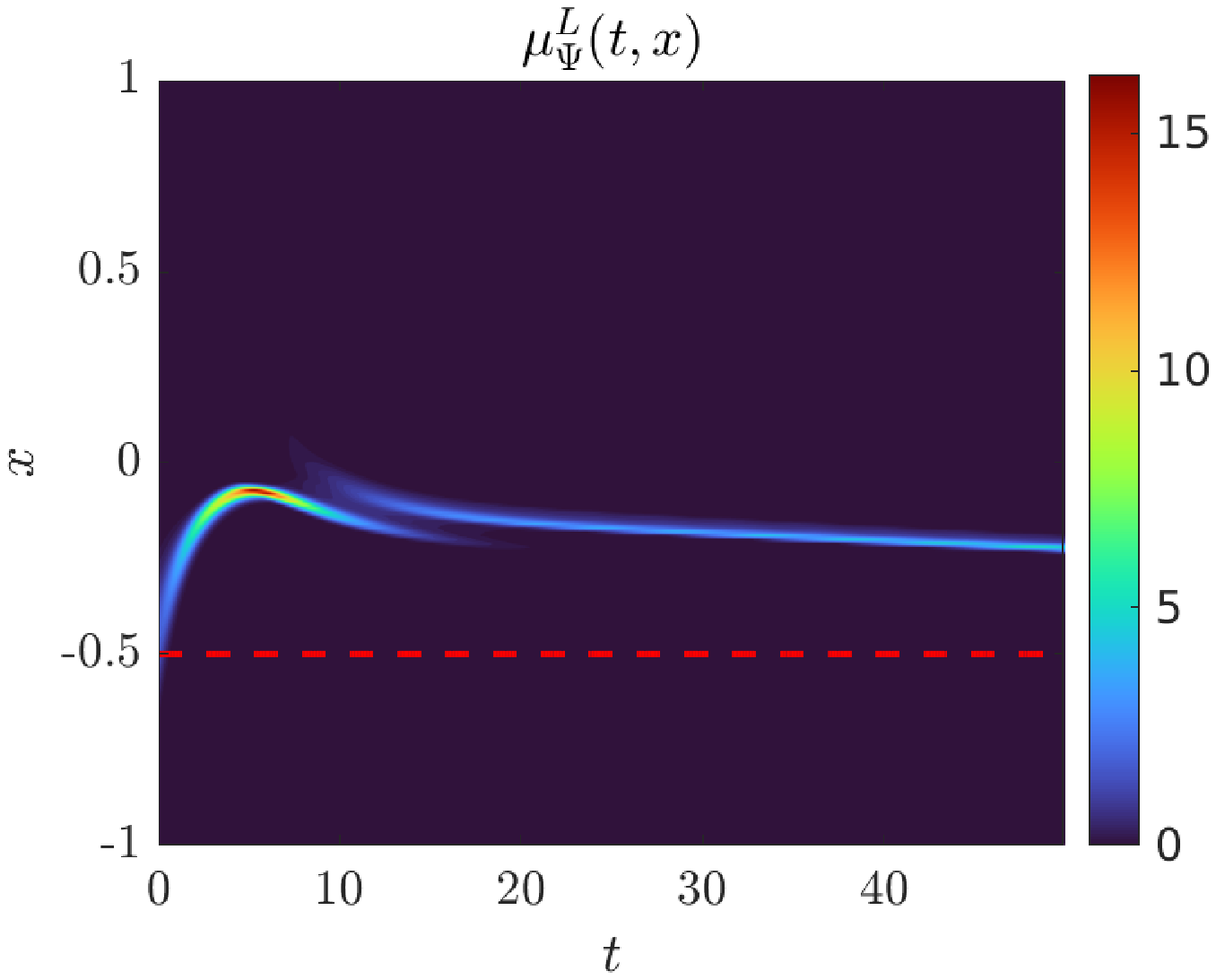}
		\includegraphics[width=0.275\linewidth]{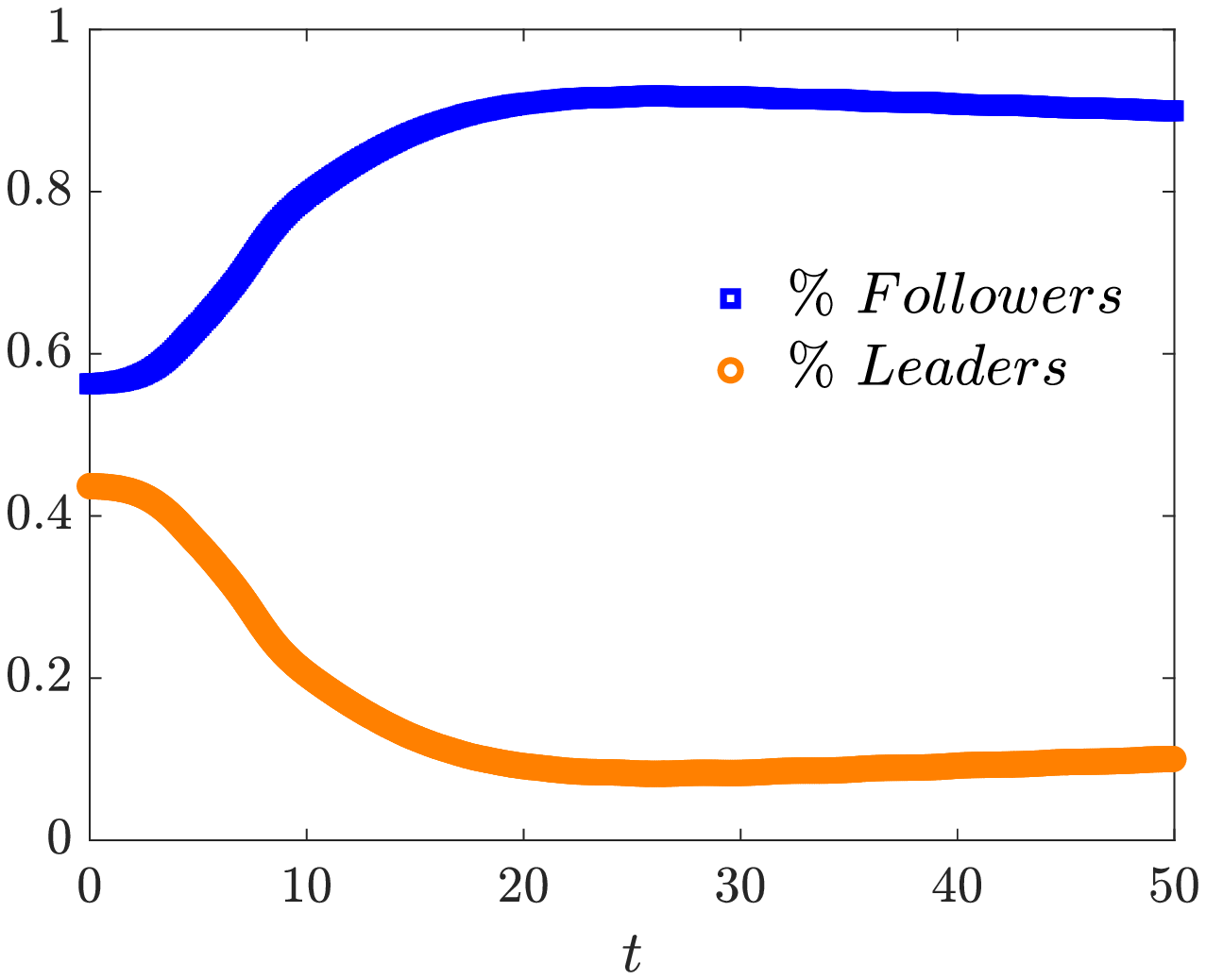}
		\caption{{\em Test 1.}~ Upper row: uncontrolled case; bottom row: controlled case.
			Left and central plots depict marginals in the time-space domain. Right plot shows the percentage of population associated with leader or follower populations. Right plot shows the percentage of population associated with leader and follower as functions of time.} \label{fig:evo_Test1}
	\end{center}
\end{figure}


\subsection{Two-leader game}\label{sec:2LG}

A rather natural extension of the situation considered in Section~\ref{sec:LF} consists in studying the interaction between three different populations: one of followers, still denoted with the label~$F$, and two of leaders, denoted by~$L_{1}$ and $L_{2}$, respectively, competing for gaining consensus among the followers and working to attract them towards their own objectives. A policy maker may choose to promote one of the two populations of leaders by favoring the interactions among these leaders and the followers. We discuss here how to model such a scenario within our analytical setting.

The set~$U$ consists now of three labels $U\coloneq \{ F, L_{1}, L_{2}\}$ endowed with the distance $d \colon U \times U \to [0, +\infty)$ defined as
\begin{align*}
&d(F, L_{1}) = d(L_{1}, F) = d(F, L_{2}) = d(L_{2}, F) = d(L_{1}, L_{2}) = d(L_{2}, L_{1}) \coloneq 1\,,\\
& d(F, F) = d(L_{1}, L_{1}) = d(L_{2}, L_{2}) \coloneq 0\,.
\end{align*}
The space of probability measures~$\sP(U)$ is identified with the simplex of~$\R^{3}$
\begin{displaymath}
\bigg\{ \lambda = (\lambda_{F}, \lambda_{L_{1}}, \lambda_{L_{2}}) \in \R^{3}: \, \lambda_{\bullet} \geq 0 \text{ for } \bullet \in U, \, \sum_{\bullet \in U} \lambda_{\bullet} = 1\bigg\}
\end{displaymath}
or, equivalently, to the subset~$\Delta$ of~$\R^{2}$
\begin{displaymath}
\Delta \coloneq \{\lambda =  (\lambda_{L_1}, \lambda_{L_2}) \in \R^{2}: \, 0 \leq \lambda_{L_1}, \lambda_{L_2} \leq 1, \, \lambda_{L_1} + \lambda_{L_2} \leq 1\}\,.
\end{displaymath}
Hence, in a discrete model the scalar values $\lambda_{L_{1}, i}, \, \lambda_{L_{2}, i}$ stand for the probability of the $i$-th particle of being an $L_{1}$-leader and an $L_{2}$-leader, respectively. Clearly, $\lambda_{F, i} = (1 - \lambda_{L_{1},i} - \lambda_{L_{2},i})$ represents the probability of being a follower. 

Assuming that the policy maker wants to promote the goals of the leaders $L_{1}$, the influence of the controls on the populations dynamics may be tuned by the function $h_{\Psi}(x, \lambda) = \overline{h}(\lambda_{L_1})$ for a bounded non-negative Lipschitz function $\overline{h} \colon[0,1]\to [0, 1]$ such that $\overline{h} (\lambda_{L_1}) = 1$ for $\lambda_{L_1}$ close to~$1$ and $\overline{h}(\lambda_{L_1}) = 0$ for $\lambda_{L_1}$ close to~$0$. 
Considering a cost function~$\phi$ of the form~\eqref{e:cost-phi}, for instance, the control~$u \in \R^{d}$ will act only on the $L_{1}$-leaders, as a consequence of Proposition~\ref{p:N-particles}.

Given $\Psi\in\cP(\R^d\times \Delta)$ and a Lipschitz continuous function $f = (f_{L_1}, f_{L_2}) \colon \Delta\to \Delta$ such that $f_{L_j}(\lambda) = f_{L_j}(\lambda_{L_j})$, for $j=1, 2$, we define the followers and leaders distributions as
\begin{equation*}\label{I7-2}
\begin{split}
&\mu^{L_{j}}_\Psi(B)\coloneqq\int_{B \times \Delta} f_{L_j}(\lambda_{L_j}) \,\mathrm{d}\Psi(x, \lambda), \qquad \text{for $j=1, 2$}\,, 
\\
&
\mu^F_\Psi(B)\coloneqq\int_{B \times\Delta} (1 - f_{L_1}(\lambda_{L_1}) - f_{L_2}(\lambda_{L_2})) \,\mathrm{d}\Psi(x, \lambda)
\end{split}
\end{equation*}
for every Borel subset $B$ of~$\mathbb{R}^d$. In the discrete setting, the leaders and followers distributions are
\begin{align*}
& \mu_{\Psi^N}^{L_{j}}(B)=\frac1N\sum_{i\,:\,x_i\in B}  f_{L_j}(\lambda_{L_{j},i}), \qquad \text{for $j=1, 2$},
\\
& \mu_{\Psi^N}^F(B)=\frac1N\sum_{i\,:\,x_i\in B}( 1 - f_{L_1}(\lambda_{L_{1},i}) - f_{L_2}(\lambda_{L_{2},i}) )\,.
\end{align*}
A possible choice for $f_{L_j}$ is any Lipschitz regularization of the indicator function of the set $\{\lambda_{L_j} \ge m\}$ with $m> \frac12$ and such that $f_{L_j}(\lambda_{L_j}) = 0$ for $\lambda_{L_j} \leq \frac12$, compatible with the request that~$f$ maps~$\Delta$ in $\Delta$. 
We further notice that the choice $f_{L_j}(\lambda_{L_j})=\lambda_{L_j}$ is still allowed with the same interpretation given in~\eqref{I7-discrete}.

The velocity field~$v_{\Psi}(x, \lambda)$ in~\eqref{e:vel-ex} can be easily modified for the current scenario by setting, for instance,
\begin{equation*}
v_{\Psi}(x, \lambda) \coloneq  v^{L_1}_{\Psi}(x, \lambda)+v^{L_2}_{\Psi}(x, \lambda)+v^{F}_{\Psi}(x, \lambda),
\end{equation*}
where
\begin{equation}\label{e:vel-ex2}
v^{\star}_{\Psi}(x, \lambda)\coloneqq f_\star(\lambda) \int_{\R^d\times\Delta} \sum_{\bullet\in\{F,L_1,L_2\}} K^{\star\bullet}(x-x') f_\bullet(\lambda') \,\de\Psi(x',\lambda')
\end{equation}
under the additional position that $f_{F}(\lambda)=1-f_{L_1}(\lambda_{L_1})-f_{L_2}(\lambda_{L_2})$.
The transition $\T_{\Psi}(x, \lambda)$ is now given by
\begin{equation}\label{e:trans-ex2}
\!\!\!\!\!\!\!\!\!\mathcal{T}_{\Psi}(x, \lambda) \coloneq \!\!\left(\begin{array}{ccc}
\!\!- \alpha_{L_{1}L_{1}}(x, \Psi) &\!\! \alpha_{L_{1}L_{2}}(x, \Psi) & \!\!\alpha_{L_{1}F}(x, \Psi) \\
\!\!\alpha_{L_{2}L_{1}}(x, \Psi) & \!\!-\alpha_{L_{2}L_{2}}(x, \Psi) & \!\!\alpha_{L_{2}F}(x, \Psi) \\
\!\!\alpha_{FL_{1}}(x, \Psi) & \!\!\alpha_{FL_{2}}(x, \Psi) & \!\!- \alpha_{FF}(x, \Psi)
\end{array}\right)
\!\!\!
\left(
\begin{array}{ccc}
\!
g_{L_1}(\lambda_{L_1}) \\
\!
g_{L_2}(\lambda_{L_2}) \\
\!
1 - g_{L_1}(\lambda_{L_1}) - g_{L_2}(\lambda_{L_2})
\end{array}\right)
\end{equation}
where the transition rates $\alpha$ are defined as in~\eqref{e:alpha} with the obvious modifications, and $g_{L_j}$ have similar properties as $f_{L_j}$. To comply with~$(\T_{0})$, we need (see~\cite[Proposition~5.1]{MorSol19})
\begin{equation}\label{e:stoc-matrix}
\alpha_{\bullet\bullet} (x, \Psi) = \sum_{\substack{\star \in U\\ \star \neq \bullet}} \alpha_{\star \bullet} (x, \Psi)\,,
\end{equation}
in view of which we can write (omitting the dependence on $x,\Psi$)
\begin{equation*}
\T_\Psi(x,\lambda)=\left(\begin{array}{cc}
-(\alpha_{L_1L_1}+\alpha_{L_1F}) & \alpha_{L_1L_2}-\alpha_{L_1F} \\
\alpha_{L_2L_1}-\alpha_{L_2F} & -(\alpha_{L_2L_2}+\alpha_{L_2F})
\end{array}\right)
\left(
\begin{array}{ccc}
g_{L_1}(\lambda_{L_1}) \\
g_{L_2}(\lambda_{L_2})
\end{array}\right)+
\left(
\begin{array}{ccc}
\alpha_{L_1F} \\
\alpha_{L_2F}
\end{array}\right)
\end{equation*}
in order to determine the evolution of the two independent parameters $\lambda_{L_1}$ and $\lambda_{L_2}$\,.

Since the policy maker promotes the $L_{1}$-leaders, the Lagrangian should penalize the distance of the population~$L_{1}$ from their goal. As in~\eqref{e:lag-1}, this is done by considering a function of the form
\begin{displaymath}
\widetilde{\mathcal{L}}_1 (x, \lambda) \coloneq \theta(\lambda_{L_1}) | x - \overline{x}|^2,
\end{displaymath}
where $\overline{x} \in \R^{d}$ denotes the desired goal of the $L_{1}$-leaders and $\theta \colon [0, 1] \to [0, 1]$ is a continuous function which is $0$ close to $0$ and $1$ close to~$1$. 
With the same idea, the second term~\eqref{e:lag-2} is modified in order to penalize only the distance of the $L_{1}$-leaders from the barycenter of the followers
\begin{equation*}
\widetilde{\cL}_2(x,\lambda,\Psi)=\theta(\lambda_{L_1}) \bigg|x-\fint_{\R^d} x'\,\de\mu_\Psi^F(x')\bigg|^2.
\end{equation*}
Again, we notice that $\widetilde{\cL}_{2}$ is continuous as long as $\mu_\Psi^F (\R^{d})>0$. Finally, the Lagrangian~$\widetilde{\cL}$ of the system has the same structure of~\eqref{e:lag-3}, i.e., $\widetilde{\cL} = \alpha \widetilde{\cL}_{1} + (1 - \alpha) \widetilde{\cL}_{2}$ for a parameter~$\alpha \in [0, 1]$ to be tuned.

\subsubsection{Test~2: Opinion dynamics with competing leaders} 
We consider the opinion dynamics presented in Test 1,  where the opinion variable is $x\in[-1,1]$ with $\{\pm1\}$  two opposite opinions. We introduce two populations of leaders competing over the consensus of the followers.
The first population of leaders $L_1$ has a radical attitude aiming to mantain their position, and their strategy is driven by the policy maker. Instead, the population  $L_2$ is characterized by a populistic attitude, without the intervention of an optimization process: they are willing to move from their position in order to have a broader range of interaction with the remaining agents.

The interaction field $v_{\Psi}$ \eqref{e:vel-ex2} is characterized by bounded confidence kernels with the following structure
\begin{align}
K^{\star \bullet}(x-x') = \chi_\varepsilon(\{|x-x'|\leq \kappa_{\star\bullet}\}),\qquad\text{for $\star,\bullet\in\{F,L_1,L_2\}$,}
\end{align}
where $\eps\geq0$ is a regularization parameter for the characteristic function $\chi$ and $\kappa_{\bullet\star}$ represent the confidence intervals with the following numerical values,
\begin{align*}
\kappa_{FF} &= 0.35,\quad\kappa_{FL_1} = 0.5,\quad \kappa_{FL_2}= 0.5,\quad\kappa_{L_1F} =\kappa_{L_2F}=0,\cr
\kappa_{L_1L_1} &= 0.4,\quad \kappa_{L_2L_2} = 0.8,\quad\kappa_{L_1L2}=\kappa_{L_2L2}= 0.
\end{align*}
The weighting functions $f_{L_1},f_{L_2}$ are such that $ f_{L_j}(\lambda_{L_j})\equiv 1-\ell(\lambda_{L_j})$ with
\begin{align}\label{e:ell-g}
\ell(\lambda) = \dfrac{e^{C(\lambda-\bar \lambda)}}{1+e^{C(\lambda-\bar \lambda)}}, \quad C = 10^3,\quad \bar \lambda = 0.5.
\end{align}
The transition operator $\mathcal{T}_\Psi(x,\lambda)$ in \eqref{e:trans-ex2} is identified by the following quantities
\begin{align*}\label{coeff}
\alpha_{L_jF}(x,\Psi) &= a_{L_jF}\left(1-\mathcal D_L(x,\Psi)\right),\qquad \alpha_{FF}(x,\Psi)= \sum_{j=1,2}\alpha_{L_jF}(x,\Psi)\cr
\alpha_{L_jL_j}(x,\Psi)& = a_{L_jL_j}\left(1-\mathcal D_F(x,\Psi)\right),
\qquad
\alpha_{FL_j}(x,\Psi) = \alpha_{L_jL_j}(x,\Psi),
\end{align*}
and $\alpha_{L_1L_2}(x,\Psi)=\alpha_{L_2L_1}(x,\Psi)=0$, coherently with respect to \eqref{e:stoc-matrix}.
Functions $\mathcal D_F$ and $ \mathcal D_L$ represent the concentration of followers and the total concentration of leaders at position~$x$, defined similarly to \eqref{e:conc-ex1}. We use the following parameters
\[
a_{L_jF}=0.015,\quad a_{L_jL_j} = 0.025,\qquad j = {1,2}.
\]
The weighting function $g_j(\lambda)$ is defined as in \eqref{e:ell-g} with $C=20$ and $\bar \lambda = 0.5$.
Finally, the cost functional is defined by the Lagrangian defined in \eqref{e:lag-3}, with $\lambda = \lambda_{L_1}$ since only the radical leaders are controlled. Radical leaders aim to steer followers towards $\bar x=-0.75$ and keeping track of followers average position with weighting parameter $\alpha = 0.85$ and $\theta(\lambda) = 1-\ell(\lambda)$. We account for quadratic penalization of the control in \eqref{e:cost-phi} by choosing $\gamma = 2$.

\begin{figure}[!ht]
	\begin{center}
		\includegraphics[width=0.375\linewidth]{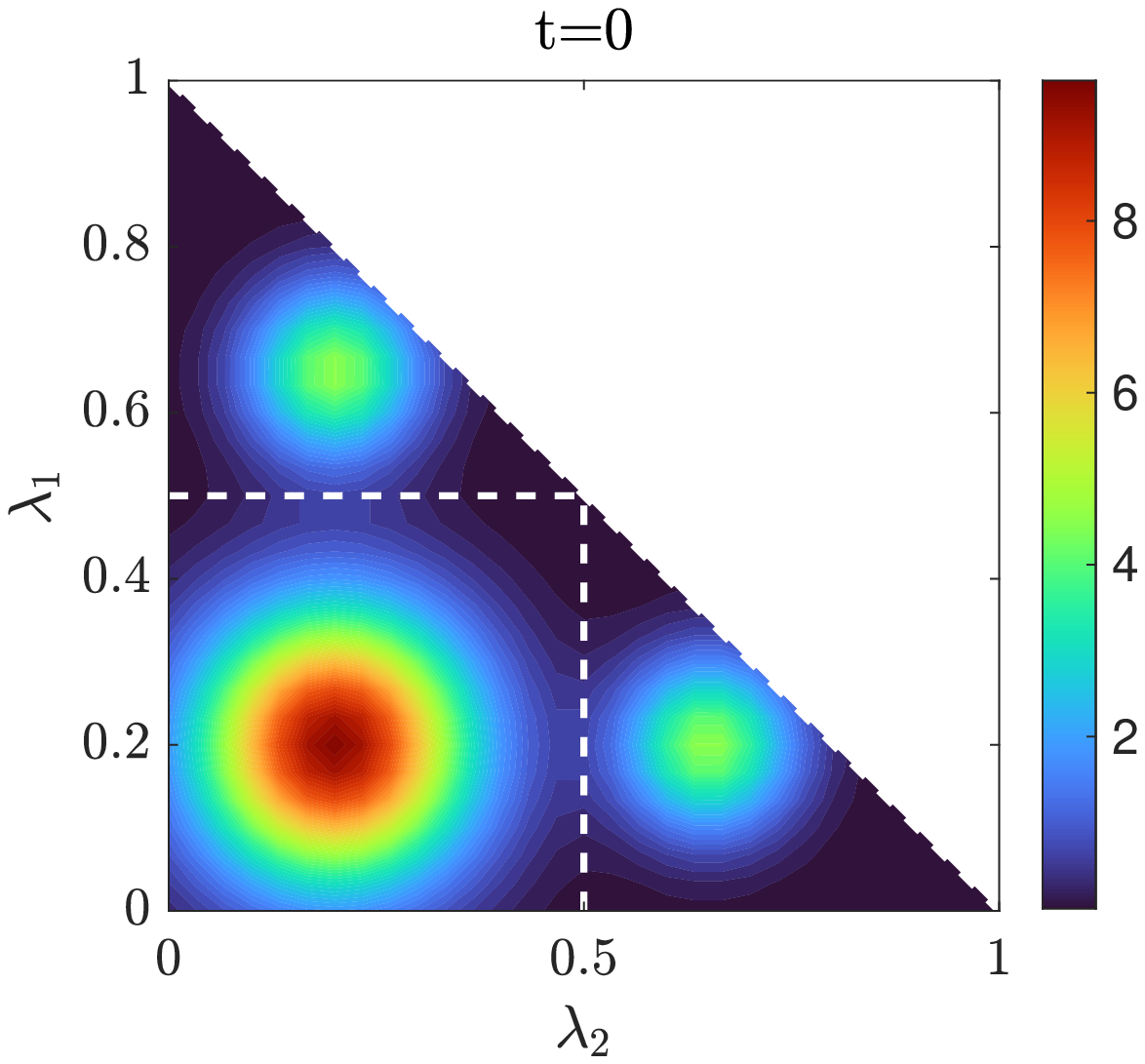}
		\hspace{+0.5cm}
		\includegraphics[width=0.375\linewidth]{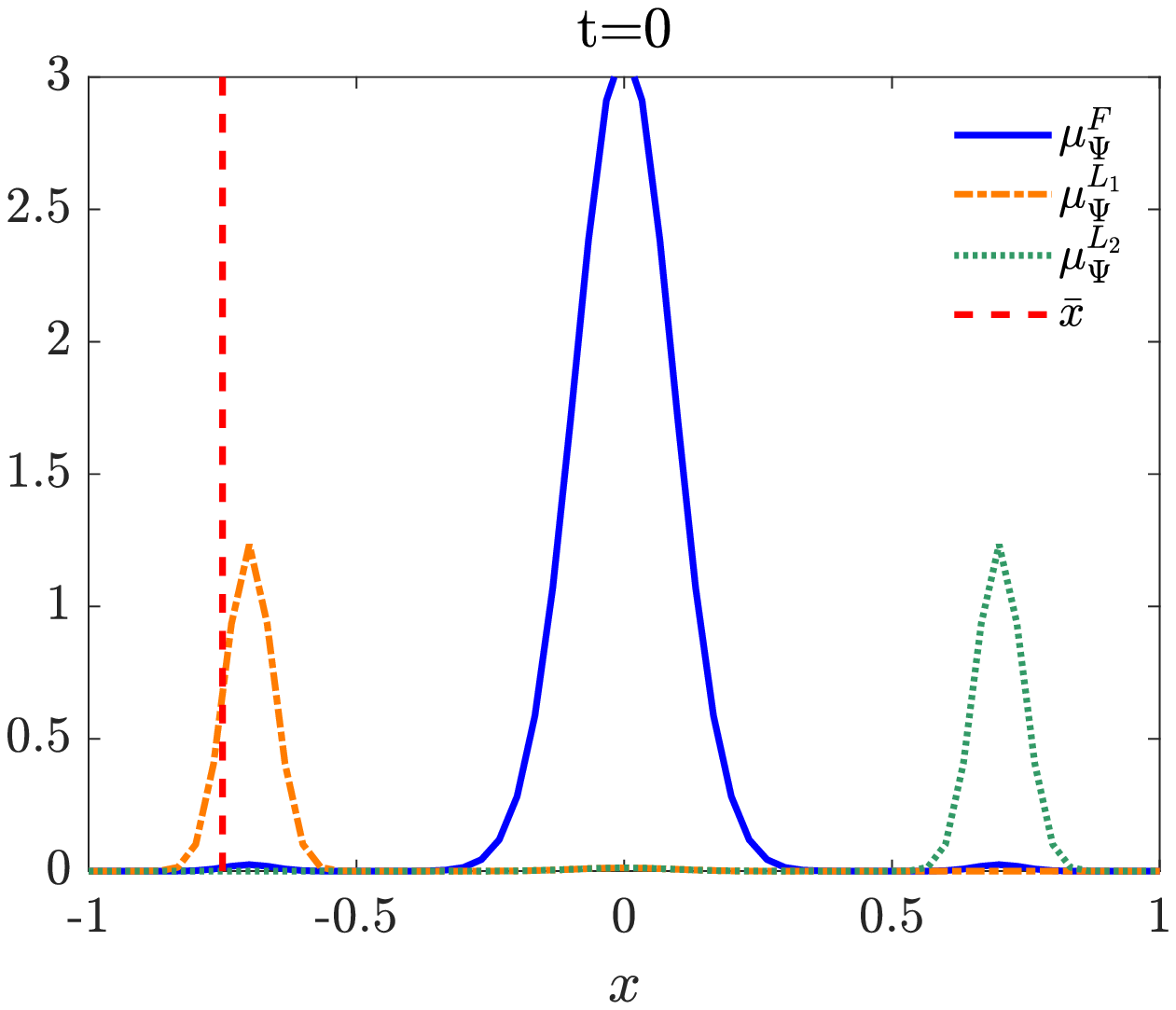}
	\end{center}
\caption{{\em Test 2}. Initial distribution $\Psi_0(x,\blambda)$ and the marginals associated with the opinion space $x$, $\mu^{F}_\Psi(t,x)$, $\mu^{L}_\Psi(t,x)$, and the label space $(\lambda_{L_{1}},\lambda_{L_{2}})$, $\nu_\Psi(t,\blambda)$. Red dashed line corresponds to the target opinion for $L_1$ leaders.}\label{fig:init_data_Test2}
\end{figure}

In Figure~\ref{fig:init_data_Test2}, we report the choice of the initial data, and the marginals $\mu^{F}_\Psi(t,x)$, $\mu^{L_1}_\Psi(t,x)$ and $\mu^{L_2}_\Psi(t,x)$ relative to the opinion space, and to the label space $\nu_\Psi(t,\blambda)$,  defined as follows
\[
\Psi_0(x,\blambda) \coloneqq C_0 \left(\exp\left\{-\frac{(x-x_F)^2}{\sigma_{x,F}^2}-\frac{|\blambda-\bar\blambda_F|^2}{{\sigma_{\lambda,F}^2}}\right\}+\sum_{j=1,2}\exp\left\{-\frac{(x-x_j)^2}{{\sigma_{x,j}^2}}-\frac{|\blambda-\bar\blambda_{j}|^2}{{\sigma_{\lambda,j}^2}}\right\}\right),
\]
where here $\blambda=(\lambda_{L_1},\lambda_{L_2})$, the parameters are ${\sigma_{\lambda,F}^2}=1/40,\sigma_{\lambda,j}^2=1/100,$ $\sigma_{x,j}^2 =1/60$,~$\sigma_{x,F}^2=1/250$, $\bar \blambda_F=(0.2,0.2), \bar \blambda_1=(0.2,0.65),\bar \blambda_2=(0.65,0.2)$, $x_F=0$, $x_1=-0.65,x_2=0.65$, and~$C_0$ is a normalizing constant.

Figure \ref{fig:evo_noctrl_Test2} reports from left to right four frames of the marginals up to time $t=\{5,15, 27.5, 50\}$, without control. Without the action of a policy maker, the majority of followers 
are driven close to the initial position of populist leaders $L_2$, who interact with a wider portion of agents.
 In Figure~\ref{fig:evo_ctrl_Test2}, the control action of the policy maker is activated resulting in a different distribution of the followers: while the populistic leaders retain some capability of attraction, the portion of the followers which is driven towards the target position $\bar x$ of the radical leaders $L_1$ is considerably larger than in the uncontrolled case.

\begin{figure}[!ht]
	\begin{center}
		\includegraphics[width=0.225\linewidth]{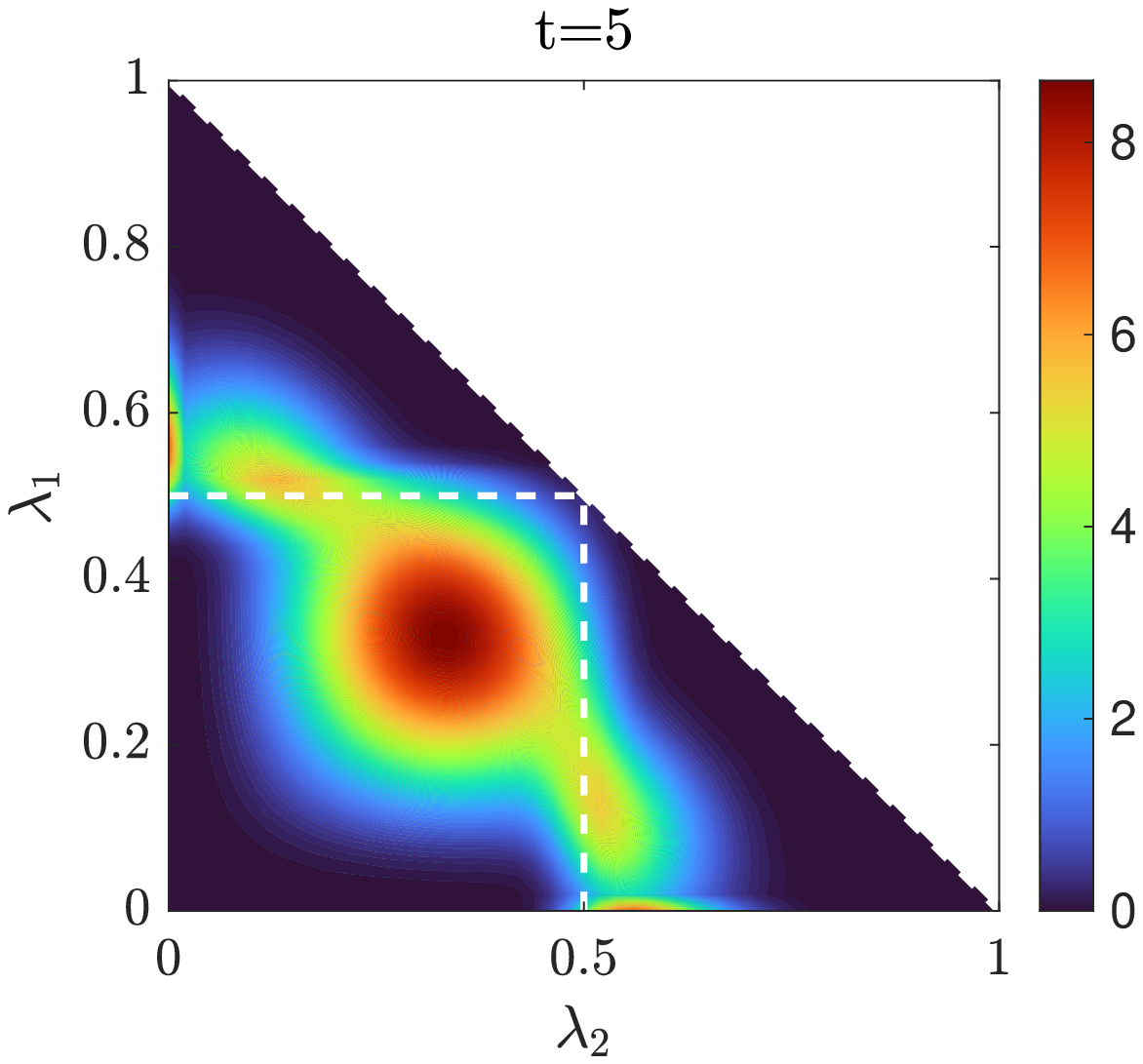}
		\includegraphics[width=0.225\linewidth]{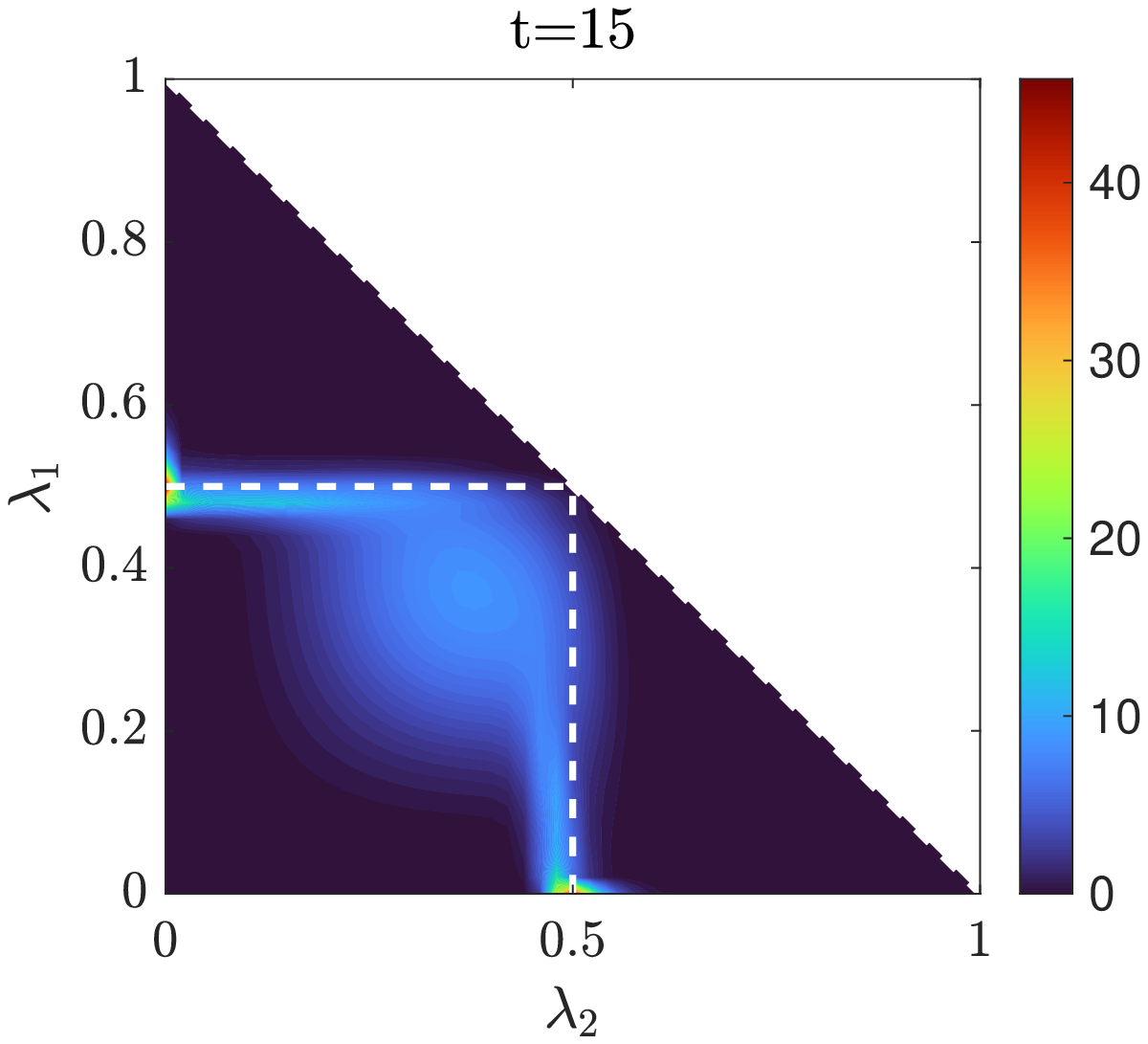}
		\includegraphics[width=0.225\linewidth]{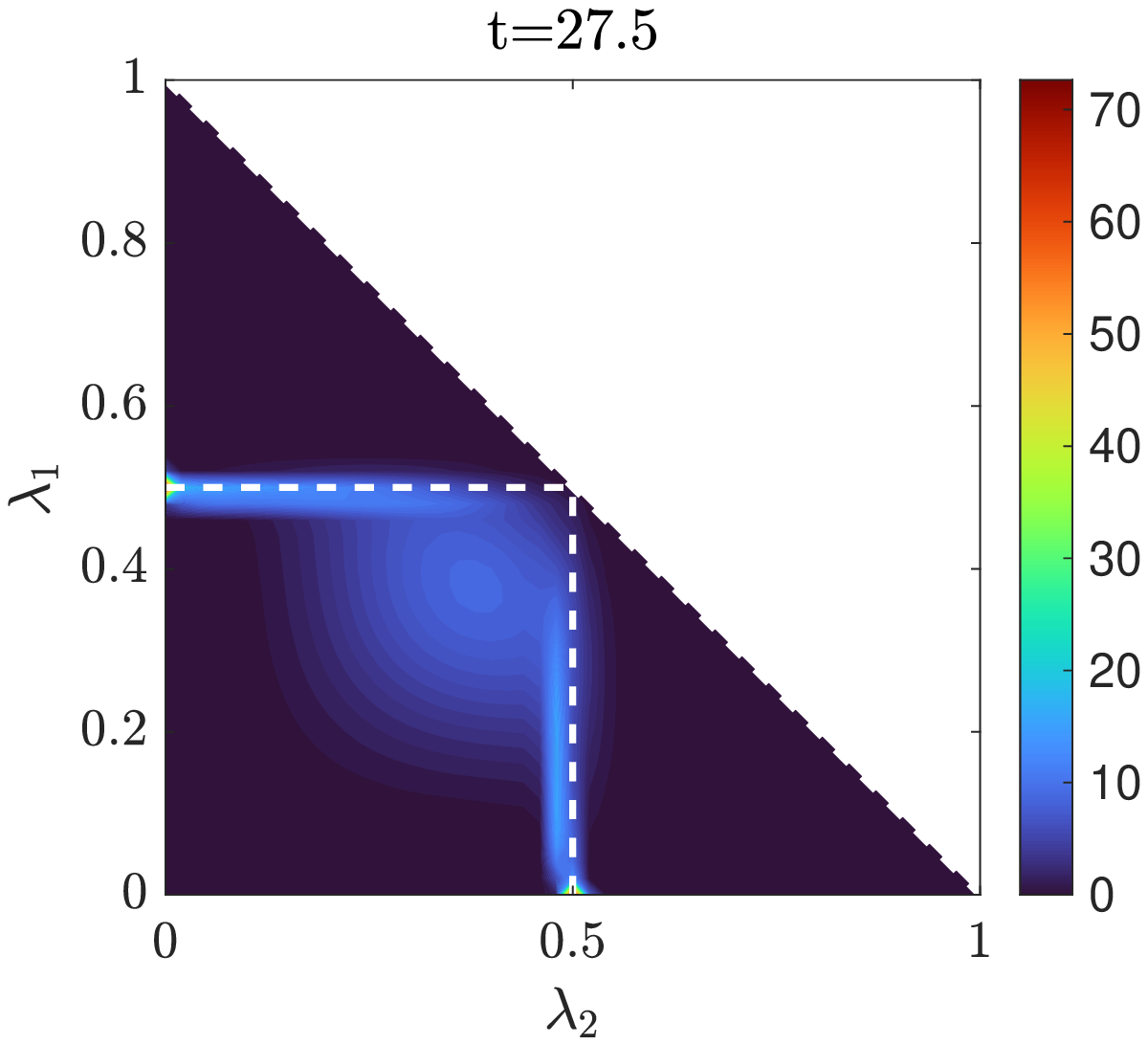}
		\includegraphics[width=0.225\linewidth]{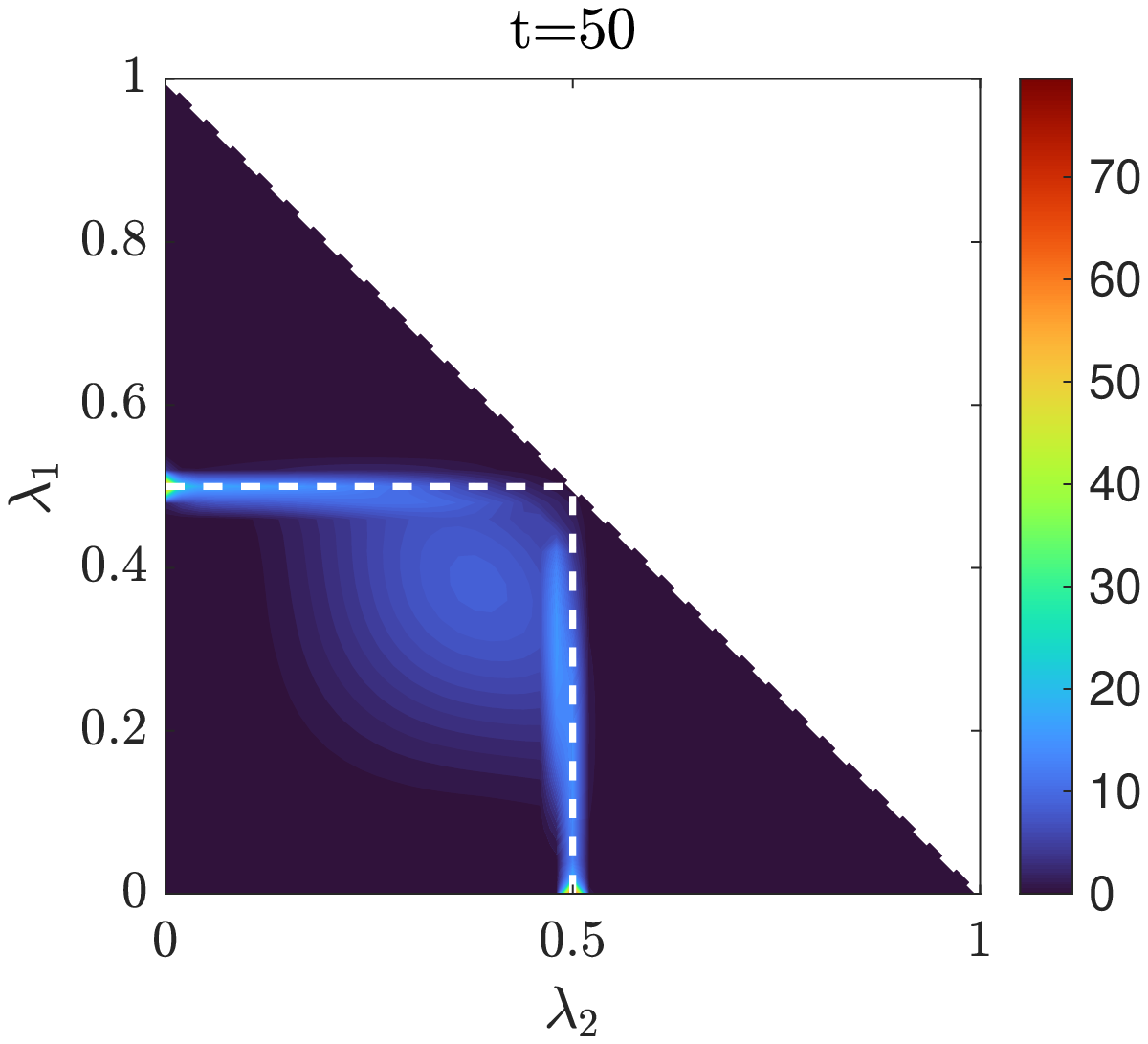}
		\\
\includegraphics[width=0.225\linewidth]{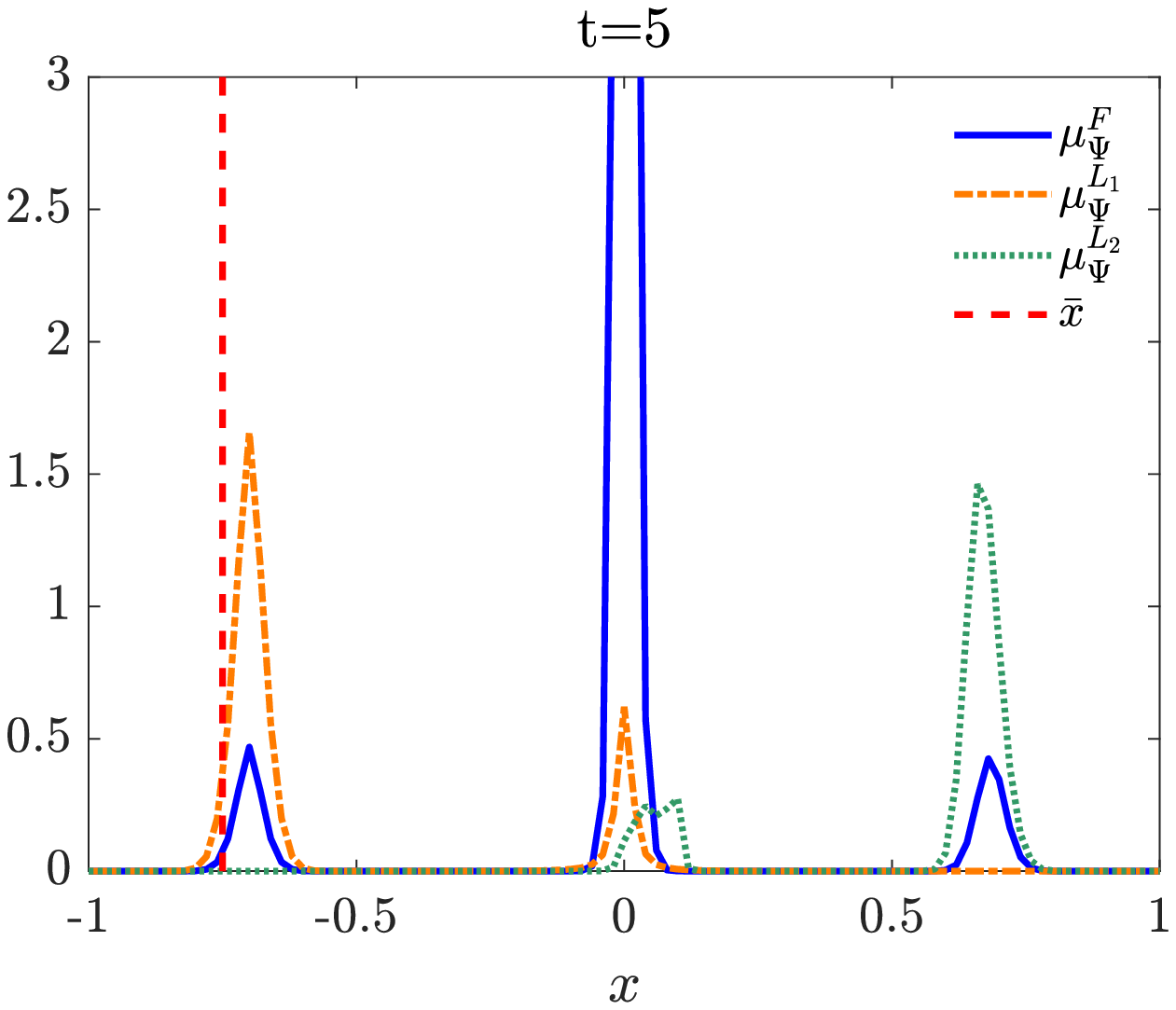}
\includegraphics[width=0.225\linewidth]{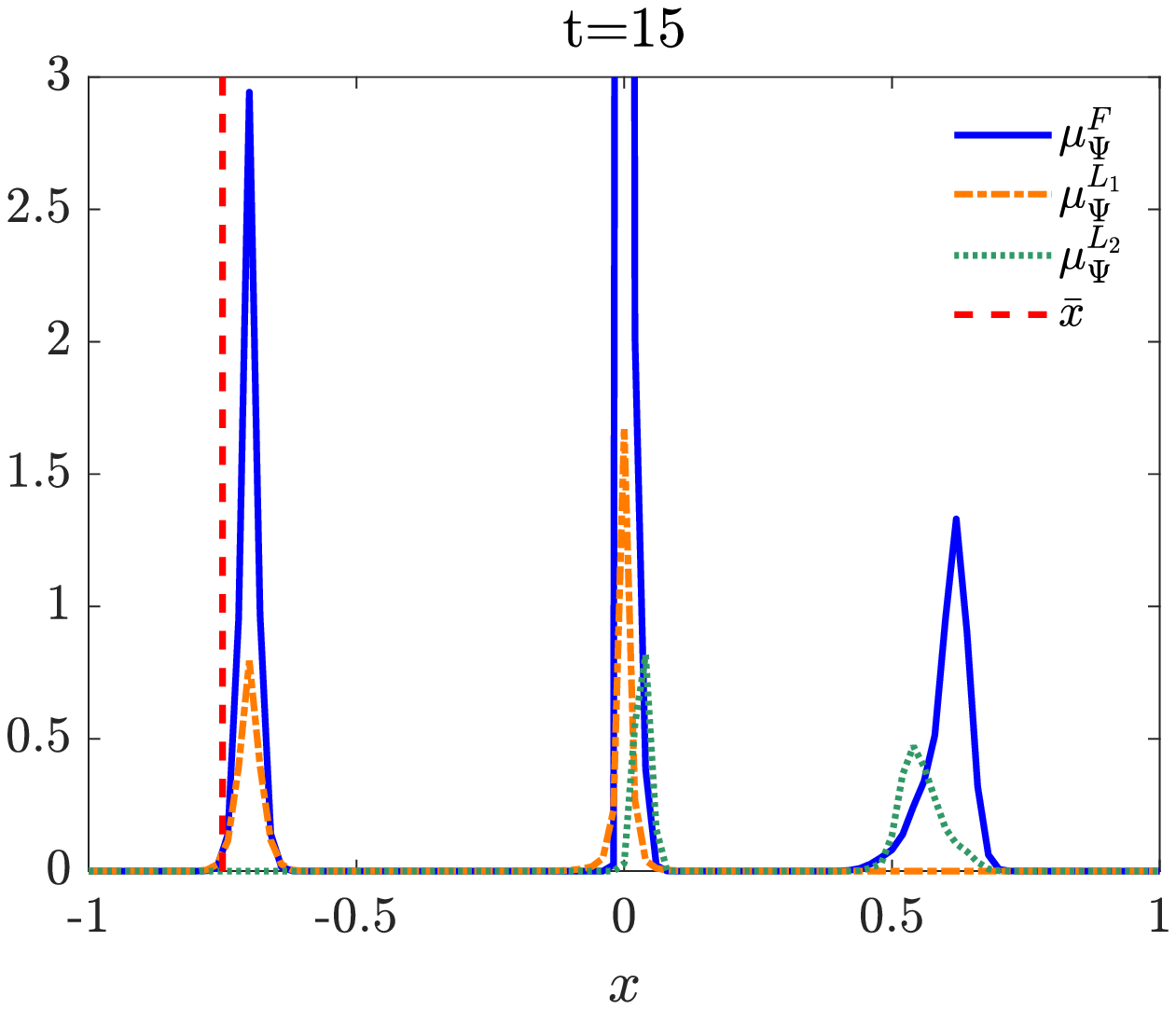}
\includegraphics[width=0.225\linewidth]{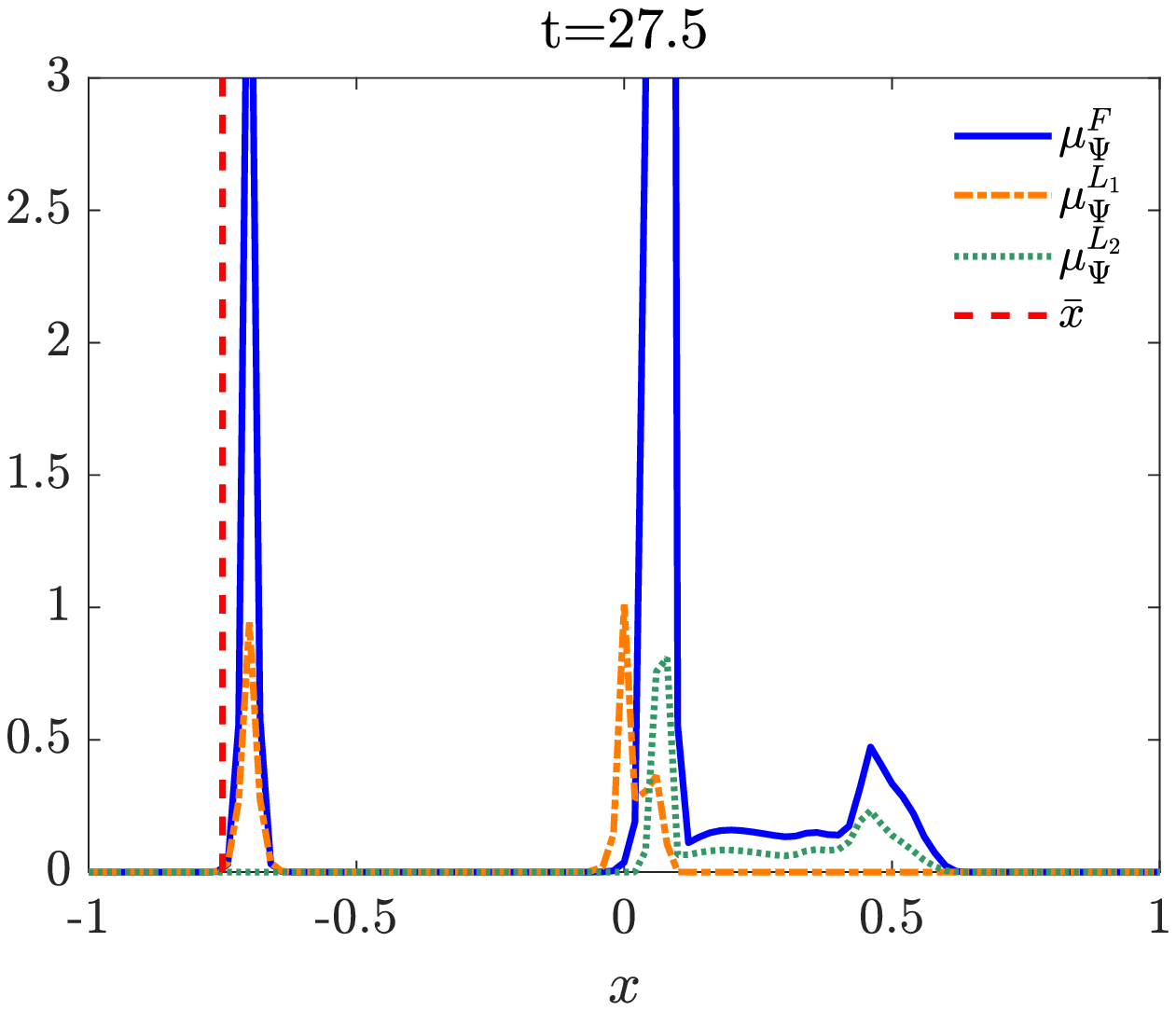}
\includegraphics[width=0.225\linewidth]{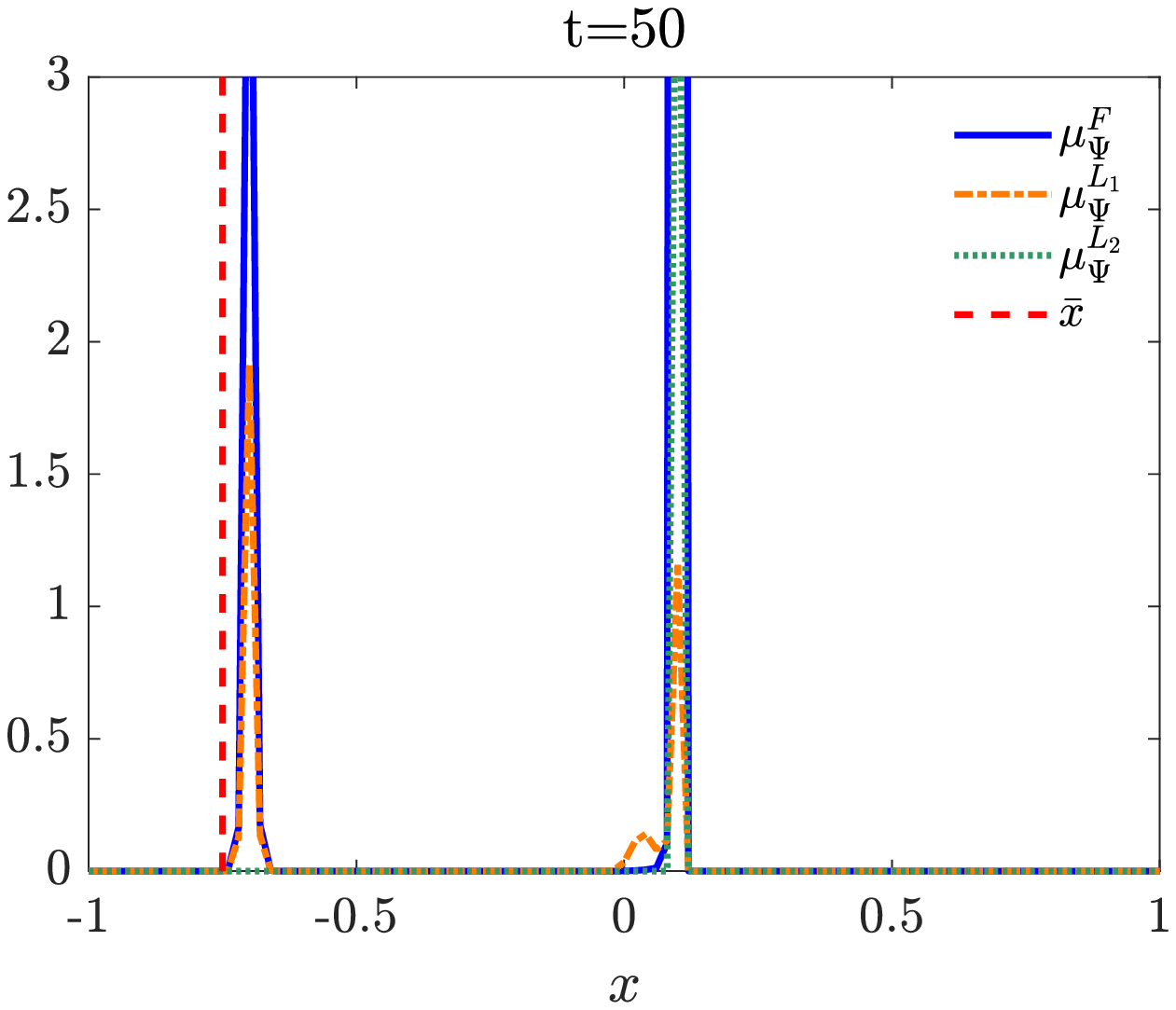}
		\caption{{\em Test 2.}~Evolution of the marginals without control for time frames   $t=\{5,15,27.5, 50\}$. Bottom row depicts $\mu^{F}_\Psi(t,x),$ $\mu^{L_1}_\Psi(t,x)$ and $\mu^{L_2}_\Psi(t,x)$; top row  shows $\nu_\Psi(t,\blambda)$.} \label{fig:evo_noctrl_Test2}
	\end{center}
\end{figure}

\vspace{5mm}

\begin{figure}[!ht]
	\begin{center}
		\includegraphics[width=0.225\linewidth]{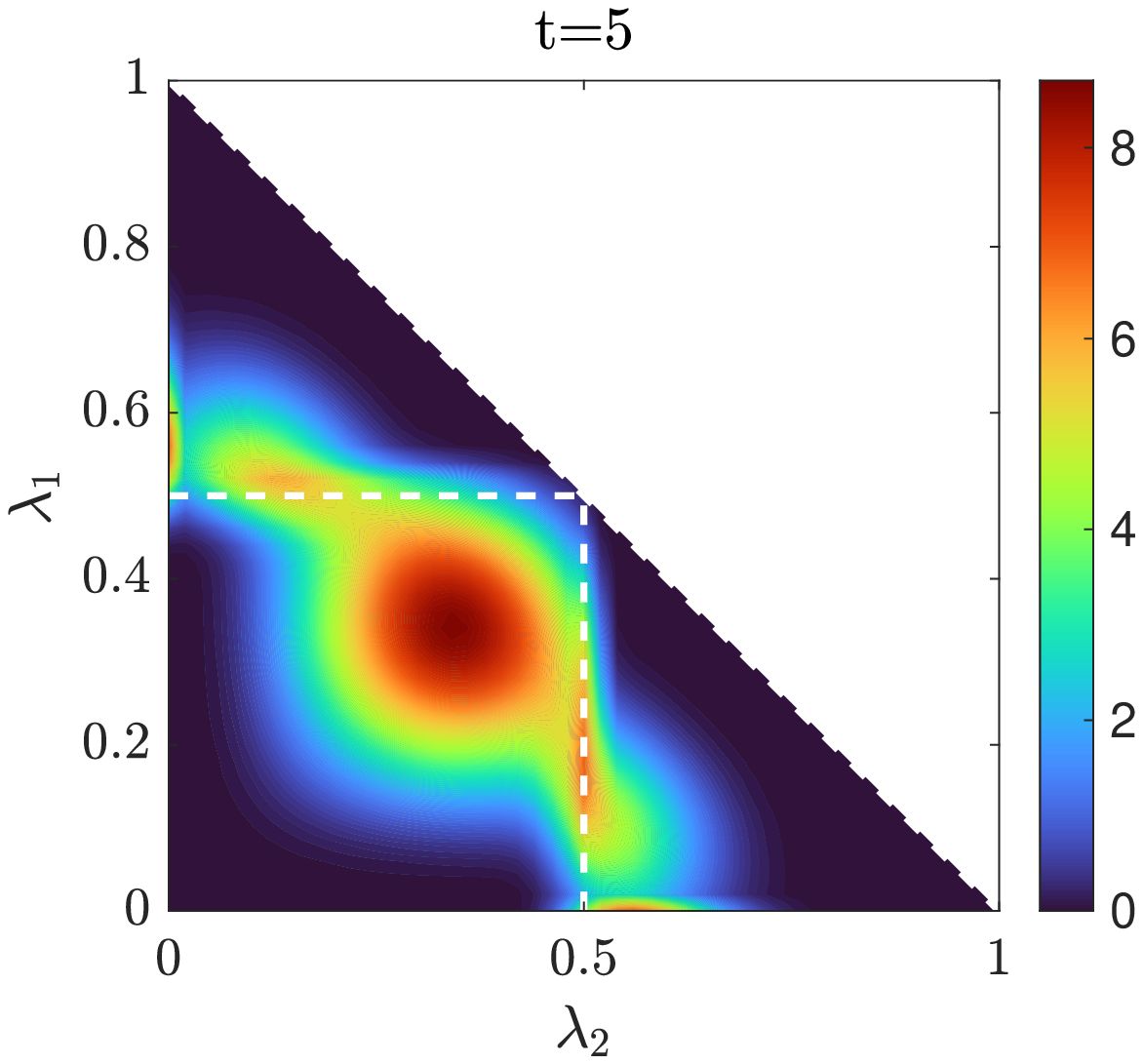}
	\includegraphics[width=0.225\linewidth]{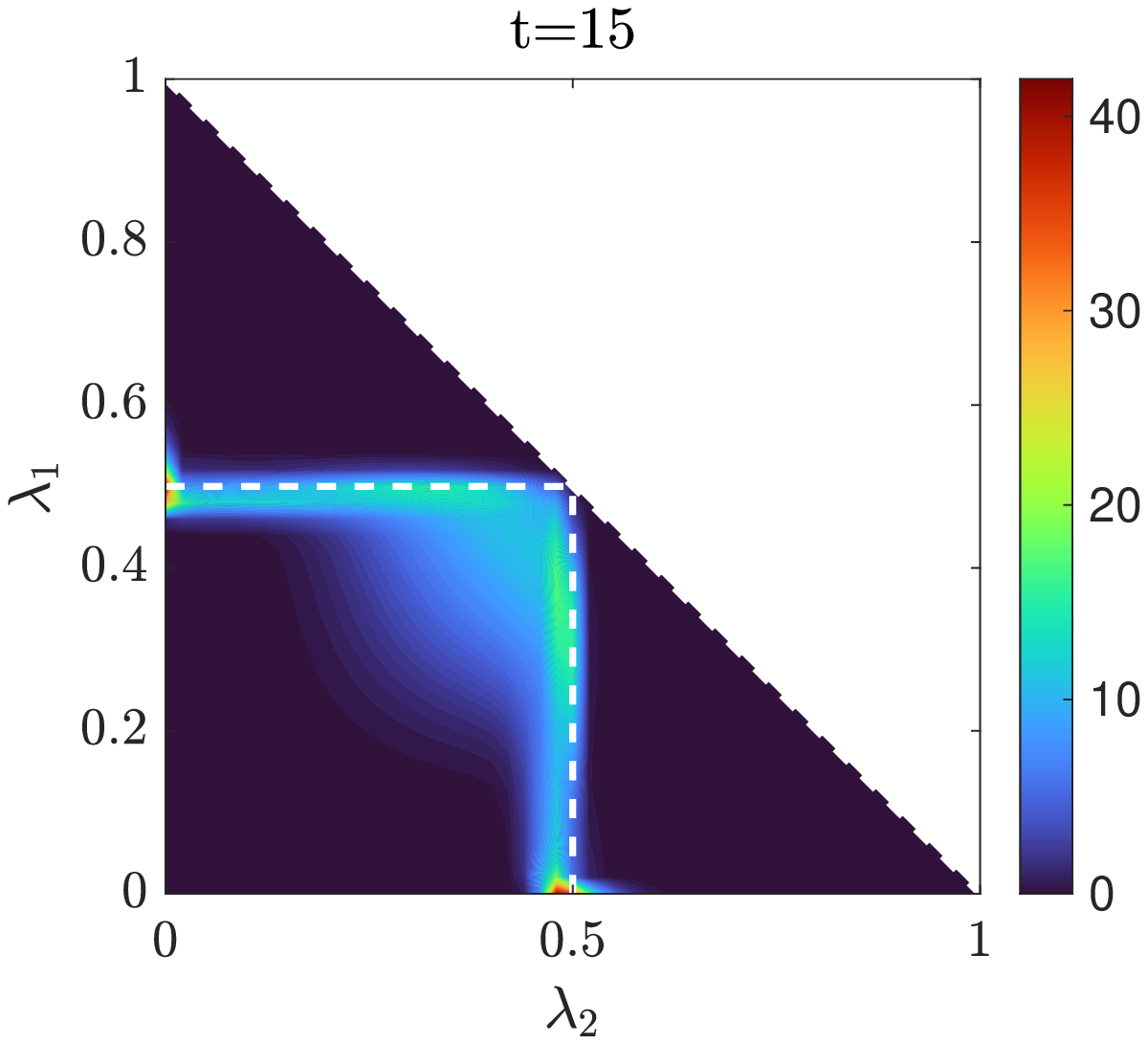}
	\includegraphics[width=0.225\linewidth]{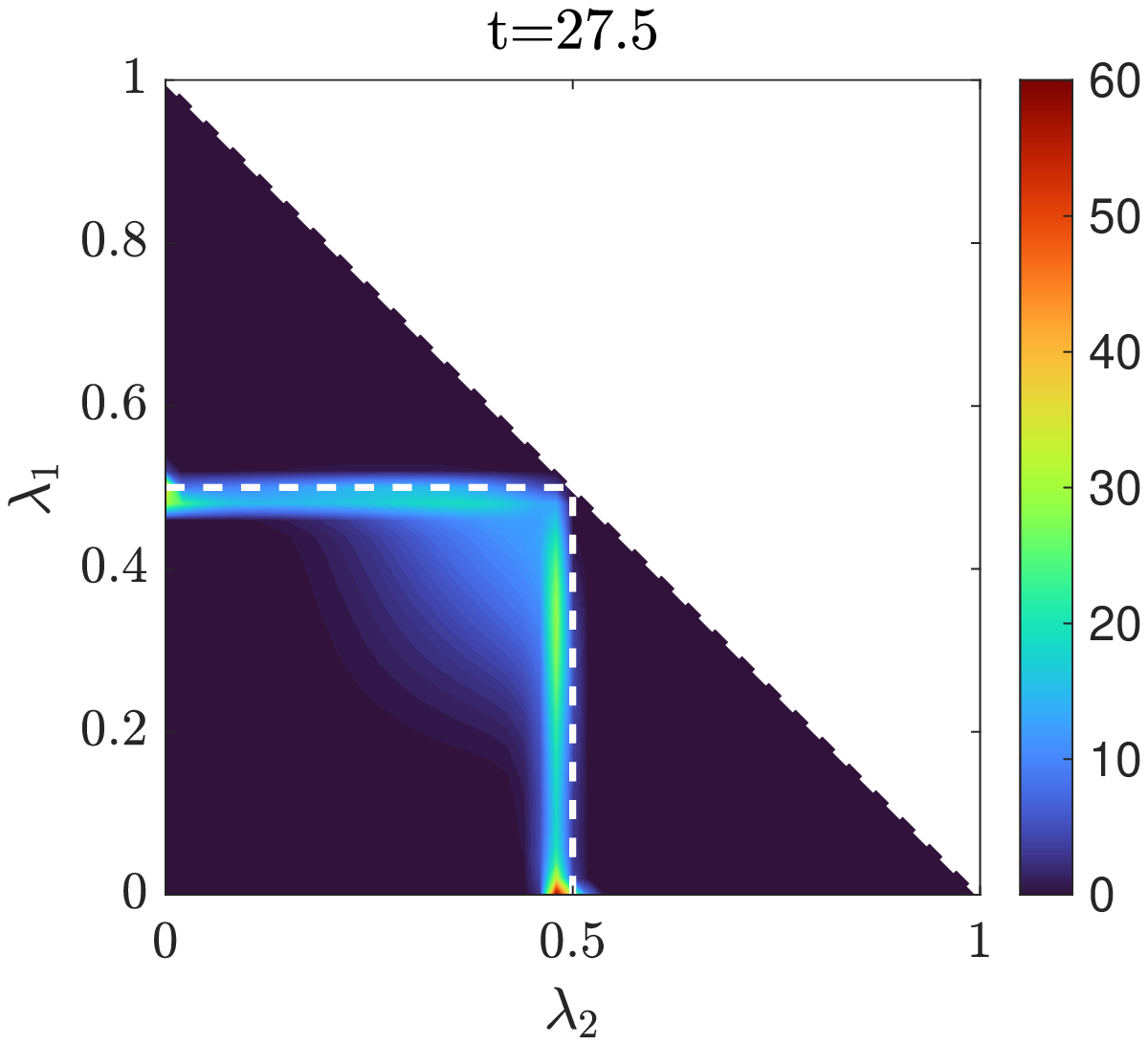}
	\includegraphics[width=0.225\linewidth]{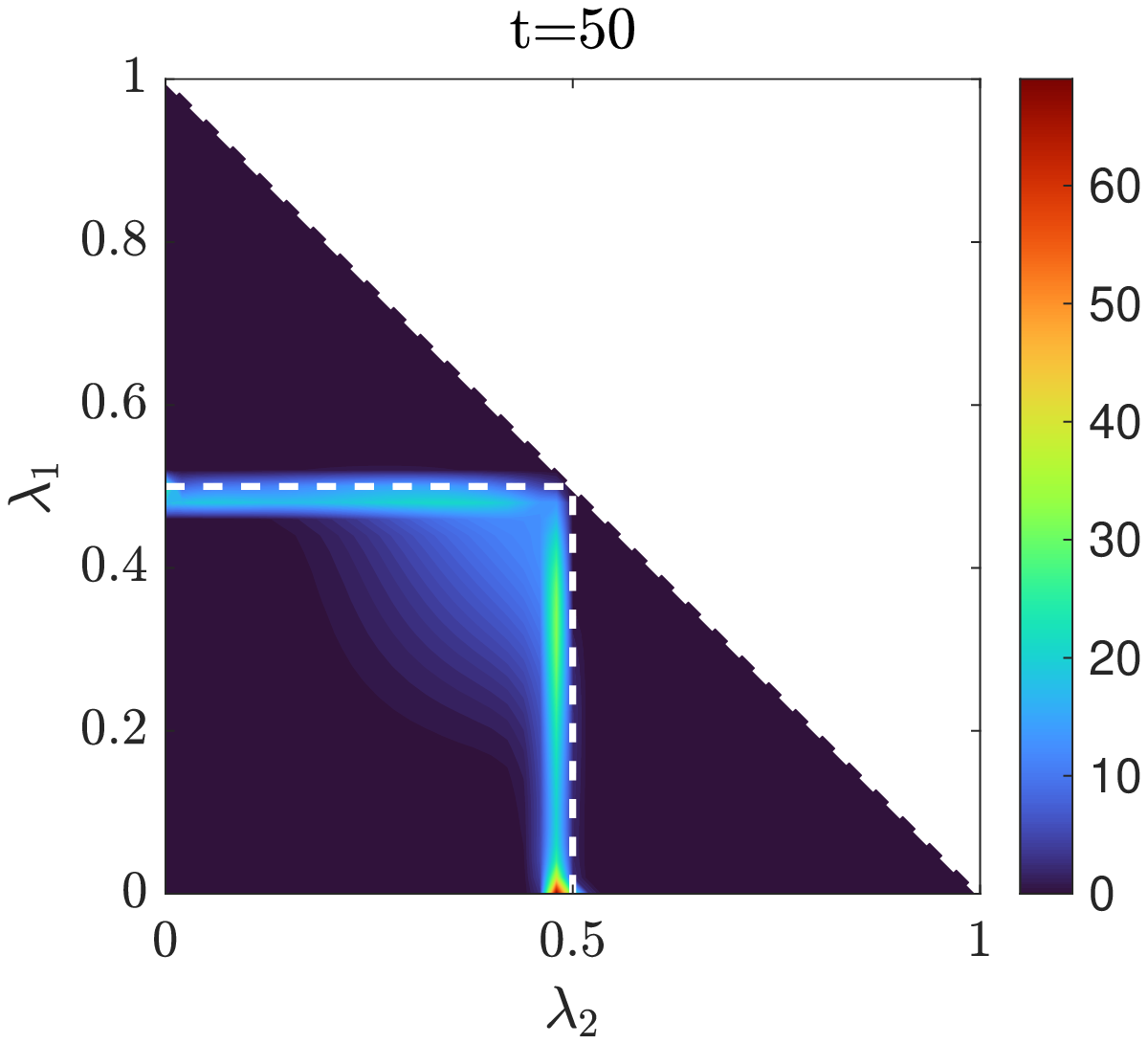}
		\\
		\includegraphics[width=0.225\linewidth]{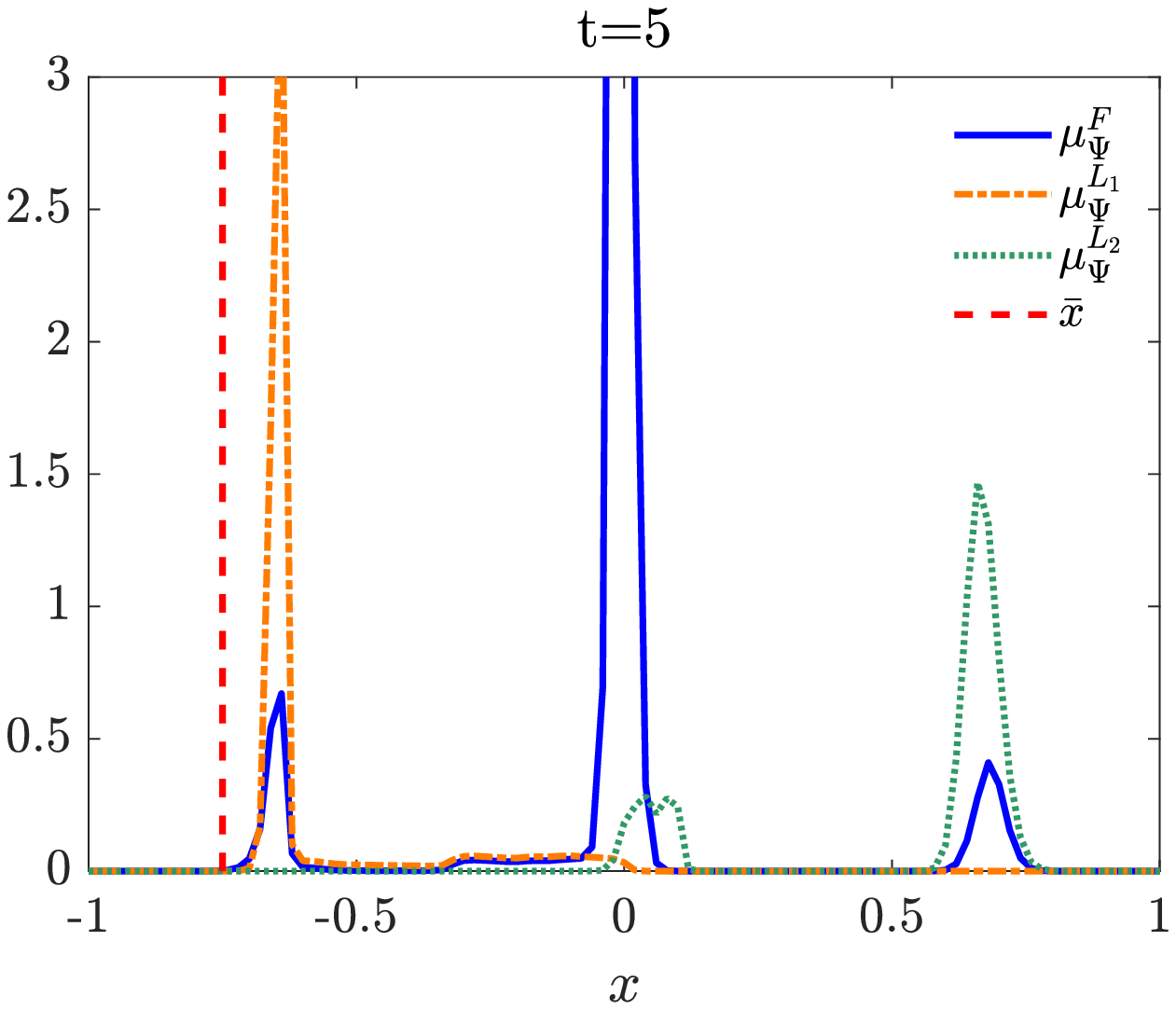}
		\includegraphics[width=0.225\linewidth]{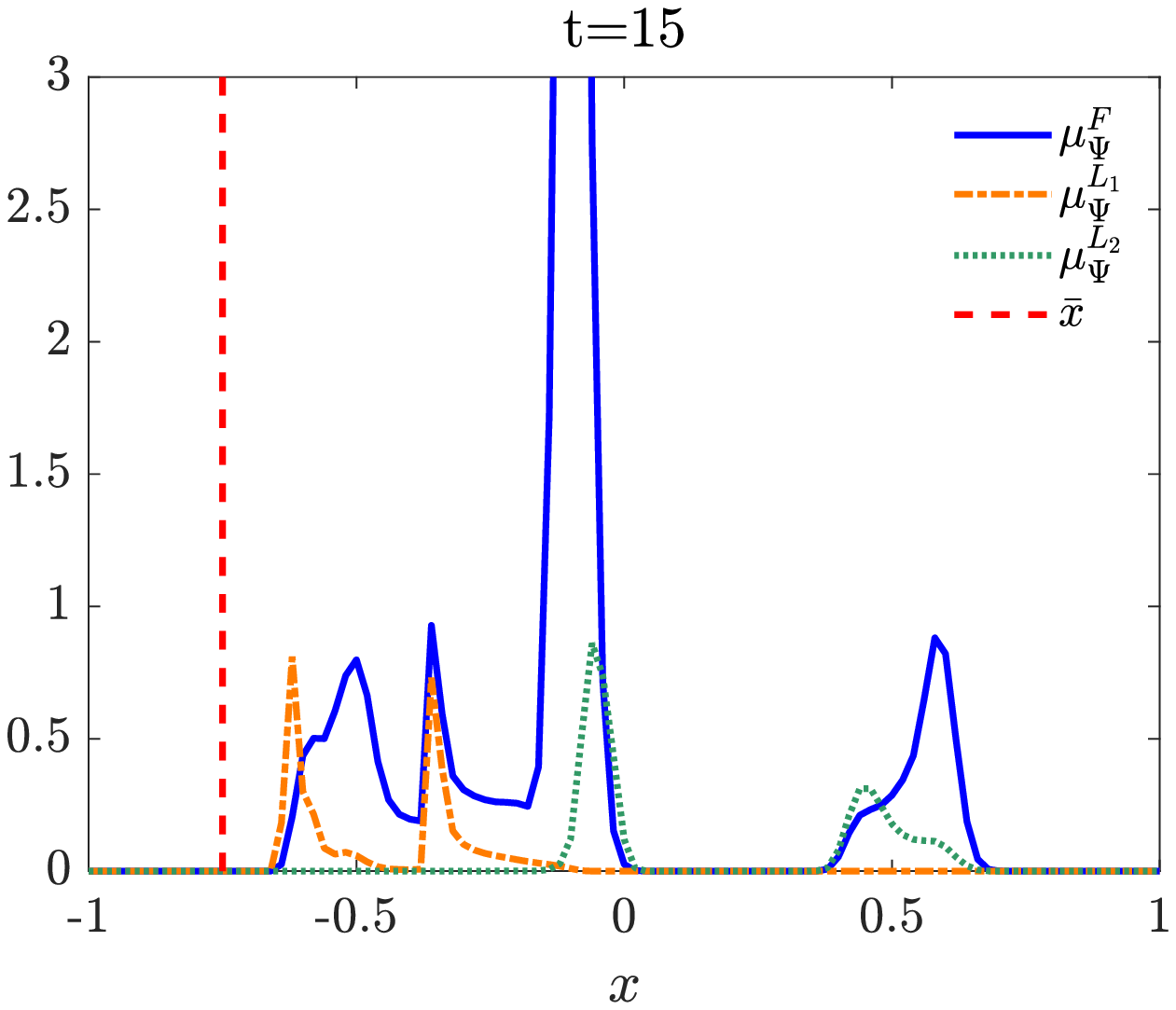}
				\includegraphics[width=0.225\linewidth]{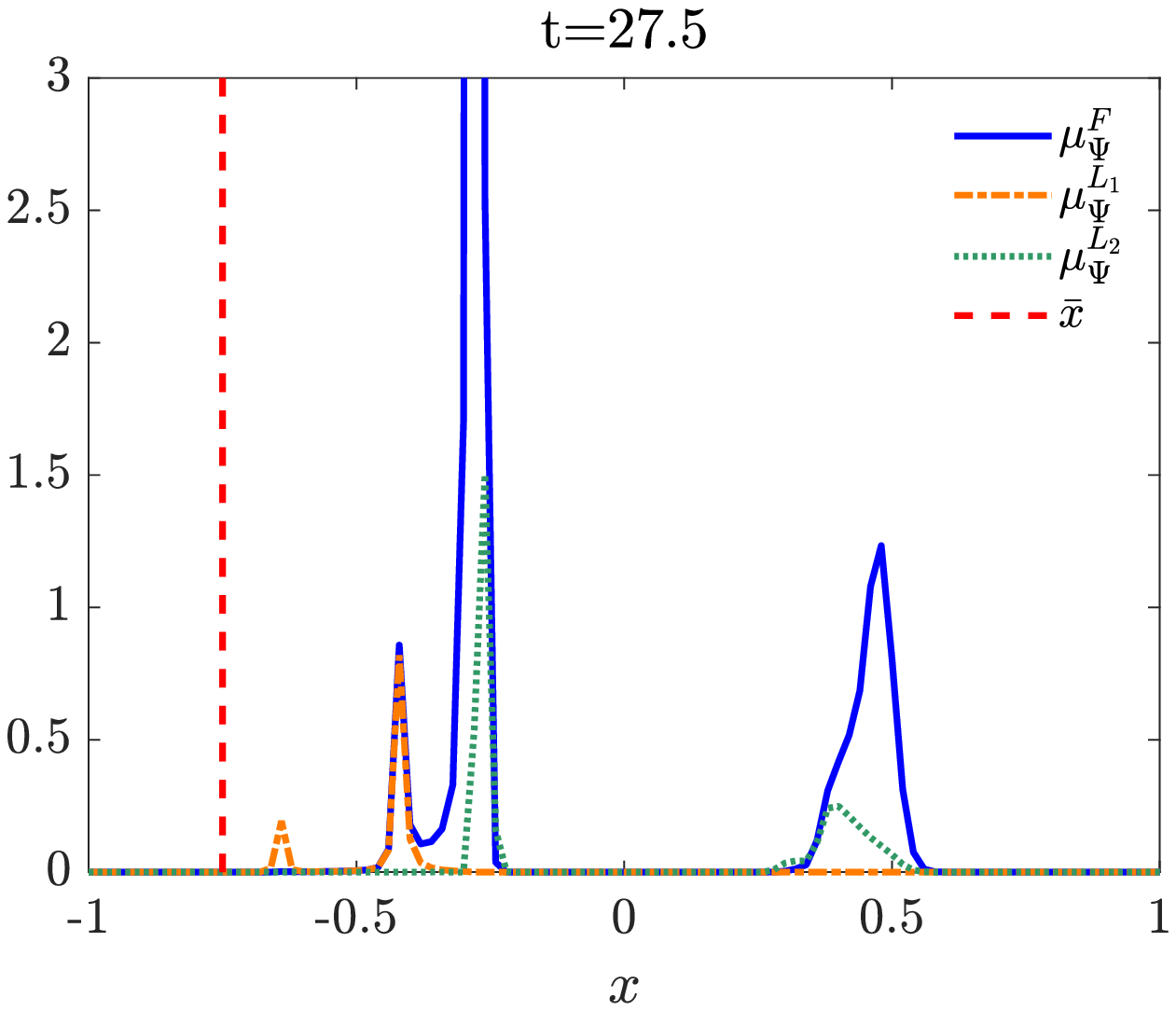}
		\includegraphics[width=0.225\linewidth]{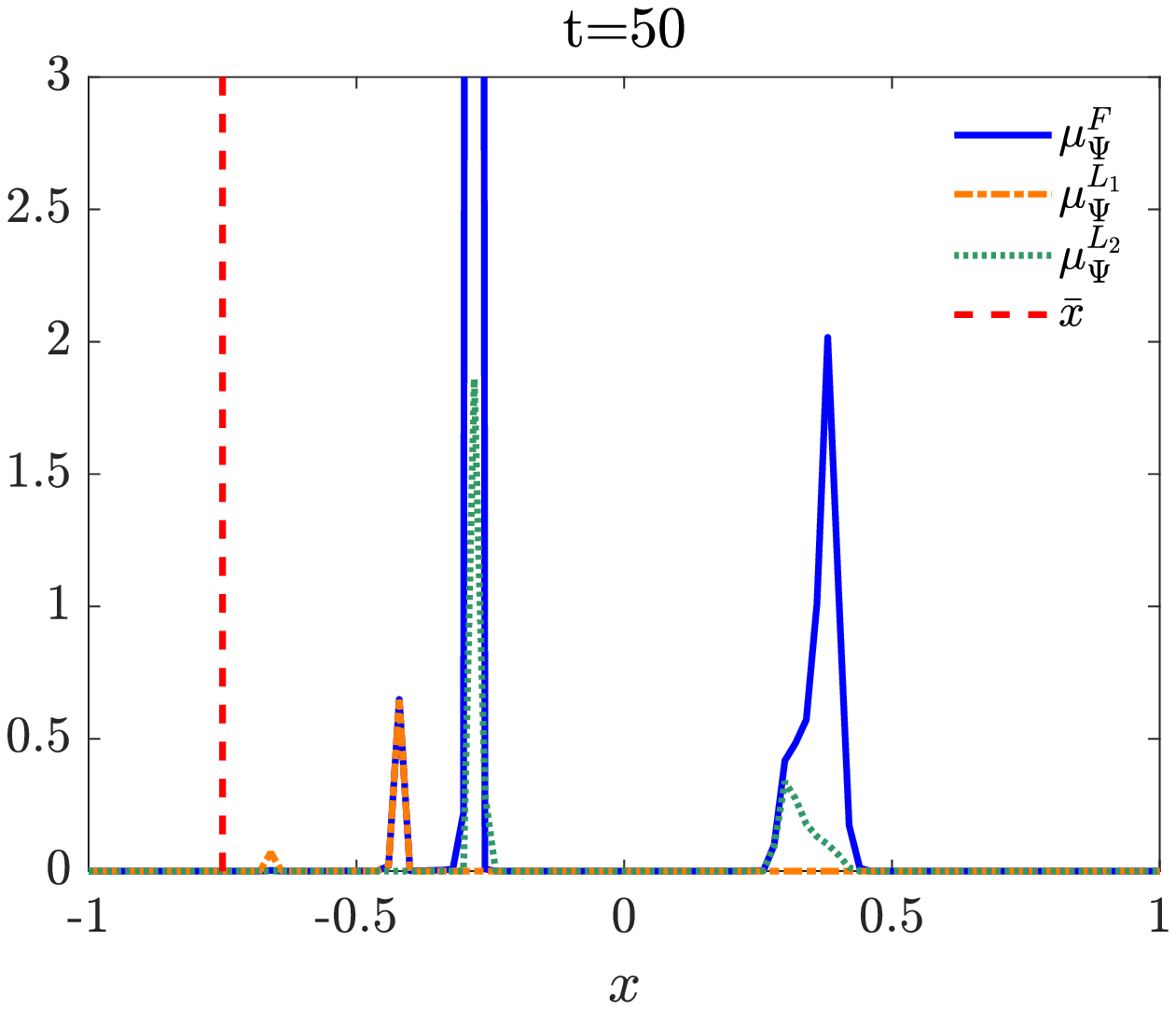}
		\caption{{\em Test 2.}~Evolution of the marginals with control for time frames  $t=\{5,15,27.5, 50\}$. Bottom row depicts $\mu^{F}_\Psi(t,x),$ $\mu^{L_1}_\Psi(t,x)$ and $\mu^{L_2}_\Psi(t,x)$; top row  shows $\nu_\Psi(t,\blambda)$.} \label{fig:evo_ctrl_Test2}
	\end{center}
\end{figure}
\newpage

\begin{figure}[!ht]
	\begin{center}
		\includegraphics[width=0.35\linewidth]{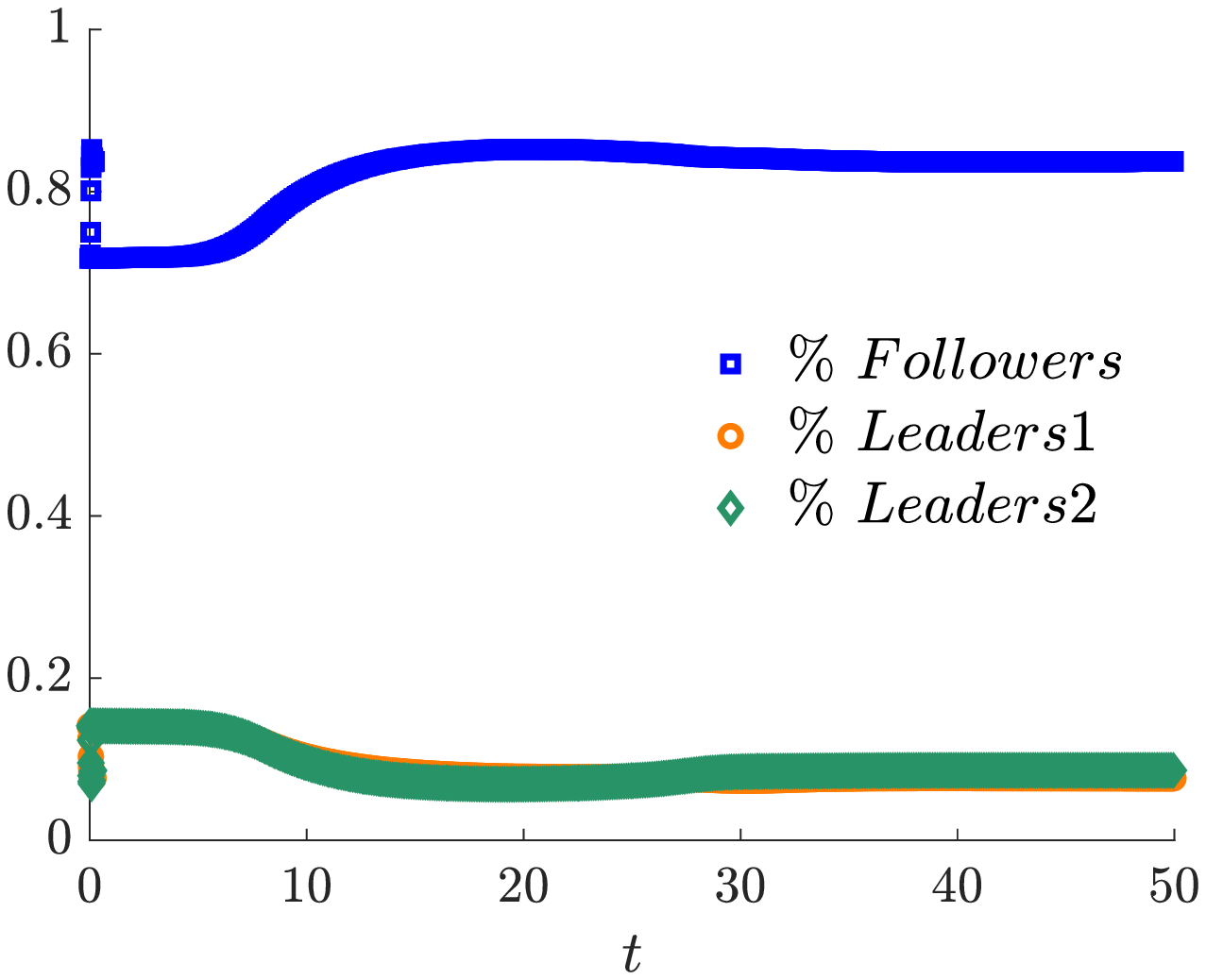}
		\hspace{+0.5cm}
				\includegraphics[width=0.35\linewidth]{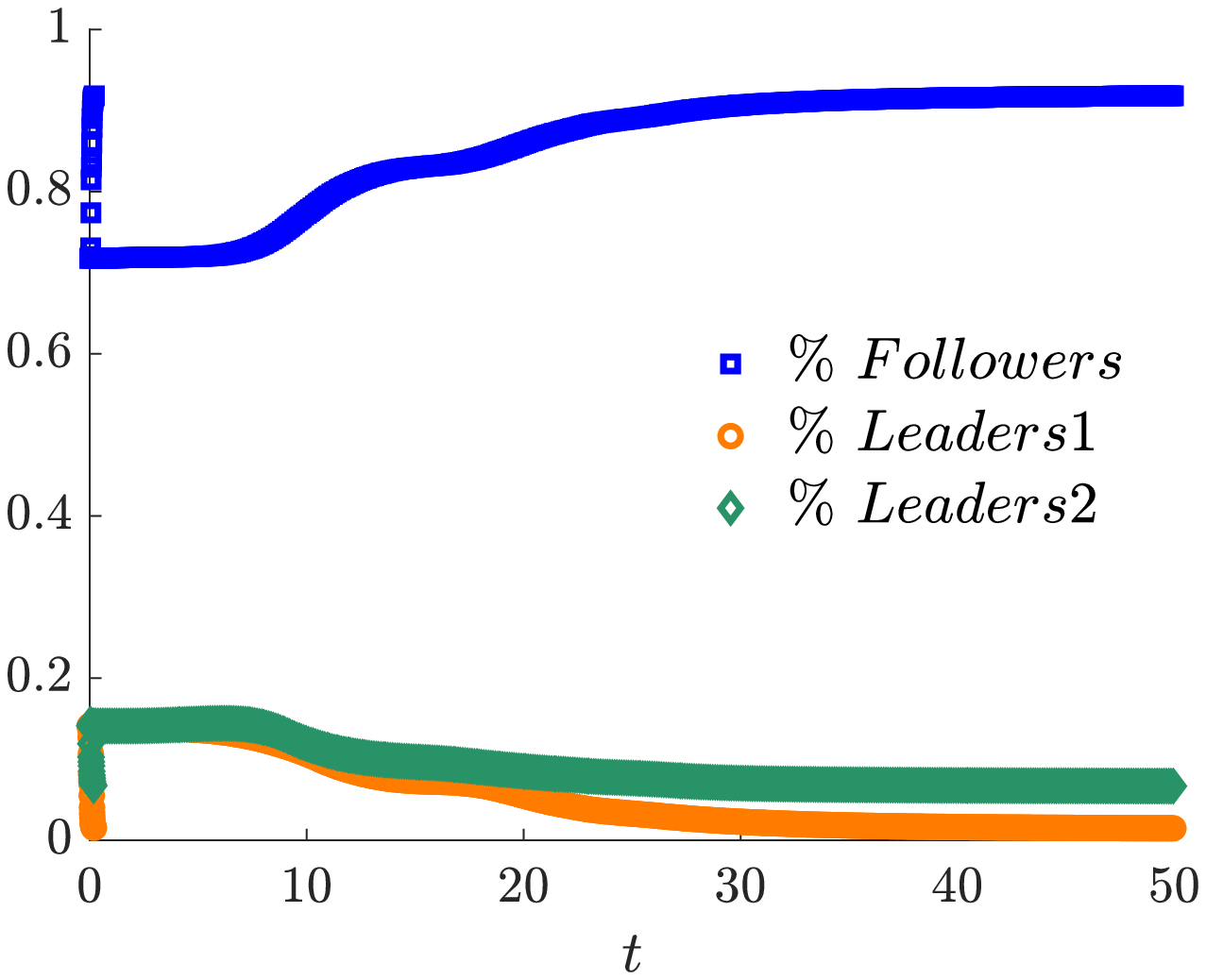}
		\caption{{\em Test 2.}~Percentage of population associated with leader $L_1,L_2$ and follower as functions of time. Left plot: uncontrolled case. Right plot: controlled case.} \label{fig:evo_noctrl}
	\end{center}
\end{figure}

\bigskip

\noindent\textbf{Acknowledgments}
The work of GA was partially supported by the MIUR-PRIN Project 2017, No.~2017KKJP4X \emph{Innovative numerical methods for evolutionary partial differential equations and applications} and by RIBA 2019, No.~RBVR199YFL \emph{Geometric Evolution of Multi Agent Systems}.
The work of SA was supported by the FWF through the projects OeAD-WTZ CZ 01/2021 and I 5149.
The work of MM was partially supported by the \emph{Starting grant per giovani ricercatori} of Politecnico di Torino and by the MIUR grant Dipartimenti di Eccellenza 2018-2022 (E11G18000350001).
The work of FS was supported by the project \emph{Variational methods for stationary and evolution problems with singularities and interfaces} PRIN 2017 financed by the Italian Ministry of Education, University, and Research.
GA is a member of the GNCS of INdAM and MM and FS are members of the GNAMPA group of INdAM.

\bibliographystyle{siam}
\bibliography{AAMS.bib}

\end{document}